\documentclass[12pt]{amsart}

\usepackage{Commands}

\usepackage{tikz, pgfplots}

\allowdisplaybreaks
\numberwithin{equation}{section}

\def\lra{\longrightarrow}
\def\({\big(}\def\){\big)}
\def\lek{\lesssim} \def\gek{\gtrsim}
\def\sbt{\subset} 
\def\g{\gamma} \def\th{\theta}
\def\lt{\left} \def\rt{\right}
\def\sms{\setminus}
\def\es{\emptyset}
\def\Om{\Omega} \def\om{\omega}
\def\Ga{\Gamma}
\def\lam{\lambda}
\def\fr{\noindent}
\def\cM{\mathcal M}
\def\Int{{\rm Int}}
\def\cEP{\cE{\rm P}}
\def\P{{\rm P}}

\def\bu{\bigcup}

\def\lt{\left}
\def\rt{\right}

\makeatletter
\newcommand{\mylabel}[2]{#2\def\@currentlabel{#2}\label{#1}}
\makeatother

\title[]{{\large C}ritically {\large F}inite {\large R}andom {\large M}aps of an {\large I}nterval}
\author[\sc Jason ATNIP]{\sc Jason ATNIP}
\address{ Jason Atnip,
School of Mathematics and Statistics, University of New South Wales,
Sydney, NSW 2052,
Australia}
\email{j.atnip@unsw.edu.au\newline\hspace*{0.3cm} 
Web: http://atnipmath.com}
\author[\sc Mariusz URBA\'NSKI]{\sc Mariusz URBA\'NSKI}
\address{Mariusz Urba\'nski, Department of Mathematics,
 University of North Texas, Denton, TX 76203-1430, USA}
\email{urbanski@unt.edu\newline \hspace*{0.3cm} Web:
http://www.math.unt.edu/$\sim$urbanski}

\begin{document}
\begin{abstract} 
We consider random multimodal $C^3$ maps with negative Schwarzian derivative, defined on a finite union of closed intervals in $[0,1]$, onto the interval $[0,1]$ with the base space $\Om$ and a base invertible ergodic map $\th:\Om\to\Om$ preserving a probability measure $m$ on $\Om$. We denote the corresponding skew product map by $T$ and call it a critically finite random map of an interval. We prove that there exists a subset $AA(T)$ of $[0,1]$ with the following properties:
\begin{enumerate}
\item For each $t\in AA(T)$ a $t$--conformal random measure $\nu_t$ exists. We denote by $\lm_{t,\nu_t,\om}$ the corresponding generalized eigenvalues of the corresponding dual operators $\tr_{t,\om}^*$, $\om\in\Om$. 

\item Given $t\ge 0$ any two $t$--conformal random measures are equivalent.

\item The expected topological pressure of the parameter $t$:
\begin{align*}
		\cEP(t):=\int_{\Om}\log\lm_{t,\nu,\om}dm(\om).
\end{align*} 
is independent of the choice of a $t$--conformal random measure $\nu$. 

\item The function 
$$
AA(T)\ni t\longmapsto \cEP(t)\in\mathbb R
$$
is monotone decreasing and Lipschitz continuous. 

\item With $b_T$ being defined as the supremum of such parameters $t\in AA(T)$ that $\cEP(t)\ge 0$, it holds that 
$$
\cEP(b_T)=0 
\  \  \  {\rm and} \  \  \
[0,b_T]\sbt \Int(AA(T)).
$$
\item $\HD(\cJ_\om(T))=b_T$ for $m$--a.e $\om\in\Om$, where $\cJ_\om(T)$, $\om\in\Om$, form the random closed set generated by the skew product map $T$.

\item $b_T=1$ if and only if $\union_{\Dl\in \cG}\Dl=[0,1]$,
and then $\cJ_\om(T)=[0,1]$ for all $\om\in\Om$. 
\end{enumerate}
\end{abstract}

\maketitle
\tableofcontents

\section{Introduction}

This paper has three primary sources of motivation: random dynamical systems, smooth multimodal maps of an interval with negative Schwarzian derivative and conformal dynamical systems. The general idea behind random dynamics is that one considers a generalized skew product map. More precisely, 
\begin{itemize}
\item a probability space $\Om$ is given with a probability measure $m$, 

\item an ergodic invertible measurable map $\theta:\Om\to\Om$ preserving measure $m$ is given,

\item for every $\om\in\Om$ a closed subset $\cJ_\om$ of a complete metrizable space $X$ is given,

\item for every $\om\in\Om$ a continuous map $T_\om:\cJ_\om\to \cJ_{\th(\om)}$ is given. Moreover the map $\Om\ni \om\mapsto \cJ_\om$ is a closed random set, i.e. it satisfies a certain measurability condition which we will describe later. 
\end{itemize}
Then the induced map
\begin{align*}
T:	\cJ:=\union_{\om\in\Om}\set{\om}\times\cJ_\om \to\cJ,
\end{align*}
$$
T(\om,x):=\(\th(\om),T_\om(x)\)
$$
is called the corresponding generalized skew product map or a random dynamical system. In order to see randomness more clearly, consider the maps
$$
T_\om^n:=T_{\th^{n-1}(\om)}\circ T_{\th^{n-2}(\om)}\circ\ldots T_{\th(\om)}\circ T_\om:\cJ_\om\lra \cJ_{\th^n(\om)}.
$$
One can view this composition scheme as iterating randomly the maps $T_\om$
according to the random process governed by the measure preserving ergodic map $\th:\Om\to\Om$. This is ``particularly random'' when $\th$ is a Bernoulli  shift, which essentially means that the maps $T_\om$ form an independent and identically distributed sequence of random variables. The general task of random dynamical systems is to search for probability measures $\mu$ on $\cJ$ whose marginal is equal to $m$, i.e. 
$$
\nu\circ\pi^{-1}_\Om=m,
$$
where 
$$
\pi_\Om:\Om\times X\lra\Om
$$ 
is the projection onto the first coordinate. One typically looks for such measures with some prescribed dynamical, stochastic, or geometric properties, and then one studies them. 

More on the abstract setting of random dynamical systems can be found in \cite{Arnold}, \cite{kifer1}, \cite{kifer2}, \cite{kifer+1survey}. We also recommend the little, well written, book of Hans Crauel \cite{Crauel}. Additional papers on random dynamics, somewhat related to our current work include \cite{Bogenschutz1}--\cite{Bogenschutz6}, \cite{Crauel-Flandoli}, \cite{Khanin-Kifer}, \cite{kifer1}--\cite{kifer5}, \cite{kifer+1survey}, \cite{MSU},  \cite{Rugh1} and \cite{Rugh2}. This list is not by any means complete. Our approach to randomness stems from \cite{MSU}.  

Our second motivation, deterministic iteration of smooth multimodal maps of an interval with negative Schwarzian derivative, also has a long history and forms a well established fully developed sophisticated theory. It all began with the seminal paper \cite{Misiurewicz} of Michal Misiurewicz in 1981 and has been rapidly developing ever since. A good, reader friendly introduction to this theory can be found in the classical book \cite{CE} by Pierre Collett and Paul Eckmann. A full systematic exposition is available in the book \cite{MS} by Wellington de Melo and Sebastian van Strien. Given a $C^3$ map $f$, the Schwarzian derivative $S(f)$ of $f$ is given by the formula
\begin{align*}
S(f)=\left(\frac{f''}{f'}\right)'-\frac{1}{2}\left(\frac{f''}{f'}\right)^2=\frac{f'''}{f'}-\frac{3}{2}\left(\frac{f''}{f'}\right)^2.
\end{align*}
The key feature of negative Schwarzian derivative $S(f)$ is that it entails a bounded distortion property almost as strong as the Koebe's Distortion Theorem for univalent holomorphic maps on the complex plane $\mathbb C$. This is the most important property due to which the above mentioned theory of deterministic maps of an interval works so well, and due to which our treatment of random maps with negative Schwarzian derivative was possible. We describe it now. 

The probability space $\Om$ endowed with the $\theta$-invariant probability measure $m$ is  as above. Given a finite collection $\cG$ of closed subintervals of the unit interval $[0,1]$ with mutually empty interiors, the set $\cM$, with full notation of Section~\ref{SRIMM} denoted by $\cM(\cG;\kappa, A,\g)$, is a collection of $C^3$ maps with negative Schwarzian derivatives, from the intervals in $\cG$ onto $[0,1]$ satisfying some mild natural uniformity conditions. We iterate these maps randomly. More precisely, 
$$
\Om\ni\om\longmapsto T_\om\in\cM(\cG;\kappa, A,\g)
$$ 
is a measurable map in the sense of Crauel \cite{Crauel}. We cannot iterate these maps yet as described above. This would be actually possible if we demand that the union of intervals forming $\cG$ is equal to the whole interval $[0,1]$. But we do not assume this. We allow this but are far from assuming this. Then already the first iterate $T_\om(x)$, $x\in[0,1]$, may not be well defined. We therefore define the random closed sets $\cJ_\om(T)$ as the sets of all points $x\in[0,1]$ for which all iterates $T_\om^n(x)$, $n\ge 0$, are well defined, i.e. they belong to elements of $\cG$. Then we have a good map 
\begin{align*}
T:	\cJ(T):=\union_{\om\in\Om}\set{\om}\times\cJ_\om(T) \lra\cJ(T). 
\end{align*}
We call this map a random critically finite map of the interval $[0,1]$. Our primary goals in this paper are twofold. To investigate the existence, uniqueness, and properties of random conformal measures for $T$ and random $T$--invariant measures (both with marginal $m$) absolutely continuous/equivalent with respect to conformal measures. 

A random measure $\nu=\nu_t$ on $\cJ(T)$ is called $t$--\textit{conformal} if there exists a measurable function $\lm:\Omega\to(0,\infty)$ such that
\begin{align}\label{eqn: conformal measure}
\nu_{\ta(\om)}(T_\om(A))=\lm_\om\int_A\absval{T'_\om}^td\nu_\om
\end{align}
for every special set $A\sub \cJ_\om(T)$ and $m$-a.e. $\om\in\Omega$, and also 
$$
\nu_\om\(T_{\om,\Dl}^{-1}(A)\)
=\lm_{\om}^{-1}\int_A\absval{\(T_{\om,\Dl}^{-1}\)'}^t d\nu_{\th(\om)}
$$
for every $\om\in\Om$, every $\Dl\in\cG$ and every Borel set $A\sbt [0,1]$. We call the former property throughout the paper, $t$--Fconformal while the latter $t$--Bconformal. Quite frequently, in random (and deterministic as well) dynamical systems both $t$--Fconformality and $t$--Bconformality are equivalent. But here, because of the existence of critical points of the maps $T_\om$, the situation is much more subtle. We fully settle (see Corollary~\ref{c520180405}) the relation between Fconformality and Bconformality of critically finite maps in Section~\ref{FBCM} entitled ``$t$--Fconformal, $t$--Bconformal, and $t$-conformal Measures''. The existence of $t$--conformal measures is one of the major issues in our paper, one which however, trivially disappears if the union of the elements of $\cG$ is the entire interval $[0,1]$ and one is merely interested in $1$--conformal measures. Lebesgue measure on $[0,1]$ makes it then. We do deal with non-trivial conformal measures. In order to present our corresponding result meaningfully, we need some further definitions; they will be needed to formulate some other major theorems as well. We introduce some technical condition (see Definition~\ref{d1tp3}) and denote by $AA(T)$ the set of all parameters $t\ge 0$ that satisfy it. We prove (see Theorem~\ref{t1tp5}) that 

\medskip \centerline{For each $t\in AA(T)$ a $t$--conformal random measure $\nu_t$ exists.}

\medskip and

\medskip \centerline{Given $t\ge 0$ any two $t$--conformal random measures are equivalent.}

\medskip 

But this is not the end of this issue. We investigate the structure of the set $AA(T)$ further. Indeed, for every $t\in AA(T)$, we introduce the expected topological pressure of the parameter $t$:
\begin{align*}
		\cEP(t):=\int_{\Om}\log\lm_{t,\nu,\om}dm(\om).
	\end{align*} 
We show that this number is independent of the choice of $t$--conformal measure $\nu$. We prove that that the function 
$$
AA(T)\ni t\longmapsto \cEP(t)\in\mathbb R
$$
is monotone decreasing and Lipschitz continuous. We define $b_T$ as the supremum of such parameters $t\in AA(T)$ that $\cEP(t)\ge 0$. We show that
$$
\cEP(b_T)=0 
\  \  \  {\rm and} \  \  \
[0,b_T]\sbt \Int(AA(T)).
$$

The construction of $t$--conformal measures for $t\in AA(T)$ is technical and involved. First we truncate the original system $T$ by defining some sequence $\cJ^{(k)}(T)$, $k\ge 1$, of $T$--invariant closed sets whose points, under iterates of $T$, omit some appropriate neighborhoods of zero. These neighborhoods are constructed in such a way that the maps $T|_{\cJ^{(k)}(T)}:\cJ^{(k)}(T)\to \cJ^{(k)}(T)$, $k\ge 1$, satisfy all the requirements of the distance expanding random maps considered in \cite{MSU}. In particular, we get the corresponding $t$-conformal measures $\nu_k^{(t)}$, $k\ge 1$, for the system $T|_{\cJ^{(k)}(T)}$. We finally show that any weak limit in the narrow topology, in sense of \cite{Crauel}, of the $\nu_k^{(t)}$, $k\ge 1$, is a $t$--conformal measure for the map $T:\cJ\to\cJ$. 

\medskip Now, we pass to $T$-invariant measures absolutely continuous with respect to conformal measures. We prove the following. 

\begin{theorem}
For every admissible parameter $t\geq 0$, i.e. belonging to AA(T), there exists a unique $T$--invariant random measure on $\cJ(T)$ absolutely continuous with respect to $\nu_t$. In addition $\mu_t$ is equivalent to $\nu_t$ and ergodic.
\end{theorem}

We construct the measure $\nu_t$ globally as a weak limit, in the narrow topology, of ergodic averages of the measures $\nu_t\circ T^{-n}$, $n\ge 1$. 
A technical reasoning then shows that $\mu_t$ is equivalent ot $\nu_t$, ergodic and unique. The last property, in particular, is obtained by making use of the above mentioned fact, that any two $t$--conformal measures are equivalent. 

\, Our third main theme in this paper is the question of what is the value of Hausdorff dimension of the fiberwise sets $\cJ_\om(T)$, $\om\in\Om$. It follows immediately from ergodicity of the map $\theta:\Om\to\Om$ that the function $\Om\ni\om\mapsto\HD(\cJ_\om(T))\in[0,1]$ is $m$--a.e constant. But what is its value? Our answer is a version of Bowen's formula tracing back to and primarily motivated by the seminal pioneering work of Rufus Bowen (\cite{bowen}) on the Hausdorff dimension of the limit sets of quasi--Fuchsian groups where he employed for the first time the machinery of thermodynamic formalism to determine the Hausdorff dimension. Indeed, we prove that 
$$
\HD(\cJ_\om(T))=b_T
$$ 
for $m$--a.e. $\om\in\Om$, i.e. $\HD(\cJ_\om(T))$ is equal to the first zero of the expected topological pressure function $\cEP(t)$, $t\ge 0$. Our proof primarily stems from the work \cite{MSU} but is technically considerably more involved due to the existence of critical points and lack of hyperbolicity. 

As the last result of this paper we prove the following theorem which shows that the sets $\cJ_\om(T)$, $\om\in\Om$, are all, up to a set of $m$--measure zero, true fractals unless
\begin{align*}
	I_*=\union_{\Dl\in\cG}\Dl=[0,1],
\end{align*}
in which case each set $\cJ_\om(T)$ is equal to $[0,1]$. 

\begin{theorem}
	If $T:\cJ(T)\to\cJ(T)$ is a random critically finite map, then 
\begin{align*}
		b_T=1 \  \text{ if an only if }  \  \union_{\Dl\in \cG}\Dl=[0,1],
	\end{align*}
and then $\cJ_\om(T)=[0,1]$ for all $\om\in\Om$. 
\end{theorem}

\section{Preliminaries on Random Measures}\label{sec: prelim on rand meas}

We consider a Polish space (complete separable metrizable space) $X$, with Borel $\sg$-algebra $\sB$, and a complete probability space $(\Om,\sF,m)$.

\, Let $\sB\otimes\sF$ be the product $\sg$-algebra of $\sF$ and $\sB$. A probability measure $\nu$ on $\Om\times X$ with respect to the product $\sg$-algebra $\sB\otimes\sF$ is said to be \textit{random probability measure} relative to $m$ if it has marginal $m$, i.e. if
$$
\nu\circ\pi^{-1}_\Om=m,
$$
where 
$$
\pi_\Om:\Om\times X\lra\Om
$$ 
is the projection onto the first coordinate, i.e. defined by $\pi_\Om(\om,x)=\om$. If $(\nu_\om)_{\om\in\Om}$ are disintegrations of $\nu$ with respect to the partition $\(\{\om\}\times X\)_{\om\in\Om}$, then these satisfy the following properties:
\begin{enumerate}
	\item For every $B\in\sB$, the map $\Om\ni\om\longmapsto\nu_\om(B)\in X$ is measurable, 
	\item For $m$-a.e. $\om\in\Om$, the map $\sB\ni B\longmapsto\nu_\om(B)\in X$ is a Borel probability measure.
\end{enumerate}

\noindent Equivalently, to come up with a random probability measure relative to $m$, one can start with a family $(\nu_\om)_{\om\in\Om}$ of probability measures on $(X,\sB)$ satisfying conditions (1) and (2) above, and then define the measure $\nu$ on $(\Om\times X,\sB\otimes\sF)$ by the formula
$$
\nu(A):=\int_\Om\nu_\om(A_\om)\,dm,
$$
where the sets $A_\om$, $\om\in\Om$, are uniquely determined by the condition that
$$
\{\om\}\times A_\om=A\cap (\{\om\}\times X).
$$
Equivalently
$$
A_\om=\pi_X(A\cap(\{\om\}\times X)),
$$
where 
$$
\pi_X:\Om\times X\lra X
$$ 
is the projection onto the second coordinate, i.e. defined by $\pi_X(\om,x)=x$. The space of all random probability measures relative to $m$ will be denoted in the sequel by 
$$
M^1_m(X).
$$

Denote by $2^X$ the family of all subsets of $X$. Let $\rho$ be any complete metric on $X$ compatible with its topology. Following \cite{Crauel} we say that a function $F:\Om\to 2^X$ is a closed random set if and only if the set $F(\om)\sbt X$ is a closed subset of $X$ for all $\om\in\Om$ and 
$$
\Om\ni\om\longmapsto \rho(x,F(\om)):=\inf\{\rho(x,y):y\in F(\om)\}\in[0,+\infty)
$$ 
is measurable for every $x\in X$. We frequently identify a closed random set with its graph
$$
{\rm graph}(F):=\bu_{\om\in\Om}\{\om\}\times F(\om). 
$$
Of course 
$$
{\rm graph}(F)_\om=F(\om)
$$
and with our identification, we have
$$
F_\om=F(\om).
$$
We say that a measurable set $U\sub\Om\times X$ is an open random set if its complement is a closed random set. We again identify an open random set with its graph. Since our probability measure $m$ on $\Om$ is complete,
Proposition 2.4 of \cite{Crauel} implies the following.

\begin{prop}\label{p120180619}
A function $F:\Om\to 2^X$ is a closed random set if and only if the set $F(\om)\sbt X$ is a closed subset of $X$ for all $\om\in\Om$ and graph$(F)$ is a measurable subset of $\Om\times X$. 
\end{prop}

\noindent This proposition in turn directly implies the following.

\begin{prop}\label{p220180619}
A function $F:\Om\to 2^X$ is an open random set if and only if the set $F(\om)\sbt X$ is an open subset of $X$ for all $\om\in\Om$ and graph$(F)$ is a measurable subset of $\Om\times X$. 
\end{prop}

\noindent As a direct consequence of these two propositions we get the following.

\begin{prop}\label{p220180619b}
Random closed and open sets behave naturally under set theoretical operations. More precisely.
\begin{enumerate}
\item Any countable intersection of closed random sets is a closed random sets. 

\, \item Any countable union of open random sets is an open random sets. 

\, \item Any finite union of of closed random sets is a closed random sets.

\, \item Any finite intersection of open random sets is an open random sets. 
\end{enumerate}
\end{prop}
 
 Having all of that, the standard proof (see for example, the proof of Lemma 1.5.7 in \cite{cohn_measure_2013}) yields inner and outer regularity of any random measure $\nu$. Precisely, we have the following.
\begin{theorem}\label{t1mis119}
Every random measure $\nu\in M^1_m(X)$, where $X$ is a Polish space, is inner and outer regular, meaning that if $A$ is a measurable set in $\Om\times X$ then 
	\begin{align*}
	\nu(A)&=\sup\set{\nu(F): F\sub A\text{ is a closed random set}}\\
	&=\inf\set{\nu(G): G\bus A\text{ is an open random set}}.
	\end{align*}
\end{theorem}

\, 

Consider a closed random  set 
$$
\Om\ni\om\longmapsto\cJ_\om\sub X
$$ 
and let
\begin{align*}
	\cJ:=\union_{\om\in\Om}\set{\om}\times\cJ_\om.
\end{align*}

The random measure $\nu$ is said to be supported on $\cJ$ if 
$$
\nu(\cJ)=1.
$$
Equivalently:
$$
\nu_\om(\cJ_\om)=1
$$
for $m$--a.e. $\om\in\Om$.

Denote by $M_m^1(\cJ)$ the set of all random probability measures relative to $m$ supported on $\cJ$.

Following Crauel (\cite{Crauel}), we say that a function $g:\Om\times X\lra \RR$ is said to be random continuous if the function
$$
g_\om:X\lra\RR,
$$ 
given by the formula
$$
g_\om(z)=g(\om,z),
$$
is continuous for $m$-a.e. $\om\in\Om$, and for each $x\in X$, the function 
$$
g_x:\Om\lra\RR
$$ 
is measurable, where 
$$
g_x(\om)=g(\om,x).
$$
A function $g:\cJ\to \RR$ is said to be random continuous if it has an extension to a random continuous from $\Om\times X$ to $\RR$.

Then, according to \cite{Crauel}, the function $g:\cJ\to \RR$ is measurable. Still following \cite{Crauel}, we put 
\begin{align*}
\norm{g}_\infty:=\esssup\(\norm{g_\om}_\infty:\om\in\Om\)
\end{align*}
and denote by 
$$
C_b(\cJ)
$$
the space of all random continuous functions $g$ from $\cJ$ to $\RR$ for which
$$
\norm{g}_\infty<+\infty.
$$
Obviously $\norm{\cdot}_\infty$ is a norm on the vector space $C_b(\cJ)$ 
and makes this space a Banach space. 

We denote by $C_b^*(\cJ)$ the set of all elements $s$ in 
$$
\bigcup_{\om\in\Om}\{\om\}\times C_b^*(\cJ_\om).
$$
such that for every $g\in C_b(\cJ)$, the map
$$
\Om\ni\om\longmapsto s_\om(g_\om)\in\RR
$$
is measurable and 
\begin{align*}
	\norm{s}_\infty:=\esssup(\norm{s_\om}_\infty:\om\in\Om)
\end{align*}
is finite. Obviously $\norm{\cdot}_\infty$ is a norm on the vector space $C_b^*(\cJ)$ and makes this space into a Banach space. 

\section{Conformal Measures I}
We start with a very general setting. Suppose that $(X,\cA,\mu)$ and $(Y,\cB,\nu)$ are probability spaces. Suppose that $T:X\to Y$ is measurable with respect to the respective $\sg$--algebras $\cA$ and $\cB$. We say that $T$ is \textit{quasi--invariant} with respect to the pair of measures $(\mu,\nu)$ if the measure $\mu\circ T^{-1}$ is absolutely continuous with respect to $\nu$. Then 
\begin{align*}
(g\mu)\circ T^{-1}<<\nu
\end{align*}
for every non--negative function $g\in L^1(\mu)$. Let 
\begin{align}
\tr_{\mu,\nu}g:=\frac{d((g\mu)\circ T^{-1})}{d\nu}:Y\to [0,\infty] \label{1mis.cmi1}
\end{align}
be the Radon--Nikodym derivative of $g\mu\circ T^{-1}$ with respect to $\nu$. $\tr_{\mu,\nu}$ then extends to a bounded (with norm 1) linear operator from $L^1(\mu)$ to $L^1(\nu)$ by the formula 
\begin{align*}
\tr_{\mu,\nu} g=\tr_{\mu,\nu} g_+-\tr_{\mu,\nu} g_-, 		
\end{align*}  
where $g=g_+-g_-$ is the canonical decomposition of $g$ into its positive and negative parts. By the very definition \eqref{1mis.cmi1} and linearity, we have that
\begin{align}
\int_Y \tr_{\mu,\nu} g d\nu
= g\mu ( T^{-1}(Y))
=g\mu(X)
=\int_X gd\mu. \label{2mis.cmi1}
\end{align}
It also immediately follows from \eqref{1mis.cmi1} that if $g=f\circ T$, where $f:Y\to \RR$ belongs to $L^1(\nu)$, then 
\begin{align*}
\tr_{\mu,\nu}\(h\cdot (f\circ T)\)=f\cdot\tr_{\mu,\nu}(h)
\end{align*}
for every $h\in L^1(\mu)$. In particular if $f=\ind_F$, where $F\in\cB$, then 
\begin{equation}\label{1mis.cmi2}
\begin{aligned}
\mu(T^{-1}(F))
&=\int_X\ind_{T^{-1}(F)}d\mu
=\int_X\ind_F\circ Td\mu \\
&=\int_Y\tr_{\mu,\nu}(\ind_F\circ T)d\nu=\int_Y\ind_F\tr_{\mu,\nu}(\ind)d\nu \\
&=\int_F\tr_{\mu,\nu}(\ind)d\nu. 
\end{aligned} 
\end{equation}
We say that a measurable map $T:X\to Y$ is of \textit{standard type} if there exists a countable partition $(X_k)_{k=0}^\infty$ of $X$ such that 
\begin{enumerate}[(a)]
	\item Each set $X_k$ is measurable,
	\item Each set $T(X_k)$ is measurable, 
	\item For all $k\geq 0$ the map $T\rvert_{X_k}:X_k\to Y$ is 1-to-1. 	
\end{enumerate}
Then for every $x\in X$, we define $\hat J_T(x)\in[0,\infty]$ to be the \textit{reciprocal of the Jacobian} of the map 
$$
(T\rvert_{X_k})^{-1}:T(X_k)\to X_k
$$ 
with respect to the measures $\nu$ and $\mu$ evaluated at the point $T(x)$, where $k\geq 0$ is the unique integer such that $x\in X_k$. Of course, even this is not openly indicated in the symbol $\hat J_T(x)$, it does depend on both measures $\mu$ and $\nu$. We then have
\begin{align}
\tr_{\mu,\nu}g(y)=\sum_{x\in T^{-1}(y)}\frac{g(x)}{\hat J_T(x)}. \label{2mis.cmi2}
\end{align}
We assume the standard convention that
$$
1/0=+\infty \  {\rm and} \ 1/(+\infty)=0.
$$
We shall show that 
\begin{equation}\label{120180302}
\nu\(T\(\hat J_T^{-1}(0)\)\)=0.
\end{equation}
Indeed, if $\nu\(T\(\hat J_T^{-1}(0)\)\)>0$, then there exists at least one $k\ge 0$ such that 
$$
\nu\(T\(X_k\cap\hat J_T^{-1}(0)\)\)>0. 
$$
Hence,
$$
1\ge \mu\(X_k\cap\hat J_T^{-1}(0)\)
=\int_{T\left(X_k\cap\hat J_T^{-1}(0)\right)} \frac1{\hat J_T\circ(T\rvert_{X_k})^{-1}}\,d\nu
=\int_{T\left(X_k\cap\hat J_T^{-1}(0)\right)}+\infty
=+\infty.
$$
This contradiction finishes the proof of formula \eqref{120180302}. 

With the above setting, we call the map $T:X\to Y$ \textit{$\hat J_T$--Bconformal} with respect to the pair of measures $(\mu,\nu)$; the letter ``B'' comes from {``backward''. In other words, having a measurable function $\psi:X\to[0,+\infty]$, we say that the map $T:X\to Y$ is \textit{$\psi$--Bconformal} with respect to the pair of measures $(\mu,\nu)$ if it is quasi--invariant with respect to this pair and 
$$
\hat J_T=\psi.
$$

\medskip We say that a measurable set $A\sub X$ is \textit{special} if $T\rvert_A$ is 1--to--1. \newline
Relaxing quasi--invariance and given a measurable function $\psi:X\to[0,+\infty]$, we say that the map $T:X\to Y$ is \textit{$\psi$--Fconformal} with respect to the pair of measures $(\mu,\nu)$ (letter ``F'' coming from ``forward'') if for every special subset $A$ of $X$ we have that $T(A)\in\mathcal B$ and 
\begin{align}
\nu(T(A))=\int_A \psi \, d\mu \label{3mis.cmi2}
\end{align}
Formula \eqref{3mis.cmi2} means that $\psi$ is the Jacobian of $T$ with respect to the measures $\mu$ and $\nu$, and we then frequently write
$$
\psi=J_T,
$$ 
i.e denoting by $J_T$ this Jacobian. Obviously,
\begin{equation}\label{220180302}
\mu(\psi^{-1}(\infty))=0.
\end{equation}
It follows from \eqref{3mis.cmi2} that for every $k\geq 0$ we have 
\begin{align*}
\nu(T(X_k\cap\psi^{-1}(0)))=\int_{X_k\cap\psi^{-1}(0)}\,\psi d\mu=0.
\end{align*}
Hence, summing over all $k\geq 0$, we get
\begin{align}
\nu(T(\psi^{-1}(0)))=0. \label{4mis.cmi2}
\end{align}

\medskip The map $T$ is called \textit{$\psi$--conformal} with respect to the pair of measures $(\mu,\nu)$ if it is both $\psi$--Bconformal and $\psi$--Fconformal with respect to this pair of measures. We shall prove the following.
\begin{proposition}\label{p1mis.cmi3}
	Let $(X,\cA,\mu)$ and $(Y,\cB,\nu)$ be probability spaces. Assume that $T:X\to Y$ is a measurable map of standard type. Let $\psi:X\to [0,+\infty]$ be a measurable function.
\begin{enumerate}[{{\rm(a)}}]
\item If $T$ is $\psi$--Bconformal with respect to $\mu$ and $\nu$, and $\nu(T(\psi^{-1}(\infty)))=0$, then 

\centerline{$T$ is $\psi$--conformal with respect to $\mu$ and $\nu$.}
\noindent In particular:
$$
J_T=\psi=\hat J_T.
$$
Conversely:
\item If $T$ is $\psi$--Fconformal with respect to $\mu$ and $\nu$ and $\mu(\psi^{-1}(0))=0$, then 

\centerline{$T$ is $\psi$--conformal with respect to $\mu$ and $\nu$.}
\noindent In particular:
$$
\hat J_T=\psi=J_T.  
$$
\end{enumerate} 
\end{proposition}
\begin{proof}
$(a)$ Because of \eqref{120180302}, we have that 
$$
\nu(T(\psi^{-1}(0)))=\nu\(T\(\hat J_T^{-1}(0)\)\)=0.
$$
Let $A\sub X$ be a special set. We then have 
\begin{equation}\label{1mis.cmi3}	
\begin{aligned}
\int_A\psi d\mu
&=\int_X\psi\ind_A d\mu
=\int_Y\tr_{\mu,\nu}(\psi\ind_A)d\nu
=\int_Y\sum_{x\in T^{-1}(y)}\psi(x)\ind_A(x)\left(\frac{1}{\psi(x)}\right)\,d\nu(y)\\
&=\int_{T(A)}\ind d\nu 
=\nu(T(A)) 
\end{aligned}
\end{equation}
So, the map $T$ is $\psi$--Fconformal with respect to $\mu$ and $\nu$, thus the item $(a)$ is established.
	
(b) The $\psi$--Fconformality implies that 
	\begin{align*}
	\frac{d\nu\circ T\rvert_{X_k}}{d\mu}(x)=\psi(x).
	\end{align*}
Since $\mu(\psi^{-1}(0))=0$, this implies that $\mu\rvert_{X_k}\circ(T\rvert_{X_k})^{-1}<<\nu\rvert_{T(X_k)}$ and 
	\begin{align*}
	\frac{d\mu\rvert_{X_k}\circ (T\rvert_{X_k})^{-1}}{d\nu\rvert_{T(X_k)}}=\frac{1}{\psi}\circ(T\rvert_{X_k})^{-1}.
	\end{align*}
Therefore the map $T:X\to Y$ is quasi--invariant with respect to $\mu$ and $\nu$ and $\hat J_T=\psi$. Thus, the item (b) is established, and the proof of the entire proposition is complete.
\end{proof}

We derive the following three immediate consequences of this proposition.
\begin{corollary}\label{c1mis.cmi4}
	Let $(X,\cA,\mu)$ and $(Y,\cB,\nu)$ be probability spaces. Assume that $T:X\to Y$ is a measurable map of standard type. If $\psi:X\to[0,\infty)$ is a measurable function and $T$ is $\psi$--Bconformal with respect to $\mu$ and $\nu$, then 

\centerline{$T$ is $\psi$--conformal with respect to $\mu$ and $\nu$.}
\noindent In particular:
$$
J_T=\psi=\hat J_T.
$$
\end{corollary}
\begin{corollary}\label{c2mis.cmi4}
Let $(X,\cA,\mu)$ and $(Y,\cB,\nu)$ be probability spaces. Assume that $T:X\to Y$ is a measurable map of standard type. If $\psi:X\to(0,\infty]$ is a measurable function and $T$ is $\psi$--Fconformal with respect to $\mu$ and $\nu$, then 

\centerline{$T$ is $\psi$--conformal with respect to $\mu$ and $\nu$ and}
$$
\hat J_T=\psi=J_T.  
$$
\end{corollary}

\begin{corollary}\label{c3mis.cmi4}
Let $(X,\cA,\mu)$ and $(Y,\cB,\nu)$ be probability spaces. Assume that $T:X\to Y$ is a measurable map of standard type. If $\psi:X\to(0,\infty)$ is a measurable function, then the following conditions are equivalent:
\begin{itemize}
\item $T$ is $\psi$--Bconformal with respect to $\mu$ and $\nu$.

\item $T$ is $\psi$--Fconformal with respect to $\mu$ and $\nu$.
		 
\item $T$ is $\psi$--conformal with respect to $\mu$ and $\nu$.
\end{itemize} 
	
\fr If any of these conditions holds, then
$$
\hat J_T=\psi=J_T.  
$$
 	
\end{corollary}
Now we want to formulate an integral criterion for quasi--invariance.
Keep $(X,\cA,\mu)$ and $(Y,\cB,\nu)$ as probability spaces. Assume that $T:X\to Y$ is a measurable map of standard type. Let $\psi:X\to[0,\infty]$ be a measurable function. Denote by $L_\infty^+(X)$ the set of all non-negative bounded functions on $X$ measurable with respect the $\sg$--algebra $\cA$. For every $g\in L_\infty^+(X)$ define $\tr_\psi g:X\to(0,\infty]$ by the following formula:
\begin{align}
\tr_\psi g(y):=\sum_{x\in T^{-1}(y)}g(x)(1/\psi(x)).\label{1mis.cmi5}
\end{align}
Of course $\tr_\psi g\geq 0$ but it may take on the value $\infty$. We shall prove the following.
\begin{proposition}\label{p1mis.cmi5} 
	Let $(X,\cA,\mu)$ and $(Y,\cB,\nu)$ be probability spaces. Assume that $T:X\to Y$ is a measurable map of standard type. Let $\psi:X\to [0,\infty]$ be a measurable function. Then 
	\begin{enumerate}[(a)]
\item $T$ is $\psi$--Bconformal respect to $\mu$ and $\nu$
		
		if and only if
		
\item $\int_Y\tr_\psi gd\nu=\int_X gd\mu$ for every $g\in L^+_\infty(X)$.
	\end{enumerate}
\end{proposition}
\begin{proof}
	The implication $(a)\Rightarrow(b)$ results directly from \eqref{2mis.cmi1}. For the converse, i.e. $(b)\Rightarrow(a)$, observe that \eqref{1mis.cmi5} yields
	\begin{align*}
	\tr_\psi\(h\cdot(f\circ T)\)=f\cdot\tr_\psi(h).
	\end{align*}
	for every $f\in L^+_\infty(Y)$. Having this, \eqref{1mis.cmi2} goes through unchanged to give
	\begin{align*}
	\mu(T^{-1}(F))=\int_F\tr_\psi\ind d\nu.
	\end{align*}
	In particular, if $\nu(F)=0$, then $\mu(T^{-1}(F))=0$, meaning that $T$ is quasi--invariant. It also follows from $(b)$ that $\tr_\psi(L^+_\infty(X))\sub L^1(\nu)$. Consequently $\tr_\psi(L_\infty(X))\sub L^1(\nu)$ and $(b)$ extends to all functions $g\in L_\infty(X)$. The standard approximation argument then shows that $(b)$ extends to the space $L^1_+(\mu)$ and, by linearity, to $L^1(\mu)$. Also, for every $k\geq 0$ denote $T_k:=T\rvert_{X_k}:X_k\to T(X_k)$. Then for every measurable set $A\sub T(X_k)$, we get 
	\begin{align*}
	\mu\(T_k^{-1}(A)\)
	&=\mu(\ind_{T^{-1}_k(A)})
	=\int_Y\tr_\psi\(\ind_{T_k^{-1}(A)}\)d\nu
	=\int_Y\sum_{x\in T^{-1}(y)}\frac{1}{\psi(x)} 
	 \ind_{T_k^{-1}(A)}(x)\, d\nu(y)\\
	&=\int_A\frac{1}{\psi(T_k^{-1}(y))}\, d\nu(y).
	\end{align*}
Therefore $\hat J_T(x)=\psi(x)$ and the implication that $(b)\Rightarrow(a)$ is established. The proof of the proposition is thus complete.
\end{proof}
\begin{remark}\label{c1mis.cmi6.1}
	It follows from the proof of this proposition that in its context, $\tr_\psi g$ is well--defined $\nu$--almost everywhere for all $g\in L^1(\mu)$ and $\tr_\psi(g)\in L^1(\nu)$. Thus the formula 
	\begin{align*}
	\tr_\psi^*\nu(g)=\nu(\tr_\psi g),\quad g\in L^1(\mu),
	\end{align*}
	defines a finite measure $(\tr_\psi^*\nu(\ind)=\nu(\tr_\psi\ind)<+\infty)$ on $Y$, and both items $(a)$ and $(b)$ in Proposition \ref{p1mis.cmi5} become equivalent to 
	\begin{enumerate}
		\item[(c):] $\tr_\psi^*\nu=\mu$.
	\end{enumerate}
\end{remark}
We shall prove the following. 
\begin{proposition}\label{p1mis.cmi6} 
Let $X$ and $Y$ be compact metrizable spaces. Assume that $T:X\to Y$ is a continuous map of standard type.
Let $\mu$ and $\nu$ be Borel probability measures respectively on $X$ and $Y$. Assume that $\psi:X\to(0,\infty)$ is a continuous function such that 
\begin{enumerate}
\item $\tr_\psi(\ind)$ is bounded and, moreover, 

\item $\tr_\psi(C(X))\sub C(Y)$. 
\end{enumerate}
Denote by $\tr^*_\psi:C^*(Y)\to C^*(X)$ the corresponding dual operator. Then the following are equivalent.
\begin{enumerate}[(a)]
\item $T$ is $\psi$--Bconformal with respect to $\mu$ and $\nu$.

\item $T$ is $\psi$--F conformal with respect to $\mu$ and $\nu$.

\item $T$ is $\psi$--conformal with respect to $\mu$ and $\nu$.

\item $\tr_\psi^*\nu=\mu$ in the sense of item $(c)$ from Remark \ref{c1mis.cmi6.1}. 

\item $\tr_\psi^*\nu=\mu$ in the sense resulting from items (1) and (2) above.
		
\,\fr\fr If any of these conditions holds, then
$$
\hat J_T=\psi=J_T.  
$$
\end{enumerate}
\end{proposition}
\begin{proof}
Items	$(a)$, $(b)$ and $(c)$ are equivalent by Corollary~\ref{c3mis.cmi4}. Item (d) is equivalent to them by Remark~\ref{c1mis.cmi6.1}. Of course $(d)$ implies $(e)$. Assuming in turn $(e)$, the standard approximation procedure leads to part $(d)$. The proof is complete.
\end{proof}


\,
\section{Basics of Critically Finite Random Maps of an Interval}
\subsection{The Setting of Random Critically Finite Maps of an Interval}\label{SRIMM}
Given a $C^3$ map $g$ from an interval to $\mathbb R$, we denote the Schwarzian derivative of $g$ by $S(g)$ which is given by 
\begin{align*}
S(g)=\left(\frac{g''}{g'}\right)'-\frac{1}{2}\left(\frac{g''}{g'}\right)^2=\frac{g'''}{g'}-\frac{3}{2}\left(\frac{g''}{g'}\right)^2,
\end{align*}
where $g'$ denotes the usual derivative.

Let $\cG$ be a finite collection of closed intervals with disjoint interiors contained in $I=[0,1]$. We assume that both $0$ and $1$ belong to its union; furthermore we assume that $0$ and $1$ belong to two different elements of $\cG$. Denote them respectively by $\Delta_0$ and $\Delta_1$. Denote:
$$
I_*:=\union_{\Dl\in \cG}\Dl.
$$
For every map $g:I_*\to I$, by $g_0$ and $g_1$ we mean 
$$
g_0:=g\rvert_{\Dl_0} \  \ {\rm  and} \  \ g_1:=g\rvert_{\Dl_1}
$$ 
respectively. Let 
$$
\cG=\cG_C\cup\cG_E \  \ {\rm with} \  \  \cG_C\cap \cG_E=\emptyset.
$$
Let $\kp>1$, let $A>1$, and let $\g:\cG_C\lra\mathbb N=\{1,2,3,,...\}$ be an arbitrary function. Let $\cM(\cG;\kappa, A,\g)$ consist of all continuous maps 
$$
f:I_*\lra I
$$ 
such that the following hold
\begin{enumerate}
\item[\mylabel{(M1)}{(M1)}] For all $\Dl\in \cG$ the map $f\rvert_\Dl$ is $C^3$, injective,  $f(\Dl)=I$, and $S(f\rvert_\Dl) <0$. We denote
$$
f_\Dl^{-1}:=\(f\rvert_\Dl\)^{-1}:I\lra\Dl.
$$
\item[\mylabel{(M2)}{(M2)}] For all $\Dl\in\cG_E$ and for all $x\in\Dl$ we have $\absval{f'(x)}\geq\kp$.
	\item[\mylabel{(M3)}{(M3)}]  For all $\Dl\in\cG_C$ there is a unique point $c_\Dl\in\Dl$, in fact, $c_\Dl\in\partial\Dl$, such that $f'(c_\Dl)=0$.
	\item[\mylabel{(M4)}{(M4)}] For all $\Dl\in\cG_C$, we have $f(c_\Dl)\in\set{0,1}$ and $f(1)=0$ (chosen for ease of exposition).
	\item[\mylabel{(M5)}{(M5)}] Furthermore:
\begin{enumerate}

\item[\mylabel{(M5a)}{(M5a)}]
$$
f(0)=0, \  \ f'(0)\geq\kp>1, \  \ {\rm  and} \  \  f'(1)\leq-\kp,
$$
\item[\mylabel{(M5b)}{(M5b)}]
$$
\|f'\|_\infty\le A,
$$ 
and
\item[\mylabel{(M5c)}{(M5c)}] There exists $s\in (0,1)$ such that
$$
f_0^{-1}\(\max\{|\Dl_0|,|\Dl_1|\}\) \le s,
$$
where $|\Dl|$ denotes the length of an interval $\Dl$, 
and 
$$
\absval{f'(x)}\geq \kp
$$ 
for all $x\in[0,s]\cup [1-s,1]$.
\end{enumerate}
\item [\mylabel{(M6)}{(M6)}] \label{fix: assump: flat crit points}
For every $\Delta\in \cG_C$ there exists a $C^3$ function $A_\Delta:\Delta\to[1/A,A]\cup [-A,-1/A]$ such that 
$$
|A_\Dl'(x)|\in [1/A,A]
$$
and 
$$
f(x)=f(c_\Delta)+A_\Dl(x)(x-c_\Delta)^{1+\g_\Delta}
$$
for every $x\in\Delta$. We will frequently write $\g_c$ for $\g_\Dl$ if $c\in\Dl$.
\end{enumerate}

\begin{figure}[h]\label{graph}
	\centering	
	\begin{tikzpicture}
	\begin{axis}[
	xmin=0,
	xmax=1.001,samples=351,domain=0:1,
	ymin=0,
	ymax=1.001,
	xtick={.15,.3,.4,.5,.6,.68,.7,.8,.9},
	xticklabels={},
	extra x ticks={0,1},
	extra x tick labels={0,1},
	ytick={1},
	xmajorgrids=true,grid style=dashed
	]
	\addplot[domain=0:.3] {1-(1/.15)^2*x*(.3-x)};
	\addplot[domain=.4:.5] {(1/.01)*(x-.5)^2};
	\addplot[domain=.5:.6] {(1/(.1^(1/3)))*(x-.5)^(1/3)};
	\addplot[domain=.6:.7] {10*x^2-(10*.6^2)};
	\addplot[domain=.7:.8] {(1/(.1^(3/4)))*(x-.7)^(3/4)};
	\addplot[domain=.9:1] {(1/(.1^(3/4)))*(1-x)^(3/4)};
	\end{axis}
	\end{tikzpicture}
	\caption{A typical element of $\cM(\cG;\kappa, A,\g)$.} 
\end{figure}

\medskip
We define $f_0^{-1}:I\to\Dl_0$ to be the inverse of the map $f_0:\Delta_0\to I$ and $f_1^{-1}:I\to\Dl_1$ to be the inverse of the map $f_1:\Delta_1\to I$.

\medskip Define 
$$
\Crit(\cG):=\{c_\Dl:\Dl\in \cG_C\}
$$
Further define:
$$
\cG_C^{(0)}:=\big\{\Dl\in \cG_C:f(c_\Dl)=0\big\}
 \  \ \  {\rm and } \  \  \
\cG_C^{(1)}:=\big\{\Dl\in \cG_C:f(c_\Dl)=1\big\},
$$
and finally:
$$
\Crit_0(\cG):=\big\{c_\Dl:\Dl\in \cG_C^{(0)}\big\} 
\  \  \ {\rm and } \  \  \
\Crit_1(\cG):=\big\{c_\Dl:\Dl\in \cG_C^{(1)}\big\}.
$$
In order to avoid the standard case of random expanding systems, whose systematic account can be found for example in \cite{MSU}, we assume that 
$\Crit(\cG)\ne\emptyset$. Furthermore, for the ease and uniformity of exposition, we assume more, namely that 
\begin{enumerate}
\item [\mylabel{(M7)}{(M7)}]
$$
\Crit_0(\cG)\ne\emptyset.
$$ 
\end{enumerate}

Note that with our definition of the space $\cM(\cG;\kappa, A,\g)$, the above critical sets do not depend on the map $f\in \cM(\cG;\kappa, A,\g)$. 
	
\medskip We further assume that

\, \begin{enumerate}
\item[\mylabel{(M8)}{(M8)}] $(\Om,\sF,m)$, a probability space, is given and 
	$$
	\ta:\Om\lra\Om
	$$ 
	is an automorphism preserving measure $m$, i.e. 
	$$
	m\circ\theta^{-1}=m.
	$$
In addition, we assume that the map $\theta:\Om\lra \Om$ is ergodic with respect to $m$, and that the map
 
\, [\mylabel{(M9)}{(M9)}] 
$$
\Om\ni\om\longmapsto T_\om\in\cM(\cG;\kappa, A,\g)
$$
is a measurable map in the sense that for every $x\in I_*$ the map
$$
\Om\ni\om\longmapsto T_\om(x)\in[0,1]
$$
is measurable. 
\end{enumerate}

\noindent Then also the map 
$$
\Om\ni\om\longmapsto T_\om'(x)\in[0,1]
$$
is measurable. The map $\ta:\Om\lra\Om$ is called the base map and the function 
$$
T:\Om\times I_*\lra\Om\times I
$$ 
is the associated skew product map given by
$$
T(\om,x):=(\ta(\om),T_\om(x)).
$$ 
Given an integer $n\ge 1$ and a finite string $\Ga:=(\Ga_0,\Ga_1,\ldots,\Ga_{n-2},\Ga_{n-1}) \in \cG^n:=\prod_{1}^n\cG$ (we will usually denote the elements of the sets $\cG^n$, $n\ge 1$, by $\Ga$), we denote
\begin{equation}
T_{\om,\Ga}^{-n}
:=T_{\om,\Ga_0}^{-1}\circ T_{\th(\om),\Ga_1}^{-1}\ldots T_{\th^{n-2}(\om),\Ga_{n-2}}^{-1}\circ T_{\th^{n-1}(\om),\Ga_{n-1}}^{-1}:I\lra \Gm_0\subset I.
\end{equation}
Then, we define
$$
T_\om^n:\bigcup_{\Ga\in\cG^n}T_{\om,\Ga}^{-n}(I)\lra I
$$
by declaring that for every $\Ga\in\cG^n$ and every $x\in I$,
$$
T_\om^n\big|_{T_{\om,\Ga}^{-n}(I)}(x)=x.
$$
So, for every $\om\in\Om$, every integer $n\ge 0$, and every $x\in\bigcup_{\Ga\in\cG^n}T_{\om,\Ga}^{-n}(I)$, we have
$$
T_\om^n(x)=T_{\th^{n-1}(\om)}\circ T_{\th^{n-2}(\om)}\circ\ldots T_{\th(\om)}\circ T_\om(x).
$$
\begin{definition}
For all $\om\in\Om$ we define the set
	\begin{align*}
		\cJ_\om(T):=\big\{x\in I: T_\om^n(x)\in I_* \;\forall n\geq 0\big\}=\intersect_{n=0}^\infty T_\om^{-n}(I_*).
	\end{align*}
\end{definition}
So, $\cJ_\om(T)$ is a non--empty closed subset of $I_*\sub I$ and the map
$$
\Om\ni\om\longmapsto \cJ_\om(T)\sbt I
$$
is a random closed set. Since for every $\Dl\in\cG$, the map $T_\om\rvert_\Dl:\Dl\lra I$ is bijective, we get that 
\begin{align*}
	T_\om(\cJ_\om(T))&=\union_{\Dl\in\cG}T_\om\left(\Dl\cap T_\om^{-n}(I_*)\right )
	=\union_{\Dl\in\cG}T_\om \left(\intersect_{n=0}^\infty\Dl\cap T_\om^{-1}\left(T_{\ta(\om)}^{-(n-1)}(I_*)\right) \right)\\
	&=\union_{\Dl\in\cG}\intersect_{n=1}^\infty T_{\ta(\om)}^{-(n-1)}(I_*)
	=\intersect_{n=1}^\infty T_{\ta(\om)}^{-(n-1)}(I_*)
	=\intersect_{n=0}^\infty T_{\ta(\om)}^{-n}(I_*)\\
	&=\cJ_{\ta(\om)}(T).
\end{align*}
Setting
\begin{align*}
	\cJ(T):=\union_{\om\in\Om}\set{\om}\times\cJ_\om(T),
\end{align*}
we thus get a topological random dynamical system 
$$
T:\cJ(T)\to\cJ(T)
$$
defined by
$$
T(\om,x):=\left(\ta(\om),T_\om(x)\right).
$$
We call this map a random critically finite map of the interval $[0,1]$.

\, \subsection{Examples}

We shall now describe a very large class of critically finite random maps of an interval. Let $\Om$ be any Borel set contained in $(3,+\infty)$, let $m$ be any Borel probability measure on $\Om$, and let 
$$
\th:\Om\lra\Om
$$
be any invertible Borel measurable map preserving $m$, i.e. $m\circ\th^{-1}=m$. Fix two numbers $0<a<b<1$. Let 
$$
\begin{aligned}
\hat \cG
:=\lt\{\lt[a_j^{(0)},\frac{a_j^{(0)}+b_j^{(0)}}2\rt]\rt\}_{j=1}^s&\cup 
\lt\{\lt[\frac{a_j^{(0)}+b_j^{(0)}}2,b_j^{(0)}\rt]\rt\}_{j=1}^s \cup \\
&\cup \lt\{\lt[a_j^{(1)},\frac{a_j^{(1)}+b_j^{(1)}}2\rt]\rt\}_{j=1}^s\cup 
\lt\{\lt[\frac{a_j^{(1)}+b_j^{(1)}}2,b_j^{(1)}\rt]\rt\}_{j=1}^s
\end{aligned}
$$
be any finite collection of closed subintervals of $[a,b]$ with mutually disjoint interiors. Let 
$$
\cG=\big\{\hat\cG, [0,a], [b,1]\big\}.
$$
For every $\om\in\Om$ define the map 
$$
T_\om:I_*=[0,a]\cup \bu_{j=1}^s[a_j^{(0)},b_j^{(0)}]\cup \bu_{j=1}^s[a_j^{(1)},b_j^{(1)}]\cup [b,1]\lra I
$$
by the following formula:
$$
T_\om(x):=
\begin{cases}
a^{-1}x \  \   &{\rm if }  \  \ 
x\in [0,a] \\ \\
\lt(\frac2{b_j^{(0)}-a_j^{(0)}}\rt)^\om\lt|x-\frac{a_j^{(0)}+b_j^{(0)}}2\rt|^\om 
\  \   &{\rm if }  \  \  
x\in [a_j^{(0)},b_j^{(0)}] \\  \\
1-\lt(\frac2{b_j^{(1)}-a_j^{(1)}}\rt)^\om\lt|x-\frac{a_j^{(1)}+b_j^{(1)}}2\rt|^\om \  \   & {\rm if }  \  \ 
x\in [a_j^{(1)},b_j^{(1)}] \\ \\
\frac{x}{b-1}+\frac{1}{1-b} \  \   &{\rm if }  \  \ 
x\in [b,1].
\end{cases}
$$
It is straightforward (direct calculation) to check that all the maps $T_\om:I_*\to I$ have negative Schwarzian derivative and belong to $\cM(\cG;\kappa, A,\g)$ with appropriate parameters $\kappa$, $A$, $\gamma$. The corresponding skew--product map 
$$
T:\Om\times I_*\lra I
$$
along with the map $\th:\Om\to\Om$ and measure $m$ form a critically finite random map of an interval.

\, \subsection{Preliminaries on Critically Finite Random Maps of an Interval}

Throughout this section always take $\om\in\Om$. For every integer $k\geq 1$, set
\begin{align*}
I_\om(k)
:=[T_{\om,0}^{-k}(1),T_{\om,0}^{-(k-1)}(1)]
 =T_{\om,0}^{-(k-1)}([T_{\ta^{k-1}(\om),0}^{-1}(1),1])
 =T_{\om,0}^{-1}(I_{\ta(\om)}(k-1))\sub\Dl_0, 
\end{align*}
where the last equality sign holds assuming that $k\ge 2$. Furthermore, set
\begin{align*}
	U_\om(k):=\left[T_{\om,0}^{-k}(1),1\right]=\union_{j=1}^kI_\om(j)
\end{align*}
and
\begin{align*}
V_\om(k):=T_{\om,0}^{-k}([0,1])
=\left[0,T_{\om,0}^{-k}(1)\right]
=I\bs\intr{U_\om(k)}
=\set{0}\cup\union_{\ell=k+1}^\infty I_\om(\ell),
\end{align*}
where $\intr{A}$ denotes the interior of $A$.  Note that 
$$
V_\om(1)=\Dl_0.
$$ 
The sequence $(U_\om(k))_{k=1}^\infty$ is 
ascending while $(V_\om(k))_{k=1}^\infty$ is descending with
\begin{align*}
	\union_{k\geq 1}U_\om(k)=(0,1] \spand \intersect_{k\geq 1}V_\om(k)=\set{0}.
\end{align*}
By our conditions \ref{(M1)}--\ref{(M7)}, we have that
\begin{align}
	\absval{I_\om(k)}\comp\absval{(T_\om^{-k})'(0)}^{-1}(1-T_{\ta^{k-1}(\om),0}^{-1}(1))\comp\absval{(T_\om^k)'(0)}^{-1}\label{1mis.po}
\end{align}
and
\begin{align}
	T_{\om,0}^{-k}(1)\comp\absval{(T_\om^k)'(0)}^{-1},\label{2mis.po}
\end{align}
where in here and in the sequel the relation 
$$
A\lek B
$$
means that the exists a constant $C\in(0,+\infty)$ independent of appropriate, always clearly indicated in the context, variables in $A$ and/or $B$ such that
$$
A\le CB.
$$
Analogously, $A\gek B$. Also,
$$
A\comp B
$$
if $A\lek B$ and $B\lek A$. 
 
\, 

Let $\Dl\in\cG_C^{(0)}$. Then for all $x\in\Dl$ we have that,
\begin{align}
	\absval{T_\om'(x)}\comp\absval{x-c_\Dl}^{\gm_\Dl}\spand T_\om(x)\comp\absval{x-c_\Dl}^{1+\gm_\Dl}.\label{3mis.po}
\end{align}
Consequently,
\begin{equation}\label{120180423}
	\absval{T_\om'(x)}\comp T_\om(x)^{\frac{\gm_\Dl}{1+\gm_\Dl}}
\end{equation}
and
\begin{equation}\label{220180423}
	\absval{(T_{\om,\Dl}^{-1})'(z)}\comp z^{-\frac{\gm_\Dl}{1+\gm_\Dl}}
\end{equation}
for every $z\in[0,1]$. Therefore, if $A\sub [0,1]$ and $\diam(A)\comp\dist{0}{A}$, then 
\begin{align*}
	\absval{(T_{\om,\Dl}^{-1})'(x)}\comp(\diam(A))^{-\frac{\gm_\Dl}{1+\gm_\Dl}}
\end{align*}
for every $x\in A$. In particular, for every integer $k\geq 1$ and every $x\in I_{\ta(\om)}(k)$:
\begin{align}
	|(T_{\om,\Dl}^{-1})'(x)|
	\comp\diam(I_{\ta(\om)}(k))^{-\frac{\gm_\Dl}{1+\gm_\Dl}}
	\comp\absval{(T_\om^k)'(0)}^{\frac{\gm_\Dl}{1+\gm_\Dl}}.\label{1mis.po.2}
\end{align}
Likewise for $\Dl\in\cG_C^{(1)}$. If $x\in\Dl$, then\begin{align}
\absval{T_\om'(x)}\comp\absval{x-c_\Dl}^{\gm_\Dl}\spand 1-T_\om(x)\comp\absval{x-c_\Dl}^{1+\gm_\Dl}.\label{3mis.po}
\end{align}
Consequently
\begin{equation}\label{120180423B}
	\absval{T_\om'(x)}\comp (1-T_\om(x))^{\frac{\gm_\Dl}{1+\gm_\Dl}}
\end{equation}
and
\begin{equation}\label{220180423B}
	\absval{(T_{\om,\Dl}^{-1})'(z)}\comp (1-z)^{-\frac{\gm_\Dl}{1+\gm_\Dl}}
\end{equation}
for every $z\in[0,1]$. Therefore, if $A\sub [0,1]$ and $\diam(A)\comp\dist{1}{A}$, then 
\begin{align*}
	\absval{(T_{\om,\Dl}^{-1})'(x)}\comp(\diam(A))^{-\frac{\gm_\Dl}{1+\gm_\Dl}}
\end{align*}
for every $x\in A$.

\,
\subsection{Distortion Properties of Critically Finite Maps of an Interval}

For every bounded interval $J\sbt\mathbb R$ and every $\eta>0$ we denote by $N(J,\eta)$ the interval having the same center as $J$ and whose length is equal to $(1+2\eta)\diam(J)$. In other words, both components of $ N(J,\eta)\sms J$ have lengths equal to $\eta\diam(J)$. We call $N(J,\eta)$ the $\eta$--scaled neighborhood of $J$. Each of the two components of $ N(J,\eta)\sms J$ is called the one--sided $\eta$--scaled collar of $J$.
As a direct consequence of Theorem IV.1.2 (p.277) and Property 4 (p. 273) in \cite{MS} we get the following.

\begin{theorem}[General Bounded Distortion Property]\label{t1mis75}
	Let $\Dl\sub\RR$ be an interval and let $g:\Dl\to\RR$ be a $C^3$--diffeomorphism onto $g(\Dl)$ such that $S(g)<0$ (negative Schwarzian). If $J\sub\Dl$ is a subinterval of $\Dl$ such that $g(\Dl)$ contains an $\eta$--scaled neighborhood  of $g(J)$ then 
	\begin{align*}
	\left(\frac{\eta}{1+\eta}\right)^2\leq\frac{g'(y)}{g'(x)}\leq \left(\frac{1+\eta}{\eta}\right)^2
	\end{align*}
	for all $x,y\in J$. For the ease of notation and expression we denote
	$$
	K_\eta:=\left(\frac{1+\eta}{\eta}\right)^2.
	$$
\end{theorem}
\begin{corollary}\label{c1mis74}
	With the assumptions of Theorem~\ref{t1mis75} 
	\begin{align*}
 \Dl\bus N\lt(J,\frac{1}{2}\eta^3(2+\eta)^{-2}\rt).
	\end{align*}
\end{corollary}
\begin{proof}
Let $\Gm$ be one of the two connected components of $\Dl\bs J$. Let $D$ be the one--sided $(\eta/2)$--scaled collar of $g(J)$ contained in $g(\Gm)$. It follows from Theorem \ref{t1mis75} applied to $g^{-1}(N(g(J),\eta/2))$ that
	\begin{align*}
	\absval{\Gm}&\geq\absval{g^{-1}(D)}
\geq\frac{\diam(D)}{\sup\set{\absval{g'(x)}:x\in g^{-1}(D)}} \\
&\geq\left(\frac{\eta/2}{1+(\eta/2)}\right)^2\frac{\diam(D)}{\inf\set{\absval{g'(x)}:x\in g^{-1}(D)}}
	\geq \left(\frac{\eta}{2+\eta}\right)^2\frac{(\eta/2)\diam(J)}{\inf\set{\absval{g'(x)}:x\in J}}\\
	&\geq \frac{1}{2}\eta^3(2+\eta)^{-2}\diam(J).
	\end{align*}
Hence, $\Dl\bus N\(J,\frac{1}{2}\eta^3(2+\eta)^{-2}\)$. The proof is complete.
\end{proof}

\noindent We will mostly, if not always, apply this Theorem~\ref{t1mis75} for the maps of the form
$$
T_\om^n|_{T_{\om,\Ga}^{-n}(I)}:T_{\om,\Ga}^{-n}(I)\lra I, \  \  n\ge 1, \ \Ga\in\cG^n.
$$
Since the composition of any two maps with negative Schwarzian derivative has again negative Schwarzian derivative, as an immediate consequence of the two above facts and \ref{(M1)}, we respectively get the following.

\begin{proposition}[Special Bounded Distortion Property]\label{p120180630}
For every $\eta\in(0,1/2)$ there exists $K_\eta\in[1,+\infty)$ such that for  all $x$ and $y$ in $[\eta, 1-\eta]$, every integer $n\ge 0$, and every $\Gamma\in\cG^n$, we have that
$$
K_\eta^{-1}\le
\frac{\big|\(T_{\om,\Ga}^{-n}\)'(y)\big|}{\big|\(T_{\om,\Ga}^{-n}\)'(x)\big|}
\le K_\eta,
$$
and, as it consequences,
\begin{enumerate}
\item for every interval $J\sbt [\eta, 1-\eta]$,
$$
\begin{aligned}
K_\eta^{-1}\sup\big\{\big|\(T_{\om,\Ga}^{-n}\)'(z)\big|:z\in J\big\}\diam(J)
&\le \inf\big\{\big|\(T_{\om,\Ga}^{-n}\)'(z)\big|:z\in J\big\}\diam(J) \le \\
&\le \diam\(T_{\om,\Ga}^{-n}(J)\) \le \\
\le \sup\big\{\big|\(T_{\om,\Ga}^{-n}\)'(z)\big|:z\in J\big\}\diam(J)
&\le K_\eta\inf\big\{\big|\(T_{\om,\Ga}^{-n}\)'(z)\big|:z\in J\big\}\diam(J),
\end{aligned}
$$

\item for every $z\in [\eta,1-\eta]$ and every $r>0$ such that $B(z,r)\sbt [\eta,1-\eta]$:
$$
B\lt(T_{\om,\Ga}^{-n}(z),K_\eta^{-1}\big|\(T_{\om,\Ga}^{-n}\)'(z)\big|r\rt)
\sbt T_{\om,\Ga}^{-n}\(B(z,r)\)
\sbt B\lt(T_{\om,\Ga}^{-n}(z),K_\eta\big|\(T_{\om,\Ga}^{-n}\)'(z)\big|r\rt).
$$
\end{enumerate}
\end{proposition}

\,
\subsection{Some Selected Dynamical Properties of Critically Finite Maps of an Interval}
\begin{definition}\label{d2mis75}
	Given $\om\in\Om$, we call an interval $J\sub I$ an $\om$--homterval if for every $n\geq 0$, the map $T_\om^n$ is well--defined on $J$ and $T_\om^n(J)\sub J_n$ with some $J_n\in\cG$.  
\end{definition}
\begin{remark}\label{r1mis77}
	Of course an equivalent characterization of an $\om$--homterval $J\sub I$ is that all iterates $T_\om^n:J\to I$, $n\geq 0$, are well--defined and form diffeomorphisms from $J$ onto $T_\om^n(J)$. 
\end{remark}
Our first crucial technical fact is the following.
\begin{proposition}\label{p2mis77}
	If $T:\cJ(T)\to\cJ(T)$ is a critically finite map, then $T$ has no $\om$--homtervals for any $\om\in\Om$. 
\end{proposition}
\begin{proof}
	By the way of contradiction, suppose that $J\sub I$ is an $\om$--homterval for some $\om\in\Om$. We may assume without loss of generality that $J$ is a maximal, in the sense of inclusion, $\om$--homterval contained in $I$. Seeking contradiction suppose that
	\begin{align}\label{1mis77}
	\liminf_{n\to\infty}\frac{\dist{0}{T_\om^n(J)}}{\diam\(T_\om^n(J)\)}=0.
	\end{align} 
	Fix an arbitrary $n\geq 1$ such that $T_\om^n(J)\cap\Dl_0\nonempty$. Since $J$ is an $\om$--homterval, we then have that $T_\om^n(J)\sub\Dl_0$. Hence, there exists $k\geq 1$ such that
	\begin{align*}
	T_\om^n(J)\cap\intr{I_{\ta^n(\om)}(k)}\nonempty.
	\end{align*} 
	Assume that $k\geq 3$. If $T_\om^n(\cJ(T))\cap I_{\ta^n(\om)}(j)\nonempty$ for some $j\geq k+2$, then 
	\begin{align*}
		T_\om^n(\cJ(T))\bus I_{\ta^n(\om)}(k+1).	
	\end{align*}
	Hence, 
	\begin{align*}
		T_\om^{n+k+1}(\cJ(T))=T_{\ta^n(\om)}^{k+1}(I_{\ta^n(\om)}(k-1))=[0,1].	
	\end{align*}
	This contradiction shows that 
	\begin{align}\label{2mis77}
	T_\om^n(\cJ(T))\cap\union_{j=k+2}^\infty I_{\ta^n(\om)}(j)=\emptyset.
	\end{align}
	Therefore, 
	\begin{align}\label{3mis77}
	\dist{0}{T_\om^n(J)}\geq \diam\(I_{\ta^n(\om)}(k+2)\).
	\end{align} 
	By the same argument as that leading to \eqref{2mis77}, we get
	\begin{align*}
	T_\om^n(J)\cap\union_{j=1}^{k-2}I_{\ta^n(\om)}(j)=\emptyset.
	\end{align*}
	So, 
	$$
	T_\om^n(J)\sub I_{\ta^n(\om)}(k-1)\cup I_{\ta^n(\om)}(k)\cup I_{\ta^n(\om)}(k+1).
	$$ 
In fact either $T_\om^n(J)$ is contained in the first two terms or in the last two terms of this union. Therefore, by looking up at \ref{(M5b)} to write the second inequality below, we get,
\begin{align}
\diam\(T_\om^n(J)\)
&\leq \diam\(I_{\ta^n(\om)}(k-1)\)+\diam\(I_{\ta^n(\om)}(k)\)+\diam\(I_{\ta^n(\om)}(k+1)\) \\
&\leq C_1\diam\(I_{\ta^n(\om)}(k+1)\)
\end{align}
with some universal constant $C_1>0$. So, applying also \eqref{3mis77}, we conclude that 
	\begin{align*}
	\frac{\dist{0}{T_\om^n(J)}}{\diam\(T_\om^n(J)\)}\geq C_1^{-1}.
	\end{align*}
Hence, we have proved that for every $n\geq 0$,
\begin{align}\label{1mis79}
	\frac{\dist{0}{T_\om^n(J)}}{\absval{T_\om^n(J)}}\geq C_2:=\min\set{C_1^{-1},\inf\set{\dist{0}{I_\tau(2)}:\tau\in\Om}}>0,	
	\end{align}
where the second term in this minimum takes care of the cases $k=1$ or $2$ and the infimum there is positive because of \ref{(M5b)}.
	Since $T_\tau(1)=0$ for every $\tau\in \Omega$, because of \ref{(M5b)}, and because of the last formula of \ref{(M5a)}, formula \eqref{1mis79} yields 
	\begin{align*}
	\frac{\dist{1}{T_\om^n(J)}}{\diam\(T_\om^n(J)\)}\geq C_3>0
	\end{align*}
	with some universal constant $C_3>0$. This implies that for every $n\geq 0$ there exists $H_n$, a $C:=\min\set{C_2,C_3}$--scaled neighborhood of $T_\om^n(J)$ contained in $(0,1)$. Therefore, it follows from Corollary \ref{c1mis74} that if $T_{\om,J}^{-n}:(0,1)\to[0,1]$ is the continuous inverse branch of $T_\om^n$ mapping $T_\om^n(J)$ onto $J$, then 
	\begin{align*}
	T_{\om,J}^{-n}((0,1))\bus N(J,\ep),
	\end{align*}
where $\ep:=\frac{1}{2}C^3(2+C)^{-2}$. But then, all the iterates $T_\om^n$ are well-defined and diffeomorphic on $N(J,\ep)$. So, $N(J,\ep)$ is an $\om$--homterval properly contained in $J$. This contradicts the maximality of $J$ and finishes the proof. 
\end{proof}
\begin{corollary}\label{c1mis81}
	For every $\om\in\Om$ we have 
	\begin{align*}
	\lim_{n\to\infty}\sup\set{\diam(T_{\om,\Ga}^{-n}((0,1))): \Ga\in\cG^n}=0.
	\end{align*} 
\end{corollary}
\begin{proof}
	Suppose on the contrary that for some $\om\in\Om$ this limit is not zero. This means that there exists $\ep>0$ and a sequence $\seq[k]{n_k}$ of positive increasing integers such that
	\begin{align}\label{1mis81}
	\diam\(T_{\om,\Ga_{k}}^{-n_k}((0,1))\)\geq \ep
	\end{align}
	for some $\Ga_k\in\cG^{n_k}$ and every $k\in\NN$. For every $k\geq 1$ let $x_k$ be the middle of the interval $T_{\om,\Ga_k}^{-n_k}(0,1)$. Passing to a subsequence if necessary, we may assume that the sequence $\seq[k]{x_k}$ converges to some point $x\in[0,1]$. Because of \eqref{1mis81}, we have that $x\in[\ep/2,1-\ep/2]$, and for all $k\geq 1$ large enough
	\begin{align*}
	\left[x-\frac{\ep}{4},x+\frac{\ep}{4}\right]\sub T_{\om,\Ga_k}^{-n_k}\((0,1)\).
	\end{align*}
	But this means that for all such $k$s, the map $T_\om^{n_k}\rvert_{[x-\ep/4,x+\ep/4]}$ is well--defined and diffeomorphic. Therefore $T_\om^\ell\rvert_{[x-\ep/4,x+\ep/4]}$ is well--defined and diffeomorphic for all $\ell\geq 0$. This however means that $[x-\ep/4,x+\ep,4]$ is an $\om$-homterval. This contradiction with Proposition \ref{p2mis77} finishes the proof.
\end{proof}

\,
\subsection{Fiberwise Invariant Measures and Lyapunov Exponents}

Denote by $M^1_m(T)$ the subset of $M^1_m(\cJ(T))$ consisting of all $T$--invariant probability measures. Denote further by $M_m^{1,0}(T)$ the subset of all elements $\mu\in M_m^1(T)$ for which 
\begin{align*}
\mu(\Om\times\set{0,1})=0.
\end{align*}
Eventually denote by $M_m^1(T)_e$ and $M_m^{1,0}(T)_e$ the respective subsets of ergodic measures in $M_m^1(T)$ and in $M_m^{1,0}(T)$. For every $\mu\in M^1_m(T)$ set
\begin{align*}
\chi_\mu:=\int_{\cJ(T)}\log\absval{T_\om'(z)}d\mu(\om,z)=\int_\Om\int_{\cJ_\om(T)}\log\absval{T_\om'(z)}d\mu_\om(z)dm(\om),
\end{align*}
which we call the \textit{Lyapunov exponent} of $T$ with respect to the measure $\mu$. Note that this integral is well--defined since the function $\cJ(T)\ni(\om,z)\longmapsto\log\absval{T_\om'}$ is, by \ref{(M5b)}, bounded above, although it may however be equal to $-\infty$.

 We shall prove the following. 
\begin{proposition}\label{p1mis83}
	If $\mu\in M_m^{1,0}(T)_e$ then $\chi_\mu>0$.
\end{proposition}
\begin{proof}
	Since $\mu\in M_m^{1,0}(T)_e$, there exists $\eta\in(0,1/2)$ such that
	\begin{align}\label{2mis83}
	\mu(\Om\times[\eta,1-\eta])>\frac{3}{4}.
	\end{align}
	By virtue of Egorov's Theorem, Corollary \ref{c1mis81} yields a measurable set $\Om_1\sub\Om$ such that 
\begin{align}\label{3mis83} 
m(\Om_1)>3/4
\end{align}
and
$$
\diam\(T_{\om,\Ga}^{-n}\((0,1)\)\)<\frac{1-2\eta}{2K_\eta}
$$	
for all $\om\in\Om_1$, all $n\geq 1$ large enough, say $n\geq q\geq 1$, and all $\Ga\in\cG^n$.
It then follows from Proposition~\ref{p120180630} that
	\begin{align*}
	\frac{1-2\eta}{2K_\eta}
	&\geq\diam\(T_{\om,\Ga}^{-n}\((0,1)\)\) 
	\geq\diam\(T_{\om,\Ga}^{-n}\([\eta,1-\eta]\)\)\\
	&\geq \frac{1-2\eta}{K_\eta}\sup\big\{\absval{(T_{\om,\Ga}^{-n})'(x)}:x\in[\eta,1-\eta]\big\}.
	\end{align*}
	Hence, 
	\begin{align*}
	\sup\big\{\absval{(T_{\om,\Ga}^{-n})'(x)}:x\in[\eta,1-\eta]\big\}\leq \frac{1}{2}.
	\end{align*}
	This means that 
	\begin{align}\label{1mis83}
	\absval{(T_\om^n)'(z)}\geq 2
	\end{align}
	whenever $\om\in\Om_1$, $n\geq q$, and $T_\om^n(z)\in[\eta,1-\eta]$. It follows from \eqref{2mis83} and \eqref{3mis83} that 
	\begin{align*}
	\mu\(\Om_1\times[\eta,1-\eta]\)>\frac{1}{2}.
	\end{align*}
	For every $(\om,x)\in\cJ(T)$, let $\seq[k]{n_k}$ be the sequence of consecutive visits to $\Om_1\times[\eta,1-\eta]$ by the map $T^q$. In particular $n_1\ge 1$ and,  
	\begin{align}\label{1mis85}
	T^{qn_k}(\om,x)\in\Om_1\times[\eta,1-\eta]
	\end{align}
	for every $k\geq 1$.
	In view of Birkhoff's Ergodic Theorem, there exists a measurable set $\cJ_0\sub\cJ(T)$ such that $\mu(\cJ_0)=1$,
	\begin{align*}
	\lim_{k\to\infty}\frac{k}{n_k}=\mu(\Om_1\times[\eta,1-\eta])>\frac{1}{2},
	\end{align*}
	and 
	\begin{align*}
	\lim_{k\to\infty}\frac{1}{qn_k}\log\absval{(T_\om^{qn_k})'(x)}=\chi_\mu
	\end{align*}
	for all $(\om,z)\in\cJ_0$. It then follows from \eqref{1mis83} and \eqref{1mis85} that
	\begin{align*}
	\chi_\mu&=\frac{1}{q}\lim_{k\to\infty}\frac{1}{n_k}\sum_{j=0}^{k-1}\log\absval{(T_{\ta^{qn_j}(\om)}^{q(n_{j+1}-n_j)})'(T_\om^{qn_j}(z))}\\
	&\geq\frac{1}{q}\lim_{k\to\infty}\frac{1}{n_k}\sum_{j=0}^{k-1}\log 2
	=\frac{\log 2}{q}\lim_{k\to\infty}\frac{k}{n_k}\\
	&>\frac{\log 2}{2q}>0.
	\end{align*}
	The proof is complete.
\end{proof}

\section{Conformal Measures and Truncated Systems}
Our closest goal now is to prove the existence of conformal measures with some prescribed properties. Since the the appropriate Perron--Frobenius operators $\tr_t$, $t>0$, are not bounded, not even well--defined (on the the Banach space of random continuous bounded functions), the standard method based on the Schauder--Tichonov  Fixed Point Theorem does not directly apply. In order to remedy this difficulty, we truncate our system by ``chopping off'' some suitable, smaller and smaller, neighborhoods of zero. We then, rather easily, construct conformal measures for such truncated systems, by virtue of, the now applicable, Schauder--Tichonov Fixed Point Theorem. The ``true'' conformal measures, i.e. the ones for the original random critically finite map $T:\cJ(T)\to\cJ(T)$, are then obtained as the weak (in the narrow topology) cluster points of the ``truncated'' measures. We start with the following. For every $\om\in\Omega$ and every integer $k\geq 1$, we define 
\begin{align*}
\cJ_\om^{(k)}(T):=\intersect_{n=0}^\infty T_\om^{-n}\left(I_*\cap U_{\ta^n(\om)}(k)\right).
\end{align*}
Because of Corollary~\ref{c1mis81}, for every $\Ga\in \mathcal G^\NN$, the intersection
$$
\pi_\om(\Ga):=\bigcap_{n=0}^\infty T_{\om,\Ga|_n}^{-n}(I)
$$
is a singleton. Obviously,
$$
\pi_\om\((\cG\setminus \{\Dl_0\})^\NN\)\subset \cJ_\om^{(k)}(T)
$$
for all $k\ge 1$. Then, in particular,
$$
\cJ_\om^{(k)}(T)\nonempty.
$$ 
We now shall prove the following result which is necessary for all our further considerations.

\begin{lemma}\label{l1mis.cm1}
	For all $k\geq 1$ and all $\om\in\Omega$, the set $\cJ_\om^{(k)}(T)$ is closed. Furthermore, the set 
	\begin{align*}
		\cJ^{(k)}(T):=\union_{\om\in\Omega}\set{\om}\times\cJ_\om^{(k)}(T)
	\end{align*}
	is a closed random $T$--invariant subset of $\cJ(T)$, invariance meaning that
	$$
	T(\cJ^{(k)}(T))\subset \cJ^{(k)}(T).
	$$
\end{lemma} 
\begin{proof}
	The closedness of the sets $\cJ_\om^{(k)}(T)$ follows immediately from the fact that both sets $I_*$ and $U_{\ta^n(\om)}(k)$ are closed and the maps $T^n_\om$ are continuous. Recalling that the measure $m$ is complete, we see that $\cJ^{(k)}(T)$ is a random closed set as it is a countable intersection of random closed sets. Checking $T$--invariance of $\cJ^{(k)}(T)$, we write 
	\begin{align*}
		T_\om(\cJ_\om^{(k)}(T))&=T_\om\left(\intersect_{n=0}^\infty T_\om^{-n}\left(I_*\cap U_{\ta^n(\om)}(k)\right)\right)
		\sub\intersect_{n=0}^\infty T_\om\left(T_\om^{-n}\left(I_*\cap U_{\ta^n(\om)}(k)\right)\right)\\
		&\sub\intersect_{n=1}^\infty T_\om\left(T_\om^{-n}\left(I_*\cap U_{\ta^n(\om)}(k)\right)\right)\\
		&\sub \intersect_{n=1}^\infty T_{\ta(\om)}^{-(n-1)}\left(I_*\cap U_{\ta^n(\om)}(k)\right)\\
		&=\intersect_{n=1}^\infty T_{\ta(\om)}^{-(n-1)}\left(I_*\cap U_{\ta^{n-1}(\ta(\om))}(k)\right)   \\
		&=\intersect_{n=0}^\infty T_{\ta(\om)}^{-n}\left(I_*\cap U_{\ta^{n}(\ta(\om))}(k)\right)   \\
		&=\cJ_{\ta(\om)}^{(k)}(T).
	\end{align*} 
	The proof is now complete.
\end{proof}

Since $\cJ^{(k)}(T)$ is a closed random set, the space $M^1_m(\cJ^{(k)}(T))$ is 
well--defined, and, we recall, it consists of all measures $\mu$ in $M^1_m(\cJ(T))$ for which $\mu(\cJ^{(k)}(T))=1$.

Because of Lemma~\ref{l1mis.cm1}, we can consider the random dynamical system 
\begin{align*}
T|_k:=	T|_{\cJ^{(k)}(T)}:\cJ^{(k)}(T)\lra\cJ^{(k)}(T).
\end{align*}
It follows directly from the definition of the sets $\cJ^{(k)}(T)$ and our conditions \ref{(M1)}--\ref{(M7)} that for every $k\geq 1$ there exists $\delta_k>0$ such that
$$
\cJ_\om^{(k)}(T)\cap B(\Crit(\cG),\delta_k)=\es
$$
for all $\om\in\Om$ and then that there exists $\eta_k>0$ such that
\begin{equation}\label{120180416}
	\absval{T_\om'\rvert_{\cJ^{(k)}(T)}}\geq\eta_k
\end{equation}
for all $\om\in\Om$.
As already said, we first look for conformal measures for the maps $T|_k:\cJ^{(k)}(T)\to \cJ^{(k)}(T)$. Also, as was already mentioned, the most transparent advantage of working with those truncated maps is that the abstract Perron--Frobenius operators, defined below, are continuous. 

Fix $t\geq 0$. For every $g\in C(\cJ_\om^{(k)}(T))$ and every $x\in\cJ_{\ta(\om)}^{(k)}(T)$, we define (abstract, i.e. not related at the beginning to any quasi--invariant measure) operators
\begin{align}\label{trunc transf op}
	\tr_{t,k,\om}(g)(x):=\sum_{y\in (T|_k)_\om^{-1}(x)}g(y)\absval{T_\om'(y)}^{-t}.
\end{align} 
We also call them Perron--Frobenius operators. Our goal is to produce $t$--conformal measures for the maps $T|_k:\cJ^{(k)}(T)\to\cJ^{(k)}(T)$ and to show that the corresponding Perron--Frobenius operators are in fact given by the starting abstract formula \eqref{trunc transf op} above. Noting that no points  $T_{\om,0}^{-\ell}(1)$ belong to $\cJ_\om^{(k)}(T)$ for any $\om\in\Om$ and any integer $\ell\ge 0$, it follows from our definition of $\cJ^{(k)}(T)$, that $\tr_{t,k,\om}(g)\in C_b(\cJ_{\ta(\om)}^{(k)}(T))$, and so the linear operator 
\begin{align*}
	\tr_{t,k,\om}:C_b(\cJ_{\om}^{(k)}(T))\lra C_b(\cJ_{\ta(\om)}^{(k)}(T))
\end{align*}
is continuous. We can therefore consider the dual operators
\begin{align*}
	\tr_{t,k,\om}^*:C_b^*(\cJ_{\ta(\om)}^{(k)}(T))\lra C_b^*(\cJ_\om^{(k)}(T)),
\end{align*}
which are linear and bounded. From now on we suppress the parameter $t$ and abbreviate $\tr_{k,\om}$ for $\tr_{t,k,\om}$. 

The global Perron--Frobenius operator $\tr_k:C_b(\cJ^{(k)}(T))\lra C_b(\cJ^{(k)}(T))$  defined by 
\begin{align*}
	(\tr_kg)_\om(x):=\tr_{k,\ta^{-1}(\om)}g_{\ta^{-1}(\om)}(x),
\end{align*}
is obviously linear and bounded. Recall that the global dual Perron--Frobenius operator $\tr_k^*:C_b^*(\cJ^{(k)}(T))\lra C_b^*(\cJ^{(k)}(T))$ given by 
\begin{align*}
	(\tr_ks)_\om^*:=\tr_{k,\ta(\om)}^*(s_{\ta(\om)}),
\end{align*}
is also obviously linear and bounded. The dual space $C_b^*(\cJ^{(k)}(T))$ endowed with the weak (narrow) topology becomes a locally convex topological vector space and the dual Perron--Frobenius operator $\tr_k^*$ is, with this topology, continuous.

Fix $\Delta\in\mathcal G\setminus\{\Dl_0\}$. Noting that for every $k\geq 1$ and for every $x\in \cJ_{\ta(\om)}^{(k)}(T)$, $T_{\om,\Delta}^{-1}(x)\in\cJ_\om^{(k)}(T)$, then recalling that the set $\mathcal G$ is finite, and invoking \eqref {120180416} along with \ref{(M5b)}, we obtain the following. 

\begin{observation}\label{o1mis.cm3}
For every integer $k\ge 1$ and for all $\om\in\Omega$, the function $\tr_{k,\om}\ind:\cJ_{\om}^{(k)}(T)\lra\RR$ is everywhere positive on $\cJ_{\ta(\om)}^{(k)}(T)$. Furthermore there exist constants $\beta_1\in(0,+\infty)$, independent of $k$, and $H_k\geq 1$ such that 
	\begin{align*}
		\beta_1\leq \tr_{k,\om}\ind(x)\leq H_k
	\end{align*} 
	for all $\om\in\Om$ and all $x\in\cJ_\om^{(k)}(T)$. 
\end{observation}

\noindent Hence, the formula
\begin{align*}
(F_k(\nu))_\om:=\frac{\tr_{k,\om}^*(\nu_{\ta(\om)})}{\tr_{k,\om}^*(\nu_{\ta(\om)})(\ind)},
\end{align*}
correctly defines a map 
$$
F_k:\cM^1(\cJ^{(k)}(T))\lra\cM^1(\cJ^{(k)}(T)).
$$
We shall prove the following.
\begin{lemma}\label{l2mis.cm3}
For all $k\geq 1$ large enough the map $F_k:M^1_m(\cJ^{(k)}(T))\to M^1_m(\cJ^{(k)}(T))$ is continuous in the narrow topology.
\end{lemma}
\begin{proof}
	Let $\Lm$ be a directed set and let $(\nu^\lm)_{\lm\in\Lm}$ be a net of measures in $\cM^1_m(\cJ^{(k)}(T))$ converging to some measure $\nu\in \cM^1_m(\cJ^{(k)}(T))$ in the narrow topology. If 
$$
h_{\om,\lam} := \frac1{\nu^\lam _\om (\tr_{k,\theta^{-1}(\om)} \ind )},
$$
then
$$ 
F_k(\nu^\lam)_\om
= \cL^*_{k,\om}\left( \frac{1}{\nu^\lam_{\theta(\om)} (\tr_{k,\om}\ind )} \nu^\lam_{\theta(\om)} \right)
=\cL^*_{k,\om}\left( h_{\theta(\om),\lam}\;  \nu^\lam_{\theta(\om)} \right)\,.
$$ 
Because of Observation~\ref{o1mis.cm3} we have
\begin{equation}\label{3.17}
H_k^{-1}\leq h_{\theta(\om),\lam}\leq \beta_1^{-1}.
\end{equation}
The measures $h_{\lam}\nu^\lam$, $\lam\in\Lambda$, need not be random probability measures but, due to \eqref{3.17}, their fiber total values $h_{\theta(\om),\lam}\nu^\lam_{\theta(\om)}(\cJ^{(k)}(T))$ are uniformly bounded away from zero and $\infty$. Since $\cM_m(\cJ^{(k)}(T))$ is compact, it follows that $\{h_{\lam}\nu^\lam:\lam\in\Lambda\}$ is a tight family. Let $\mu$ be an arbitrary accumulation point of this net.
It has been shown (as a matter of fact for sequences but the same argument works for all nets) in Lemma~2.9 of \cite{RoyUrb2011} that then 
\begin{equation}\label{5_2018_01_12} 
\mu = h \nu
\end{equation}
for some measurable function $h:\Om\lra (0,+\infty )$. Since the dual operator 
$$
\tr_k^*:C_b^*(\cJ_{\theta(\om)}^{(k)}(T))\to C_b^*(\cJ_\om^{(k)}(T))
$$ 
is continuous, the measure $\tr_k^*\mu$ is an accumulation point of the net $\(\tr_k^*\(h_{\lam}\nu^\lam\):\lam\in\Lambda\)$. But all elements of this net belong to $\cM_m^1(\cJ^{(k)}(T))$, whence $\tr_k^*\mu\in \cM_m^1(\cJ^{(k)}(T))$. This means that the disintegrations of this measure 
$$
\tr^*_{k,\om}\mu_{\theta(\om)} = h_{\theta(\om)} \tr^*_{k,\om} ( \nu _{\theta(\om)} ), \ x\in X,
$$
are all (Borel) probability measures on $\cJ_\om^{(k)}(T)$. Therefore, 
$$
1= \tr^*_{k,\om} \mu_{\theta(\om)} (\ind)
= h_{\theta(\om)} \tr^*_{k,\om}( \nu _{\theta(\om)} )(\ind).
$$
This means that
$$
h_\om= \frac{1}{\tr^*_{k,\theta^{-1}(\om)} (\nu _\om)(\ind)}.
$$
Along with \eqref{5_2018_01_12} this implies that 
$$
\mu_\om 
= \frac{1}{\tr^*_{k,\theta^{-1}(\om)} (\nu_\om) )(\ind)} \nu_\om, \quad \om\in \Om.
$$
This means that 
$$
\tr_k^*\mu=F_k(\nu).
$$
Since also 
$$
\(\tr_k^*\(h_{\lam}\nu^\lam\):\lam\in\Lambda\)=\(F_k(\nu^\lam):\lam\in\Lambda\), 
$$
we thus conclude that the net $\(F_k(\nu^\lam):\lam\in\Lambda\)$ converges to $F_k(\nu)$. The proof of continuity of the map $F_k$ is complete. 
\end{proof}

\medskip Since $\cM^1_m(\cJ^{(k)}(T))$ is a convex compact subset of the locally convex topological vector space $C_b^*(\cJ^{(k)}(T))$ and by Lemma \ref{l2mis.cm3} the map $F_k:\cM^1_m(\cJ^{(k)}(T))\lra \cM^1_m(\cJ^{(k)}(T))$ is continuous, applying the Schauder--Tichonov Fixed Point Theorem, we get the following. 
\begin{lemma}\label{l1mis.cm5}
Each map $F_k:M^1_m(\cJ^{(k)}(T))\to M^1_m(\cJ^{(k)}(T))$, $k\geq 1$, has a fixed point.
\end{lemma}
Denote a fixed point, produced by this lemma, for $k\geq 1$, by $\nu_k$. Being such a fixed point means that there exists a measurable function $\lm_k:\Om\to(0,\infty)$ such that 
\begin{align}\label{2mis.cm5}
	\tr_{k,\ta(\om)}^*\nu_{\ta(\om)}=\lm_{k,\om}\nu_{k,\om}.
\end{align}
Iterating \eqref{2mis.cm5}, we get 
\begin{align}\label{0mis.cm5}
	\tr_{k,\ta^n(\om)}^{*n}\nu_{k,\ta^n(\om)}=\lm_{k,\om}^n\nu_{k,\om},
\end{align}
where, 
\begin{align*}
\tr_{k,\ta^n(\om)}^{*n}
:=\tr_{k,\ta(\om)}^*\circ\ldots\circ\tr_{k,\ta^{n-2}(\om)}^*\circ\tr_{k,\ta^{n-1}(\om)}^*
:C_b^*(\cJ_{\th^n(\om)}^{(k)}(T))\lra C_b^*(\cJ_\om^{(k)}(T))
\end{align*}
and
\begin{align}\label{-1mis.cm5}
\lm_{k,\om}^n:=\lm_{k,\om}\lm_{k,\ta(\om)}\cdots\lm_{k,\ta^{n-1}(\om)}.
\end{align}
An equivalent form of \eqref{0mis.cm5} is
\begin{align}\label{1mis.cm5}
	\nu_{k,\om}(T_{k,\om,\Ga}^{-n}(A))=\lm_{k,\om}^{-n}\int_A\absval{(T_{k,\om,\Ga}^{-n})'}^t d\nu_{k,\ta^n(\om)}
\end{align}
for every integer $n\geq 0$, every $\Ga\in\cG^n$, and every Borel set $A\sub\cJ_{\ta^n(\om)}^{(k)}(T)$. Since each map $T_\om:\cJ_\om^{(k)}(T)\to \cJ_{\ta(\om)}^{(k)}(T)$ is obviously of standard type, using Proposition~\ref{p1mis.cmi6}, we can compactly summarize all the above in the following.

\begin{lemma}\label{l1_2018_03_23}
Given $t\ge 0$, for every integer $k\ge 1$ the random measure $\nu_k=\nu_{t,k}$ on $\cJ^{(k)}(T)$, produced in Lemma~\ref{l1mis.cm5}, is such that for $m$--a.e. $\om\in\Omega$ the map $T_\om:\cJ_\om^{(k)}(T)\to\cJ_{\ta(\om)}^{(k)}(T)$ is 

\,

\centerline{$\lam_\om|T_\om'|^t$--conformal}

\,
\fr with respect to the pair of measures $\nu_{k,\om}$ and $\nu_{k,\ta(\om)}$. We will then simply say that the map $T_\om:\cJ_\om^{(k)}(T)\to\cJ_{\ta(\om)}^{(k)}(T)$ is 

\centerline{$t$--conformal.}
\end{lemma}

Frequently, we will refer to these measures as truncated $t$--conformal measures. The word truncated alludes to the fact that these measures are supported on $\cJ_\om^{(k)}(T)$ rather than on the full set $\cJ_\om(T)$. Applying \eqref{2mis.cm5} and Observation~\ref{o1mis.cm3}, gives 
\begin{align}\label{1mis.cm5.1}
\lm_{k,\om}
=\lm_{k,\om}\nu_\om(\ind)
=\tr_{k,\om}^*\nu_{\ta(\om)}(\ind)
=\nu_{\ta(\om)}(\tr_{k,\om}\ind)
\geq\nu_{\ta(\om)}(\bt_1\ind)=\bt_1.
\end{align}
Correspondingly, we can use \eqref{0mis.cm5} to rewrite \eqref{-1mis.cm5} as
\begin{align*}
	\lm_{k,\om}^n=\nu_{\ta^n(\om)}(\tr_{k,\om}^n\ind).
\end{align*}
\medskip \section{Estimates of truncated conformal measures}
In this section we continue dealing with the truncated $t$--conformal measures $\nu_{k,\om}$ of the previous section and we establish several estimates of the values of measures $\nu_{k,\om}$ applied to some significant subsets of $\cJ_\om^{(k)}(T)$. We start by recording the following immediate consequence of items \ref{(M1)}--\ref{(M7)}, $t$--conformality of $\nu$, and the Special Bounded Distortion Property, i.e. Proposition~\ref{p120180630}. For every $2\leq j\leq k$ we have 
\begin{align}\label{1mis.cm6}
	\nu_{k,\ta(\om)}(I_{\ta(\om)}(j-1))\comp\lm_{k,\om}\absval{T_\om'(0)}^t\nu_{k,\om}(I_\om(j)).
\end{align}
It also follows from \eqref{1mis.po.2} that for any $\Dl\in \cG_C^{(0)}$,  which is non--empty by \ref{(M7)}, and every $\ell\geq 1$,
\begin{align}\label{2mis.cm6B}
1\geq 
\nu_{k,\om}(T_{\om,\Dl}^{-1}(I_{\ta(\om)}(\ell)))
\comp\lm_{k,\om}^{-1}\absval{(T_\om^\ell)'(0)}^{\frac{t\gm_\Dl}{1+\gm_\Dl}}\nu_{k,\ta(\om)}(I_{\ta(\om)}(\ell)).
\end{align}
Hence, 
\begin{align}\label{2mis.cm6}
	\nu_{k,\ta(\om)}\left(I_{\ta(\om)}(\ell)\right)
	\comp \lm_{k,\om}\absval{(T_\om^\ell)'(0)}^{-\frac{t\gm_\Dl}{1+\gm_\Dl}}\nu_{k,\ta(\om)}\left(T_{\om,\Dl}^{-1}\left(I_{\ta(\om)}(\ell)\right)\right).
\end{align}
Fix a measurable function $\ep:\Om\to(0,+\infty)$ and for every $2\leq j\leq k$, define
\begin{align}
	\Om_k(j,\ep):=\set{\om\in\Om:\nu_{k,\om}(I_\om(j))\leq \ep(\om)}.
\end{align}
Assume that $\Om_k^c(j,\ep)\nonempty$. Fix $\om\in\Om_k^c(j,\ep)$. Then by \eqref{1mis.cm6} we have 
\begin{align*}
	\lm_{k,\om}
\lek \frac{1}{\ep(\om)}\absval{T_\om'(0)}^{-t}\nu_{k,\ta(\om)}(I_{\ta(\om)}(j-1)).
\end{align*}
Inserting this into \eqref{2mis.cm6} with $\ell=j-1$, yields
\begin{align*}
\nu_{k,\ta(\om)}(I_{\ta(\om)}(j-1))
\lek\frac{1}{\ep(\om)}\nu_{k,\ta(\om)}(I_{\ta(\om)}(j-1))\absval{(T_\om^{j-1})'(0)}^{-\frac{t\gm_\Dl}{1+\gm_\Dl}}\cdot\nu_{k,\om}(T_{\om,\Dl}^{-1}(I_{\ta(\om)}(j-1))).
\end{align*}
Therefore
\begin{align*}
	\ep(\om)
\lek\absval{(T_\om^{j-1})'(0)}^{-\frac{t\gm_\Dl}{1+\gm_\Dl}}\nu_{k,\om}(T_{\om,\Dl}^{-1}(I_{\ta(\om)}(j-1)))).
\end{align*}
This implies that if 
\begin{align*}
	\ep(\om):=\Gamma \absval{(T_\om^{j-1})'(0)}^{-\frac{t\gm_\Dl}{1+\gm_\Dl}}\nu_{k,\om}(T_{\om,\Dl}^{-1}(I_{\ta(\om)}(j-1))),
\end{align*}
with a sufficiently large constant $\Gamma$, then $\Om_k^c(j,\ep)=\emptyset$. Equivalently, $\Om_k(j,\ep)=\Om$, and therefore, 
\begin{align}\label{1mis.cm7}
	\nu_{k,\om}(I_\om(j))\le \Gamma \absval{(T_\om^{j-1})'(0)}^{-\frac{t\gm_\Dl}{1+\gm_\Dl}}\nu_{k,\om}(T_{\om,\Dl}^{-1}(I_{\ta(\om)}(j-1)))
\end{align}
for every $j\geq 2$ and all $\om\in\Om$. Along with \eqref{1mis.cm5.1}, this immediately gives
\begin{equation}\label{2mis.cm7}
\begin{aligned}
\nu_{k,\om}(T_{\om,1}^{-1}(I_{\ta(\om)}(j)))
&\comp\lm_{k,\om}^{-1}\absval{T_\om'(1)}^{-t}\nu_{k,\ta(\om)}(I_{\ta(\om)}(j))
\lek \nu_{k,\ta(\om)}(I_{\ta(\om)}(j)) \\
&\le \Gamma\absval{(T_{\ta(\om)}^{j-1})'(0)}^{-\frac{t\gm_\Dl}{1+\gm_\Dl}}\nu_{k,\ta(\om)}(T_{\ta(\om),\Dl}^{-1}(I_{\ta(\om)}(j-1)).
\end{aligned}
\end{equation}
Formulas \eqref{1mis.cm7} and \eqref{2mis.cm7} along with \ref{(M5a)} immediately entail the following respective two formulas: for all integers $j\geq 2$ we have that
\begin{align}\label{1mis.cm7.1}
\nu_{k,\om}(I_\om(j))
\lek\kp^{-\frac{t\gm_{0,+}}{1+\gm_{0,+}}j},
\end{align} 
and 
\begin{align}\label{2mis.cm7.1}
\nu_{k,\om}(T_{\om,1}^{-1}(I_{\ta(\om)}(j)))
\lek\kp^{-\frac{t\gm_{0,+}}{1+\gm_{0,+}}j},
\end{align}
where 
$$
\gm_{0,+}:=\max\big\{\gm_\Dl:\Dl\in\cG_C^{(0)}\big\}\ge 1.
$$
The four formulas \eqref{1mis.cm7}--\eqref{2mis.cm7.1} entail the following four formulas ($\ell\geq 2$):
\begin{align}
	\nu_{k,\om}(V_\om(\ell))
	&\lesssim\sum_{j=l}^k\absval{(T_\om^{j-1})'(0)}^{-\frac{t\gm_\Dl}{1+\gm_\Dl}}\nu_{k,\om}(T_{\om,\Dl}^{-1}(I_{\ta(\om)}(j-1))\label{1mis.cm7.1.1},\\
	\nu_{k,\om}(V_\om(\ell))
	&\lesssim \sum_{j=l}^k\kp^{-\frac{t\gm_{0,+}}{1+\gm_{0,+}}j}
	\lesssim \kp^{-\frac{t\gm_{0,+}}{1+\gm_{0,+}}l},   \label{2mis.cm7.1.1}
\end{align}
and
\begin{align}
\nu_{k,\om}(T_{\om,1}^{-1}(V_\om(\ell)))
&\lesssim\sum_{j=l}^k\absval{(T_\om^{j-1})'(0)}^{-\frac{t\gm_\Dl}{1+\gm_\Dl}}\nu_{k,\om}(T_{\om,\Dl}^{-1}(I_{\ta(\om)}(j-1))\label{3mis.cm7.1.1},\\
\nu_{k,\om}(T_{\om,1}^{-1}(V_\om(\ell)))
&\lesssim \kp^{-\frac{t\gm_{0,+}}{1+\gm_{0,+}}l}. \label{4mis.cm7.1.1}
\end{align}
\begin{remark}\label{r1mis.cm7.1.1}
Observe that obviously for all $j\geq 3$, the same respective formulas as \eqref{1mis.cm7}, \eqref{2mis.cm7}, and \eqref{1mis.cm7.1} hold also for the sets $I_{\om}(j-1)\cup I_{\om}(j)\cup I_{\om}(j+1)$ and $I_{\ta(\om)}(j-1)\cup I_{\ta(\om)}(j)\cup I_{\ta(\om)}(j+1)$ respectively in place of $I_\om(j)$ and $I_{\ta(\om)}(j)$; only the constants on the right--hand sides of these formulas may differ, i.e. may increase. 
\end{remark}
In the formula \eqref{1mis.cm5.1} we have recorded a straight forward lower bound for $\lm_{\om}$s. We do also need upper bounds. These are more involved and we will produce them now. Keep $k\ge 2$. From Lemma \ref{l1mis.cm5}, we get the following
\begin{equation}\label{3mis.cm7.1}
\begin{aligned}
	\lm_{k,\om}
	&=\lm_{k,\om}\nu_{k,\om}(\ind)
	=\tr_{k,\om}^*\nu_{k,\ta(\om)}(\ind)
	=\nu_{k,\ta(\om)(\tr_{k,\om}\ind)}\\
	&=\int_{I_{\ta(\om)}(1)}\tr_{k,\om}\ind d\nu_{k,\ta(\om)}
	+\sum_{j=2}^k\int_{I_{\ta(\om)}(j)}\tr_{k,\om}\ind d\nu_{k,\ta(\om)}.
\end{aligned}
\end{equation}
Let us estimate first the second summand. We denote it by $\Sg_2(\om)$. The inverse images coming from non--critical inverse branches give uniformly bounded (above) contributions to $\tr_{k,\om}\ind$; therefore, using the definition \eqref{trunc transf op} and \eqref{1mis.po.2} and \eqref{1mis.cm7} we have
\begin{align}
	\Sg_2(\om)&:=\sum_{j=2}^k\int_{I_{\ta(\om)}(j)}\tr_{k,\om}\ind d\nu_{k,\ta(\om)}\\
	&\lek\left(1 +\sum_{j=2}^k\sum_{\Dl\in\cG_C^{(0)}}\absval{(T_\om^j)'(0)}^{\frac{t\gm_\Dl}{1+\gm_\Dl}}\right)\nu_{k,\ta(\om)}(I_{\ta(\om)}(j))\nonumber\\
	&\lek \left(1+\sum_{j=2}^k\sum_{\Dl\in\cG_C^{(0)}}\absval{(T_\om^j)'(0)}^{\frac{t\gm_\Dl}{1+\gm_\Dl}}\right)\absval{(T_\om^j)'(0)}^{-\frac{t\gm_\Dl}{1+\gm_\Dl}}\cdot\nu_{k,\ta(\om)}(T_{\ta(\om),\Dl}^{-1}(I_{\ta^2(\om)}(j-1)))\nonumber\\
	&\lesssim 1+\sum_{j=2}^k\nu_{k,\ta(\om)}\left(\union_{\Dl\in\cG_C^{(0)}}T_{\ta(\om),\Dl}^{-1}(I_{\ta^2(\om)}(j-1))\right)\nonumber\\
	&\lesssim 1+\nu_{k,\ta(\om)}\left(\union_{\Dl\in\cG_C^{(0)}}T_{\ta(\om),\Dl}^{-1}(V_{\ta^2(\om)}(1))\right)\lesssim 1.\label{2mis.cm7.2}
\end{align}
Likewise, $\Sg_1(\om)$, the first summand of \eqref{3mis.cm7.1} can be estimated above by 
\begin{align}\label{1mis.cm7.2}
	\Sg_1(\om)&:=\int_{I_{\ta(\om)}(1)}\tr_{k,\om}\ind d\nu_{k,\ta(\om)}\leq\nonumber\\
	&\leq\const+\int_{\cup_{j=2}^k T_{\om,1}^{-1}(I_{\ta(\om)}(j))}\tr_{k,\om}\ind d\nu_{k,\ta(\om)}
	=\const+\sum_{j=2}^k \int_{ T_{\ta(\om),1}^{-1}(I_{\ta^2(\om)}(j))}\tr_{k,\om}\ind d\nu_{k,\ta(\om)}, 
\end{align}
where the ``const'' term in this formula is responsible for the integration along the part of $I_{\ta(\om)}(1)$ not covered  by $\cup_{j=2}^k T_{\om,1}^{-1}(I_{\ta(\om)}(j))$. Denote the second summand of this formula by $\Sg_1^*$. Using \eqref{2mis.cm7}, we then get
\begin{align*}
	\Sg_1^*(\om)&:=\sum_{j=2}^k \int_{ T_{\ta(\om),1}^{-1}(I_{\ta^2(\om)}(j))}\tr_{k,\om}\ind d\nu_{k,\ta(\om)}\lek\\
	&\lek \left(1 +\sum_{j=2}^k \sum_{\Dl\in\cG_C^{(1)}}\diam\left(T_{\ta(\om),1}^{-1}(I_{\ta^2(\om)}(j))\right)\right)^{-\frac{t\gm_\Dl}{1+\gm_\Dl}}\nu_{k,\ta(\om)}\left(T_{\ta(\om),1}^{-1}(I_{\ta^2(\om)}(j))\right)\\
	&\lesssim \left(1+\sum_{\Dl\in\cG_C^{(1)}}\sum_{j=2}^k\absval{(T_{\ta(\om)}^j)'(0)}^{\frac{t\gm_\Dl}{1+\gm_\Dl}}\right)\absval{(T_{\ta^2(\om)}^{j-1})'(0)}^{-\frac{-\gm_\Dl}{1+\gm_\Dl}}\nu_{k,\ta^2(\om)}\left(T_{\ta^2(\om),\Dl}^{-1}(I_{\ta^3(\om)}(j-1))\right)\\
	&\lek 1+\sum_{j=2}^k\nu_{k,\ta^2(\om)}\left(\union_{\Dl\in\cG_C^{(1)}}T_{\ta^2(\om),\Dl}^{-1}(I_{\ta^3(\om)}(j-1))\right)\\
	&\lek 1+\nu_{k,\ta^2(\om)}\left(\union_{\Dl\in\cG_C^{(1)}}T_{\ta^2(\om),\Dl}^{-1}(V_{\ta^3(\om)}(1))\right)\lesssim 1.
\end{align*}
Combining this with \eqref{3mis.cm7.1}, \eqref{1mis.cm7.2}, and \eqref{2mis.cm7.2}, we get that 
\begin{align*}
	\lm_{k,\om}\lesssim 1.
\end{align*}
Combining in turn this with \eqref{1mis.cm5.1}, we get the following.
\begin{lemma}\label{l1mis.cm7.2}
There exist two constants $0<\bt_1\leq \bt_2<\infty$ such that 
	\begin{align}
		\bt_1\leq\lm_{k,\om}\leq\bt_2
	\end{align}
	for all $k\geq 2$ and all $\om\in\Om$.
\end{lemma}

We need also different estimates. Conformality of $\nu_k$ gives that
\begin{align}\label{4mis118}
\nu_\om(I_\om(j))\lesssim\lm_\om^{-j}\absval{(T_\om^j)'(0)}^{-t}
\end{align}
for all $j\geq 2$ and all $\om\in\Om$. Consequently, using Lemma~\ref{l1mis.cm7.2}, we get for all $\Dl\in\cG_C$ that, 

\begin{align} 
\nu_\om\left(T_{\om,\Dl}^{-1}(I_{\ta(\om)}(j)\right)
&\lesssim\lm_{\ta(\om)}^{-j}\absval{(T_{\ta(\om)}^j)'(0)}^{-t}\lm_\om^{-1}\absval{T_{\ta(\om)}^{-j}(1)}^{-\frac{\gm_\Dl}{1+\gm_\Dl}t}\nonumber\\
&\comp\lm_\om^{-(j+1)}\absval{(T_{\ta(\om)}^j)'(0)}^{-t}\cdot\absval{(T_{\ta(\om)}^{j})'(0)}^{\frac{\gm_\Dl}{1+\gm_\Dl}t}  \nonumber\\
&=\lm_\om^{-(j+1)}\absval{(T_{\ta(\om)}^j)'(0)}^{-\frac{1}{1+\gm_\Dl}t}\comp\lm_\om^{-j}\absval{(T_\om^j)'(0)}^{-\frac{t}{1+\gm_\Dl}}\nonumber\\
&\leq\lm_\om^{-j}\absval{(T_\om^j)'(0)}^{-\frac{t}{1+\gm_+}},\label{3mis118}
\end{align}
where
$$
\gm_+:=\max\set{\gm_\Dl:\Dl\in\cG_C}\ge 1.
$$
It also follows from \eqref{4mis118} that
\begin{align}\label{?:6}
\nu_\om(T_{\om,1}^{-1}(I_{\ta(\om)}(j)))\lesssim\lm_\om^{-j}\absval{(T_\om^j)'(0)}^{-t},	
\end{align}
and, consequently, by the same argument as the one leading to \eqref{3mis118}, we get for all $\Dl\in\cG$, that
\begin{align}\label{1mis118}
\nu_\om(T_{\om,\Dl}^{-1}\circ T_{\ta(\om),1}^{-1}(I_{\ta^2(\om)}(j)))
\lesssim\lm_\om^{-j}\absval{(T_\om^j)'(0)}^{-\frac{t}{1+\gm_+}}.
\end{align}

\, 

\section{The Existence of Conformal Measures}
Similarly as in the last two sections, in this section we deal with abstract Perron--Frobenius operators, this time however, for the full skew--product map $T:\cJ(T)\to\cJ(T)$. The goal is to use the results of the two previous sections to show that the weak limit points, in the narrow topology of $\cM_m^1(\cJ(T))$, of the measures $\nu_k$, $k\ge 1$, are conformal for the map $T:\cJ(T)\to\cJ(T)$, and the corresponding Perron--Frobenius operator is in fact given by the starting abstract formula. We also derive several properties of these conformal measures. 
We start with the following definition.

\begin{definition}
For every $g\in C_b(\cJ(T))$ and every $\om\in\Om$, we define
	\begin{align*}
		\tr_\om(g)(w):=\sum_{z\in T_{\om}^{-1}(w)}g(z)\absval{T_\om'(z)}^{-t}.
	\end{align*}
We still call $\tr_\om$ the transfer Perron--Frobenius operator even though $\tr_\om(g)$ need not belong to $C_b\(\cJ_{\th(\om)}(T)\)$ and the space $C_b(\cJ(T))$ is not in general preserved by the corresponding global operator.
\end{definition}

Fix $t\ge 0$, and let $\nu_{t,k}=\nu_k$, $k\ge 1$, be the truncated $t$--conformal measures produced in the previous section; see Lemma~\ref{l1mis.cm5} and Lemma~\ref{l1_2018_03_23}. We now treat them as elements of $M^1_m(\cJ(T))$. Since $M^1_m(\cJ(T))$ is a compact set with respect to the narrow topology, we will consider weak limits of the sequence $(\nu_k)_{k=1}^\infty$. Although convergence in the narrow topology does not in general entail weak convergence in fibers, nevertheless, as a consequence of Remark \ref{r1mis.cm7.1.1} and the fact that 
\begin{align*}
	I_\om(j)\sub\intr{I_{\om}(j-1)\cup I_\om(j)\cup I_{\om}(j+1)}
\end{align*} 
(and likewise for $\ta(\om)$), we obtain the following.
\begin{lemma}\label{l1mis.cm7.1.1}
	The formulas \eqref{1mis.cm7} -- \eqref{4mis.cm7.1.1}, hold also for $\nu$, any weak limit of the sequence $(\nu_k)_{k=1}^\infty$.
\end{lemma} 
As an immediate consequence of this lemma, we get the following.
\begin{proposition}\label{p1mis.cm7.1.2}
	For $\nu$, any weak limit of the sequence $(\nu_k)_{k=1}^\infty$, we have that 
	\begin{align*}
		\nu(\Om\times\set{0,1})=0.
	\end{align*} 
\end{proposition}

Keeping a random critically finite map $T:\cJ(T)\to \cJ(T)$ and a real number $t\geq 0$, a random measure $\nu=\nu_t$ on $\cJ(T)$ is called $t$--\textit{Fconformal} if there exists a measurable function $\lm:\Omega\to(0,\infty)$ such that
\begin{align}\label{eqn: conformal measure}
\nu_{\ta(\om)}(T_\om(A))=\lm_\om\int_A\absval{T'_\om}^td\nu_\om
\end{align}
for every special set $A\sub \cJ_\om(T)$ and $m$-a.e. $\om\in\Omega$. In other words, the random measure $\nu$ is $t$--Fconformal if for $m$-a.e. $\om\in\Omega$, the map $T_\om:\cJ_\om(T)\to \cJ_{\ta(\om)}(T)$ is $\lm_\om\absval{T_\om'(z)}^t$--Fconformal with respect to the pair of measures $\nu_\om$ and $\nu_{\th(\om)}$. We then also call each map $T_\om:\cJ_\om(T)\to \cJ_{\ta(\om)}(T)$, $\om\in\Omega$, $t$--\textit{Fconformal} with respect to the pair of measures $\nu_\om$ and $\nu_{\th(\om)}$.

Iterating \eqref{eqn: conformal measure}, we see that for each $n\in\NN$ we have
\begin{align*}
	\nu_{\ta^n(\om)}(T_{\om}^{n}(A))=\lm_{\om}^{n}\int_A\absval{(T_{\om}^{n})'}^t d\nu_{\om}
\end{align*} 
where
\begin{align*}
	\lm_\om^n:=\lm_{\om}\lm_{\ta(\om)}\cdots\lm_{\ta^{n-1}(\om)}.
\end{align*}

\medskip
A random measure $\nu\in M_m^1(\cJ(T))$ is called \textit{$t$--Bconformal} if  there exists a measurable function $\lm:\Omega\to(0,\infty)$ such that $m$--a.e. map $T_\om:\cJ_\om(T)\to \cJ_{\ta(\om)}(T)$ is $\lm_\om\absval{T_\om'(z)}^t$--Bconformal with respect to the pair of measures $\nu_\om$ and $\nu_{\ta(\om)}$. 
We then also call each map $T_\om:\cJ_\om(T)\to \cJ_{\ta(\om)}(T)$, $\om\in\Omega$, $t$--\textit{Bconformal} with respect to the pair of measures $\nu_\om$ and $\nu_{\th(\om)}$. We then have
\begin{equation}\label{120180405}
\nu_\om\(T_{\om,\Ga}^{-n}(A)\)
=\lm_{\om}^{-n}\int_A\absval{(T_{\om,\Ga}^{-n})'}^t d\nu_{\th^n(\om)}
\end{equation}
for every $\om\in\Om$, every integer $n\ge 0$, every $\Ga\in\cG^n$ and every Borel set $A\sbt [0,1]$.

A random measure $\nu\in M_m^1(\cJ(T))$ is called \textit{$t$--conformal} if it is both $t$--Bconformal and $t$--Bconformal. Each map $T_\om:\cJ_\om(T)\to \cJ_{\ta(\om)}(T)$, $\om\in\Omega$, is then called $t$--conformal with respect to the pair of measures $\nu_\om$ and $\nu_{\th(\om)}$. 

We shall now prove the following basic result.
\begin{theorem}\label{t1mis.cm7}
	For every $t\geq 0$, every measure $\nu^{(t)}$ which is a weak limit of the sequence $(\nu_{t,k})_{k=1}^\infty$, is $t$--Fconformal for the map $T:\cJ(T)\to\cJ(T)$. If $\nu^{(t)}(\Om\times\Crit(\cG))=0$, then $\nu^{(t)}$ is $t$--conformal. In particular,
	\begin{align*}
		\tr_{t,\om}^*\nu_{\ta(\om)}^{(t)}=\lm_\om^{(t)}\nu_\om^{(t)}	
	\end{align*}  
	for every $\om\in\Om$, where $\lm^{(t)}:\Om\to(0,\infty)$ is some measurable function; the numbers $\lm_\om^{(t)}$, $\om\in\Om$, are respectively called the generalized eigenvalues of the dual operators $\tr_{t,\om}^*$, $\om\in\Om$. In addition, there exists a measurable function $\lm^{(t)}:\Om\to(0,\infty)$ 
witnessing $t$--Fconformality of $\nu^{(t)}$ for which there are constants $0<\bt_1\leq \bt_2<\infty$ such that 
	\begin{align*}
		\bt_1\leq\lm_\om^{(t)}\leq\bt_2
	\end{align*}
	for $m$--a.e. $\om\in\Om$. 
\end{theorem}
\begin{proof}
	We shall suppress the sub and superscripts $``t"$ and write $\nu$ for $\nu^{(t)}$, $\nu_k$ for $\nu_{t,k}$ and $\lm$ for $\lm^{(t)}$. Fix $k\geq 2$ and $\om\in\Om$.
	
	Given a function $g\in C_b(\cJ(T))$ and $k\geq 2$, we begin with estimating the following quantity:
	\begin{align*}
		\Dl_k^{(1)}g(\om)
:&=\absval{\left(\tr_{k,\om}^*-\tr_\om^*\right)(\nu_{k,\ta(\om)})(g_\om)}
		=\absval{\nu_{k,\ta(\om)}\left(\tr_{k,\om}(g_\om\rvert_{\cJ_\om^{(k)}(T)})-\tr_\om(g_\om)\right) } \\
		&=\absval{\int_{\cJ_\om^{(k)}(T)} \left(\tr_{k,\om}-\tr_\om\right)(g_\om)d\nu_{k,\ta(\om)}}\\
		&=\absval{\int_{I_{\ta(\om)}(k)} -\absval{T_\om'(T_{\om,0}^{-1}(x))}^{-t}g_\om(T_{\om,0}^{-1}(x)) d\nu_{k,\ta(\om)}(x)}\\
		&\leq\int_{I_{\ta(\om)}(k)}\absval{T_\om'(T_{\om,0}^{-1}(x))}^{-t}\norm{g_\om}_\infty d\nu_{k,\ta(\om)}\\
		&=\norm{g_\om}_\infty\int_{I_{\ta(\om)}(k)}\absval{T_\om'(T_{\om,0}^{-1}(x))}^{-t}d\nu_{k,\ta(\om)}\\
		&\comp\norm{g_\om}_\infty\nu_{k,\ta(\om)}(I_{\ta(\om)}(k)).
	\end{align*}
	Now applying \eqref{1mis.cm7.1}, we get 
	\begin{align}\label{1mis.cm8}
		\Dl_k^{(1)}g(\om)\leq\norm{g_\om}_\infty \kp^{-\frac{t\gm_{0,+}}{1+\gm_{0,+}}k}.
	\end{align}
	Therefore, 
$$
		\Dl_k^{(1)}(g)
:=\absval{\int_{\Om}\left(\tr_{k,\om}^*-\tr_{\om}^*\right)(\nu_{k,\ta(\om)})(g_\om)dm(\om) }
		\leq \int_{\Om}\Dl_k^{(1)}(g)(\om)dm(\om)\leq\norm{g}_\infty\kp^{-\frac{t\gm_{0,+}}{1+\gm_{0,+}}k}.
$$
	Hence, there exists $k_0\geq 3$ such that for all $k\geq k_0$ and all $g\in C_b(\cJ(T))$, we have 
	\begin{align}\label{1mis.cm8.1}
		\Dl_k^{(1)}(g)\leq \frac{\ep}{4}\norm{g}_\infty.
	\end{align}
	Now given an integer $\ell\geq 3$, we want to estimate the quantity
	\begin{align*}
		\int_{H_\om(\ell)}\tr_\om\ind d\eta_\om,
	\end{align*}
	where 
	\begin{align*}
		H_\om(\ell):=V_\om(\ell)\cup T_{\om,1}^{-1}(V_{\ta(\om)}(\ell))
	\end{align*}
	and $\eta$ is either $\nu$ or $\nu_k$. Indeed, using Proposition \ref{p1mis.cm7.1.2} and Lemma \ref{l1mis.cm7.1.1}, we proceed in exactly the same way as in formula \eqref{2mis.cm7.2}, to get 
	\begin{align*}
		\int_{H_\om(\ell)}\tr_\om\ind d\eta_\om\lesssim \eta_{\ta(\om)}(W_{\ta(\om)}^0(\ell))+\eta_{\ta^2(\om)}(W_{\ta^2(\om)}^1(\ell)),
	\end{align*}
	where 
	\begin{align*}
		W_{\ta(\om)}^0(j):=\union_{\Dl\in\cG_C^{(0)}}T_{\ta(\om),\Dl}^{-1}(V_{\ta^2(\om)}(\ell-1))
	\end{align*}
	and 
	\begin{align*}
		W_{\ta^2(\om)}^1(j):=\union_{\Dl\in\cG_C^{(1)}}T_{\ta^2(\om),\Dl}^{-1}(V_{\ta^3(\om)}(\ell-1)),
	\end{align*}
with  
	$$
	V_{\ta^2(\om)}^*(\ell-1):=V_{\ta^2(\om)}(\ell-1)\bs\set{0}.
	$$
Therefore, invoking \eqref{2mis.po} and \eqref{3mis.po}, we further get 
	\begin{align*}
		\int_{H_\om(\ell)}\tr_\om\ind d\eta_\om
		&\leq \eta_{\ta(\om)}\left(\union_{\Dl\in\cG_C^{(0)}}\Dl\lr(\absval{\lr(T_{\ta^2(\om)}^{\ell-1} \rl)'(0)}^{-\frac{1}{1+\gm_\Dl}}\rl) \right)+ \\ & \  \  \  \  \  \  \  \  \  \  \  \  \ +\eta_{\ta^2(\om)}\left(\union_{\Dl\in\cG_C^{(1)}}\Dl\lr(\absval{\lr(T_{\ta^3(\om)}^{\ell-1} \rl)'(0)}^{-\frac{1}{1+\gm_\Dl}}\rl) \right)\\
		&\leq \eta_{\ta(\om)}\left(\union_{\Dl\in\cG_C^{(0)}}\Dl\Big(\kp^{-\frac{\ell-1}{1+\gm_\Dl}}\Big)\right)+\eta_{\ta^2(\om)}\left(\union_{\Dl\in\cG_C^{(1)}}\Dl\Big(\kp^{-\frac{\ell-1}{1+\gm_\Dl}}\Big)\right),
	\end{align*}
	where
	\begin{align*}
		\Dl(r)=\Dl\cap\left(B\big(c_\Dl,r\big)\bs\set{c_\Dl}\right).
	\end{align*}
	Therefore, 
	\begin{align*}
		\Sg_\ell^{(2)}(g)
		:&= \int_{\Om}\int_{H_\om(\ell)}\absval{\tr_\om g_\om}d\eta_\om dm(\om)
			\leq \int_\Om\int_{H_\om(\ell)}\tr(\absval{g_\om})d\eta_\om dm(\om)\\
			&\leq \int_{\Om}\norm{g_\om}_\infty\int_{H_\om(\ell)}\tr_\om(\ind)d\eta_\om dm(\om)\\
			&\lesssim\norm{g_\om}_\infty\int_\Om\left(\eta_{\ta(\om)}\!\!
			\left(\union_{\Dl\in\cG_C^{(0)}}\!\!\!\Dl\Big(\kp^{-\frac{\ell-1}{1+\gm_\Dl}}\Big)\right)+\eta_{\ta^2(\om)}\!\!\left(\union_{\Dl\in\cG_C^{(1)}}\!\!\!\Dl\Big(\kp^{-\frac{\ell-1}{1+\gm_\Dl}}\Big)\right)\right) dm(\om)\\
			&=\norm{g_\om}_\infty \eta\left(\Om\times \left(\union_{\Dl\in\cG_C}\!\!\!\Dl\Big(\kp^{-\frac{\ell-1}{1+\gm_\Dl}}\Big)\right)\right).
	\end{align*}
	Distinguish now between $\nu$ and $\nu_k$ writing respectively $\Sg_\ell^{(2)}(\nu,g)$ and $\Sg_\ell^{(2)}(\nu_k,g)$. Fix $\ep>0$. Since $\nu$ is a finite measure there then exists $\ell_0\geq 2$ such that for every $g\in C_b(\cJ(T))$ and every $\ell\geq\ell_0$,
	\begin{align}\label{6_2018_02_07}
		\Sg_\ell^{(2)}(\nu,g)\leq \frac{\ep}{4}\norm{g}_\infty.
	\end{align}  
Hence, 
\begin{align}\label{1mis.cm9}
\int_{\Om}\int_{\Int(H_\om(\ell))}\absval{\tr_\om g_\om}d\eta_\om dm(\om)
\le \Sg_\ell^{(2)}(\nu,g)\leq \frac{\ep}{4}\norm{g}_\infty.
	\end{align}  
	Assume now that 
	\begin{align*}
		\nu(\Om\times\Crit(\cG))=0.
\end{align*}  
	Then there exists $\ell_1\geq\ell_0$ and $k_1\geq k_0$ such that for all $k\geq k_1$ and all $\ell\geq \ell_1$, we have that   
	\begin{align*}
		\nu_k\left(\Om\times \left(\union_{\Dl\in\cG_C}\Dl\Big(\kp^{-\frac{\ell_1}{1+\gm_\Dl}}\Big)\right)\right)\leq\frac{\ep}{4}.
	\end{align*}
Consequently, 
	\begin{align}\label{5_2018_02_07}
		\Sg_\ell^{(2)}(\nu_k,g)\leq\frac{\ep}{4}\norm{g}_\infty.
	\end{align}
Assume now for a moment that
$$
g\ge 0
$$
everywhere on $\cJ(T)$. Then, combining \eqref{5_2018_02_07}, \eqref{1mis.cm9} and \eqref{1mis.cm8.1}, we get for all $k\geq k_1$, $\ell\geq \ell_1$, and $g\in C_b(\cJ(T))$ that
\begin{equation}
\begin{aligned}	
\Sg_k(g)
:&=\tr_k^*\nu_k(g)-\tr^*\nu(g)
=(\tr_k^*-\tr^*)\nu_k(g)+\tr^*(\nu_k-\nu)g \\
&\leq \absval{\nu_k(\tr_k g)-\nu_k(\tr g)}+\nu_k(\tr g)-\nu(\tr g)\\
&\leq \Dl_k^{(1)}(g)+\Sg_\ell^{(2)}(\nu_k,g)+\\ 
& \  \  \  \  \  \  \  \  \  \  \  \  \  \ +\int_\Om\left(\int_{(\Int(H_\om(\ell))^c}(\tr g)_{\ta(\om)}d \nu_{k,\ta(\om)}-\int_{(\Int(H_\om(\ell))^c}(\tr g)_{\ta(\om)}d\nu_{\ta(\om)}\right) dm(\om)\\
&\leq \frac{1}{2}\ep\norm{g}_\infty+\int_\Om\left(\int_{(\Int(H_\om(\ell))^c}(\tr g)_{\ta(\om)}d \nu_{k,\ta(\om)}-\int_{H_\om^c(\ell)}(\tr g)_{\ta(\om)}d\nu_{\ta(\om)}\right) dm(\om)\label{1mis.cm10}
\end{aligned}  
\end{equation}
Then, since $(\Int(H_\om(\ell))^c_{\om\in\Om}$ is a closed random set and the function $\tr g\rvert_{(\Int(H_\om(\ell))^c}$ is continuous, and since $\nu_k$ converges to $\nu$ weakly, it follows from Portmanteau's Theorem (see \cite{Crauel}, Theorem 3.17) that 
	\begin{align}\label{2mis.cm10}
		\int_\Om\left(\int_{(\Int(H_\om(\ell))^c}(\tr g)_{\ta(\om)}d \nu_{k,\ta(\om)}-\int_{(\Int(H_\om(\ell))^c}(\tr g)_{\ta(\om)}d\nu_{\ta(\om)}\right) dm(\om)\leq \frac{\ep}{2}\norm{g}_\infty	
	\end{align} 
	for all $k\geq k_1$, large enough, say that $k\geq k_2\geq k_1$. Inserting this into \eqref{1mis.cm10}, we get that 
	\begin{align}\label{3mis.cm10}
		\Sg_k(g)\leq \ep\norm{g}_\infty.		
	\end{align}   
	Likewise, with the same $k$ and $\ell$, by combining \eqref{6_2018_02_07},
	 \eqref{1mis.cm9} and \eqref{1mis.cm8.1}, we get that
\begin{equation}
\begin{aligned}\label{1_2018_021_06}
		\Sg_k(g)
		&\geq -\absval{\nu_k(\tr_k g)-\nu_k(\tr g)}+\nu_k(\tr g)-\nu(\tr g)\\
		&\geq -\Dl_k^{(1)}(g)-\Sg_\ell^{(2)}(\nu, g)+ \\
& \  \  \  \  \  \  \  \  \  \  \  \  \ +\int_{\Om}\left(\int_{H_\om^c(\ell)}(\tr g\)_{\ta(\om)}d\nu_{k,\ta(\om)}-\int_{H_\om^c(\ell)}(\tr g)_{\ta(\om)}d\nu_{\ta(\om)}\right)dm(\om) \\
&\geq -\frac{1}{2}\ep\norm{g}_\infty +\int_{\Om}\left(\int_{H_\om^c(\ell)}(\tr g\)_{\ta(\om)}d\nu_{k,\ta(\om)}-\int_{H_\om^c(\ell)}(\tr g)_{\ta(\om)}d\nu_{\ta(\om)}\right)dm(\om).
\end{aligned}
\end{equation}
But since $(H_\om^c(\ell))_{\om\in\Om}$ is a random open set, by the same argument as leading to \eqref{2mis.cm10}, we conclude that 
	\begin{align*}
\int_{\Om}\left(\int_{H_\om^c(\ell)}(\tr g)_{\ta(\om)}d\nu_{k,\ta(\om)}-\int_{H_\om^c(\ell)}(\tr g)_{\ta(\om)}d\nu_{\ta(\om)}\right)dm(\om)\geq -\frac{\ep}{2}\norm{g}_\infty
	\end{align*}
for all $k\geq k_2$ large enough, say $k\geq k_3\geq k_2$. Inserting this into \eqref{1_2018_021_06}, we get that 
	\begin{align*}
	\Sg_k(g)\geq -\ep\norm{g}_\infty.		
	\end{align*} 
	In conjunction with \eqref{3mis.cm10}, this gives that 
	\begin{align}\label{1mis.cm11}
		\absval{\Sg_k(g)}\leq\ep\norm{g}_\infty
	\end{align}
	for all non-negative functions $g\in C_b(\cJ(T))$ and all $k\geq k_3$. Now if $g\in C_b(\cJ(T))$ is arbitrary, we represent  it as $g=g_+-g_-$, where $g_+=\max(g,0)$ and $g_-=\max(-g,0)$. By applying \eqref{1mis.cm11}, we then obtain
	\begin{align}\label{2_2018_021_06}
\absval{\Sg_k(g)}
=\absval{\Sg_k(g_+)-\Sg_k(g_-)}
\le\absval{\Sg_k(g_+)}+\absval{\Sg_k(g_-)}
\le \ep\norm{g_+}_\infty+\ep\norm{g_-}_\infty
\leq 2\ep\norm{g}_\infty.
	\end{align}

Now, having Lemma \ref{l1mis.cm7.2}, Lemma 2.9 from \cite{RoyUrb2011} tells us that there exists a measurable function $\Lm:\Om\to(0,\infty)$ such that the sequence $(\lm_k\nu_k)_{k=1}^\infty$ converges weakly to $\lm\nu$ and 
	\begin{align}\label{2mis.cm11}
		\bt_1\leq\lm(\om)\leq\bt_2
	\end{align}
	for $m$--a.e. $\om\in\Om$. But since also $\lm_k\nu_k=\tr_k^*\nu_k$, by letting $k\to\infty$, formula \eqref{2_2018_021_06} gives us that 
	$$
	\absval{\tr^*\nu(g)-\lm\nu(g)}\leq2\ep\norm{g}_\infty.
	$$
Since $\ep>0$ was arbitrary, we thus conclude that $\tr^*\nu(g)=\lm\nu(g)$ for all $g\in C_b(\cJ(T))$. Hence 
 $$
 \tr^*\nu=\lm\nu,
 $$
which, along with Proposition~\ref{p1mis.cmi5} and Remark~\ref{c1mis.cmi6.1}, shows that the measure $\nu$ is $t$--Bconformal. Together with \eqref{2mis.cm11}, this completes the proof in the case when $\nu(\Om\times\Crit(\cG))=0$; $t$--Fconformality is still to be proved without this hypothesis. Indeed, what follows from our considerations without it, is that 
	\begin{align*}
		\tr^*\nu(g)=\lm\nu(g)
	\end{align*}   
	for every $g\in C_b(\cJ(T))$ such that 
	\begin{align*}
		g|_{\Om\times B(\Crit(\cG),\gm)}\equiv 0
	\end{align*}
	for some $\gm>0$. Fiberwise, this means that 
	\begin{align}\label{1mis.cm12}
		\tr_\om^*\nu_{\ta(\om)}(g_\om)=\lm_\om\nu_\om(g_\om)
	\end{align}
	for every $g_\om\in C(\cJ_\om(T))$ with $g_\om\rvert_{B(\Crit(\cG),\gm)}\equiv 0$ for some $\gm>0$. By the standard approximation procedure, we thus conclude that \eqref{1mis.cm12} continues to hold for all functions $g_\om\in L^1(\nu_\om)$ with  $g_\om\rvert_{B(\Crit(\cG),\gm)}\equiv 0$ for some $\gm>0$. Therefore, the calculation of \eqref{1mis.cmi3} can be performed for all special sets $A\sub\cJ_\om(T)$, disjoint from some open neighborhood of $\Crit(\cG)$, to give
	\begin{align}\label{2mis.cm12}
		\nu_{\ta(\om)}(T_\om(A))=\lm_\om\int_A\absval{T_\om'}^td\nu_\om.
	\end{align} 
	Since every special set disjoint from $\Crit(\cG)$ can be represented as a disjoint union of special sets disjoint from neighborhoods of $\Crit(\cG)$, formula \eqref{2mis.cm12} holds for all such sets, that is all special sets disjoint from $\Crit(\cG)$. But, by virtue of Proposition \ref{p1mis.cm7.1.2}, we have for every $c\in\Crit_0(\cG)$ that 
	\begin{align*}
		\nu_{\ta(\om)}(T_\om(c))
		=\nu_{\ta(\om)}(0)=0=\lm_\om\absval{T_\om'(0)}^t\nu_\om(c)
		=\lm_\om\int_c\absval{T_\om'}^td\nu_\om.
	\end{align*} 
The same is also true for $c\in\Crit_1(\cG)$. So, we have proved that the measure $\nu$ is $t$--Fconformal. The proof of Theorem \ref{t1mis.cm7} is complete. 
\end{proof}

\section{$t$--Fconformal, $t$--Bconformal, and $t$-conformal Measures}\label{FBCM}

This section is devoted to establishing and clarifying relationships between $t$--Fconformal, $t$--Bconformal and $t$--conformal random measures for a random critically finite map and, to describe the structure they form. We start with the following.
\begin{proposition}\label{p1mis.cmii1}
	Fix $t\geq 0$. Then 
	\begin{enumerate}[(a)]
		\item every $t$--Bconformal measure $\nu$ is $t$--conformal;
		\item a $t$--Fconformal measure $\nu$ is $t$--conformal if and only if $\nu(\Omega\times\Crit(\cG))=0$ (assuming $t>0$ for the ``$\Rightarrow$" implication).
	\end{enumerate}
\end{proposition} 
\begin{proof}
	Since $\lm_\om\absval{T_\om'(z)}$ never equals $+\infty$, $(a)$ follows immediately from Corollary \ref{c1mis.cmi4}. The part $``\Leftarrow"$ of Proposition \ref{p1mis.cmii1} follows immediately from Proposition \ref{p1mis.cmi3} (b). 
	
Proving the implication $``\Rightarrow"$ with $t>0$, we assume toward a contradiction that $\nu(\Omega\times\Crit(\cG))>0$. Then there exists $c\in\Crit(\cG)$ and $Z\sub\Omega$, a measurable set, such that $m(Z)>0$ and 
	\begin{align*}
		\nu_\om(c)>0
	\end{align*}
for all $\om\in Z$. Note that $t$--Fconformality yields
\begin{equation}\label{520180405}
\nu_\om(T_\om(\Crit(\cG)))=0
\end{equation}
for all $\om\in\Om$. Let $g:\cJ(T)\to[0,1]$ be defined by the formula
	\begin{align*}
		g_\om(z)=\ind_c(z).
	\end{align*}
Then, employing \eqref{520180405}, we have for every $\om\in\Omega$ that
	\begin{align*}
		\nu_{\ta(\om)}(\hat{\tr}_{t,\om}g_\om)
		&=\lm_\om^{-1}\int_{\cJ_{\ta(\om)}(T)}\tr_{t,\om}g_\om d\nu_{\ta(\om)}
		=\lm_{\om}^{-1}\int_{\cJ_{\ta(\om)}(T)}\sum_{x\in T^{-1}_\om(y)}\ind_c(x)\absval{T_\om'(x)}^{-t}d\nu_{\ta(\om)}(y)\\
		&=\lm_{\om}^{-1}\absval{T_\om'(c)}^{-t}\nu_{\ta(\om)}(T_\om(c))
		=\infty\cdot 0=0,
	\end{align*}
	and, for all $\om\in Z$,
	\begin{align*}
		\nu_\om(g_\om)=\int_{\cJ_\om(T)}\ind_c(z)d\nu_\om(z)=\nu_\om(c)>0.
	\end{align*}
	This contradiction finishes the proof. 
\end{proof}

\noindent As an immediate consequence of this proposition, we get the following.

\begin{corollary}\label{c520180405}
If $t\ge 0$, then a random measure $\nu$ on $\cJ(T)$ is $t$--conformal if and only if it is $t$--Bconformal. If this holds, and, in addition, $t>0$, then also
$$
\nu(\Omega\times\Crit(\cG))=0.
$$
\end{corollary}
For every $\om\in\Omega$ let 
\begin{align*}
	\Crit_-(\cG,\om):=\union_{n=0}^\infty T_\om^{-n}(\Crit(\cG))\spand \Crit_-(\cG):=\union_{\om\in\Omega}\set{\om}\times\Crit_-(\cG,\om).
\end{align*}

We shall prove the following elementary, but quite important result.

\begin{lemma}\label{l1cmii3.2}
	Fix $t>1$. If $\nu=(\nu_\om)_{\om\in\Om}$ is a $t$--conformal measure then there exists $\eta_*>0$ such that
	\begin{align*}
\inf \set{\nu_\om([\eta_*,1-\eta_*]):\om\in\Om}
\ge Q_\cG(t):=\frac{1}{2}(\kp A^{-1})^{2t}>0,
	\end{align*}
where, we recall, both $\kappa>1$ and $A>1$ come from conditions  \ref{(M1)}--\ref{(M7)} in Section~\ref{SRIMM}.
\end{lemma}
\begin{proof}
Fix $\om\in\Om$. By \ref{(M5b)} and \ref{(M5c)}, for every $x\in[0,s]$, we have that
	\begin{align}\label{1cmii3.2}
		T_{\om}([0,x])\sub[0,Ax]
	\end{align}
and 
	\begin{align}\label{2cmii3.2}
		\nu_{\ta(\om)}\(T_{\om}([0,x])\)\geq \kp^t\lm_\om\nu_\om([0,x]).	
	\end{align}
	By the same token, for every $x\in[0,s/A]$, we have that
	\begin{align}\label{3cmii3.2}
		T_\om^2([0,x])\sub[0,A^2x]
	\end{align}
	and 
	\begin{align}\label{4cmii3.2}
		\nu_{\ta^2(\om)}\(T_{\om}^2([0,x])\)\geq \kp^{2t}\lm_\om\lm_{\ta(\om)}\nu_\om([0,x]).
	\end{align}
	Also, by the same token, for every $y\in[0,\kp s]$,
	\begin{align}\label{6cmii3.2}
		T_{\om,1}^{-1}([0,y])\sub [1-\kp^{-1}y,1]
	\end{align}
	and 
	\begin{align}\label{7cmii3.2}
		\nu_\om(T_{\om,1}^{-1}([0,y]))\geq A^{-t}\lm_{\om}^{-1}\nu_{\ta(\om)}([0,y]).
	\end{align}
	Moreover, this yields 
	\begin{align}\label{1cmii3.2.1}
		\nu_\om(T_{\om,1}^{-2}([0,y]))\geq A^{-2t}\lm_\om^{-2}\nu_{\ta^2(\om)}([0,y]).
	\end{align}
Furthermore, for every $z\in[1-s,1]$,
	\begin{align}\label{9cmii3.2}
T_{\om}([z,1])\sub [0,A(1-z)]
	\end{align}
	and 
	\begin{align}\label{10cmii3.2}
		\nu_{\ta(\om)}(T_{\om,1}[z,1])\geq \kp^t\lm_\om\nu_\om([z,1]).
	\end{align}
	Finally, for every $\xi\in(1-A^{-1},1)$
	\begin{align}\label{12cmii3.2}
		T_{\om,1}^{-1}([\xi,1])\sub\left[1-\absval{\Dl_1},1-\absval{\Dl_1}+\max\set{1-\frac{\ep}{\kp}, (A(1-\xi))^{\frac{1}{1+\gm_+}}}\right]
	\end{align}
	and 
	\begin{align}\label{13cmii3.2}
		\nu_\om(T_{\om,1}^{-1}([\xi,1]))\geq A^{-t}\lm_\om^{-1}\nu_{\ta(\om)}([\xi,1]).
	\end{align}
	Now fix $\xi_1\in\left(\max\set{1-(\absval{\Dl_1}/2),1-A^{-1}},1\right)$ so close to 1 that
	\begin{align*}
		\max\set{(1-\xi_1)/\kp, (A(1-\xi_1))^{\frac{1}{1+\gm_+}}}\leq\absval{\Dl_1}/2.
	\end{align*} 
	It then follows from \eqref{12cmii3.2} that for every $\xi\in[\xi_1,1]$
	\begin{align}\label{1cmii3.3.1}
		T_{\om,1}^{-1}([\xi,1])\sub [1-\absval{\Dl_1},1-(\absval{\Dl_1}/2)].
	\end{align}
	Fix further $\eta_1\in(\max\set{\xi_1,1-s},1)$ so close to 1 that 
	\begin{align*}
		A(1-\eta_1)&\leq s,\\
		A^2(1-\eta_1)&\leq \kp s,
	\end{align*}
	and 
	\begin{align*}
		1-\kp^{-1}A^2(1-\eta_1)\geq \xi_1.
	\end{align*}
	Then by these choices, by \eqref{9cmii3.2}, \eqref{1cmii3.2}, \eqref{6cmii3.2}, \eqref{1cmii3.3.1}, \eqref{10cmii3.2}, \eqref{2cmii3.2}, and \eqref{1cmii3.2.1}, we get that 
	\begin{align*}
		T_{\om}([\eta_1,1])&\sub [0,A(1-\eta_1)]\sub [0,s],\\
		T_{\om}^2([\eta_1,1])&\sub [0,A^2(1-\eta_1)]\sub [0,\kp s],\\
		T_{\om,1}^{-1}(T_\om^2([\eta_1,1]))&\sub [1-\kp^{-1}A^2(1-\eta_1),1]\sub [\xi_1,1],\\
		T_{\om,1}^{-2}(T_\om^2([\eta_1,1]))&\sub [1-\absval{\Dl_1},1-(\absval{\Dl_1}/2)]\sub [1-\absval{\Dl_1},\xi_1],
	\end{align*}
	and
	\begin{align*}
		\nu_\om(T_{\om,1}^{-2}(T_\om^2([\eta_1,1])))&\geq A^{-2t}\lm_\om^{-2}\nu_{\ta^2(\om)}(T_\om^2([\eta_1,1]))\\
		&\geq A^{-2t}\lm_\om^{-2}\kp^t\lm_{\ta(\om)}\nu_{\ta(\om)}(T_\om([\eta_1,1]))\geq A^{-2t}\kp^t\lm_\om^{-2}\lm_{\ta(\om)}\kp^t\lm_\om\nu_\om([\eta_1,1])\\
		&=(\kp A^{-1})^{2t}\lm_\om^{-2}\lm_\om^2\nu_\om([\eta_1,1])=(\kp A^{-1})^{2t}\nu_\om([\eta_1,1]).
	\end{align*}
	Therefore, 
	\begin{align}\label{1cmii3.3}
		\nu_\om([1-\absval{\Dl_1},\eta_1])\geq (\kp A^{-1})^{2t}\nu_\om([\eta_1,1]).
	\end{align}
	Now, fix $\eta_0\in (0,s)$ so small that 
	\begin{align*}
		A^2\eta_0\leq \kp s,
	\end{align*}
	and 
	\begin{align*}
		1-\kp^{-1}A^2\eta_0\geq \xi_1.
	\end{align*}
	Then, by these choices, \eqref{3cmii3.2}, \eqref{6cmii3.2}, \eqref{1cmii3.3.1}, \eqref{4cmii3.2}, and \eqref{1cmii3.2.1}, we get that 
	\begin{align*}
		T_\om^2([0,\eta_0])&\sub [0,A^2\eta_0]\sub[0,\kp s],\\
		T_{\om,1}^{-1}(T_\om^2([0.\eta_0]))&\sub [1-A^{2}\eta_0,1]\sub[\xi_1,1],\\
		T_{\om,1}^{-2}(T_\om^2([0,\eta_0]))&\sub[1-\absval{\Dl_1},1-(\absval{\Dl_1}/2)]\sub[1-\absval{\Dl_1},\eta_1],
	\end{align*} 
	and
	\begin{align*}
		\nu_\om(T_{\om,1}^{-2}(T_\om^2([0,\eta_0])))&\geq A^{-2t}\lm_\om^{-2}\nu_{\ta^2(\om)}(T_\om^2([0,\eta_0]))\\
		&\geq A^{-2t}\lm_\om^{-2}\kp^{2t}\lm_\om^2\nu_\om([0,\eta_0])=(\kp A^{-1})^{2t}\nu_\om([0,\eta_0]).
	\end{align*}
	Therefore,
	\begin{align*}
		\nu_\om([1-\absval{\Dl_1},\eta_1])\geq (\kp A^{-1})^{2t}\nu_\om([0,\eta_0]).
	\end{align*}
	Along with \eqref{1cmii3.3} and since $\eta_0<s\leq 1-\absval{\Dl_1}$, this gives
	\begin{align}\label{1cmii3.4}
		\nu_\om([\eta_0,\eta_1])\geq \nu_\om([1-\absval{\Dl_1},\eta_1])\geq (\kp A^{-1})^{2t}\left(\nu_\om([0,\eta_0])+\nu_\om([\eta_1,1])\right).
	\end{align}
	Now if $\nu_\om([\eta_0,\eta_1])\leq 1/2$, then $\nu_\om([0,\eta_0])+\nu_\om([\eta_1,1])\geq 1/2$, hence \eqref{1cmii3.4} yields
	\begin{align*}
		\nu_\om([\eta_0,\eta_1])\geq \frac{1}{2}(\kp A^{-1})^{2t}.
	\end{align*}
	Since $\kp\leq A$, this implies, in either case, that 
	\begin{align*}
		\nu_\om([\eta_0,\eta_1])\geq \frac{1}{2}(\kp A^{-1})^{2t}.
	\end{align*}
	Setting $\eta_*:=\min\set{\eta_0,1-\eta_1}$, we thus get that 
	\begin{align*}
		\nu_\om([\eta_*,1-\eta_*])\geq \frac{1}{2}(\kp A^{-1})^{2t}.
	\end{align*}
	The proof of Lemma \ref{l1cmii3.2} is complete.
	\end{proof}

We shall prove the following.
\begin{lemma}\label{l1mis.cmii2}
	Let $t>0$. Let $\nu'$ and $\nu''$ be two $t$--conformal measures for $T:\cJ(T)\to\cJ(T)$. Then 
	\begin{enumerate}[(a)]
		\item $\nu'(\Crit_-^c(\cG))=\nu''(\Crit_-^c(\cG))=1$,
		\item $\nu'\comp\nu''$ and, equivalently, $\nu_\om'\comp\nu_\om''$ for $m$--a.e. $\om\in\Omega$,
		\item for $m$--a.e. $\om\in\Omega$ there exists an unbounded sequence $\(n_\om(k)\)_{k=1}^\infty$ such that the limit below exists and
\begin{align*}
0<\lim_{k\to\infty}\frac{(\lm_{t,\om}')^{n_\om(k)}}{(\lm_{t,\om}'')^{n_\om(k)}}
<+\infty,
\end{align*}	
\end{enumerate}
	where $\lm_t':\Om\to (0,+\infty)$ and $\lm_t'':\Om\to (0,+\infty)$ are measurable functions witnessing $t$--conformality respectively of $\nu'$ and $\nu''$.
\end{lemma}
\begin{proof}
Since $t>0$, item $(a)$ immediately follows from Corollary~\ref{c520180405} and $t$--conformality of $\nu'$ and $\nu''$. By $t$-conformality of $\nu'$ and $\nu''$, we also have that 
	\begin{align*}
\nu_\om'(\{0\})
=\lm_{\ta^{-1}(\om)}\absval{T_{\ta^{-1}(\om)}'(c)}^t\nu_{\ta^{-1}(\om)}(\{c\})
=\lm_{\ta^{-1}(\om)}\cdot 0\cdot \nu_{\ta^{-1}(\om)}(\{c\})=0,
	\end{align*}
	where $c$ is any critical point in $\Crit_0(\cG)$. By the same token, $\nu_\om'(\{1\})=0$; so
	\begin{align*}
		\nu_\om'(\set{0,1})=\nu_\om''(\set{0,1})=0
	\end{align*}
	for $m$-a.e. $\om\in\Omega$. Employing $t$--conformality again, we get that
	\begin{align}\label{1mis.cmii3}
		\nu_\om'(\set{0,1}^-_*(\om))=\nu_\om''(\set{0,1}^-_*(\om))=0	
	\end{align}
	for $m$-a.e. $\om\in\Omega$, where
	\begin{align*}
		\set{0,1}^-_*(\om)=\Crit_-^c(\cG,\om)\cap\union_{n=0}^\infty T_\om^{-n}(\set{0,1}).
	\end{align*}
	Let $\Omega_0$ be the corresponding set of measure $m$ equal to 1. Fix $\om\in\Omega_0$ and
	\begin{align*}
		x\in \cJ_\om^*(T):=\cJ_\om(T)\bs\left(\Crit_-(\cG,\om)\cup\union_{n=0}^\infty T_\om^{-n}(\set{0,1})\right)=\cJ_\om(T)\bs\left(\Crit_-(\cG,\om)\cup\set{0,1}^-_*(\om) \right).
	\end{align*}
	Since $T_\om(1)=0$ and since 0 is a uniformly expanding fixed point, for all $\om\in\Omega$ there must exist $\eta\in(0,\eta_*]$ such that the set 
	\begin{align*}
		N_\om(x):=\set{n\geq 0: T_\om^n(x)\in (\eta,1-\eta)}
	\end{align*}
	is infinite. Fix $n\in N_\om(x)$ and let  
	\begin{align}\label{2mis135}
		H_{n}(\om,x):=T_{\om,x}^{-n}([\eta,1-\eta]).
	\end{align}
	It then follows from Proposition~\ref{p120180630} and $t$--conformality of both measures $\nu'$ and $\nu''$ that 
	\begin{align*}
		\nu_\om''(H_n(\om,x))
		\leq K^t_{\eta}(\lm_\om'')^{-n}\absval{(T_\om^n)'(x)}^{-t}\nu_{\ta^n(\om)}''([\eta,1-\eta])\leq K^t_{\eta}(\lm_\om'')^{-n}\absval{(T_\om^n)'(x)}^{-t}
	\end{align*} 
and 
	\begin{align}\label{?:11}
		\nu_\om'(H_n(\om,x))
		&\geq K^{-t}_\eta(\lm_\om')^{-n}\absval{(T_\om^n)'(x)}^{-t}\nu_{\ta^n(\om)}'([\eta,1-\eta])\\
		&\geq Q_\cG(t)K_\eta^{-t}(\lm_\om')^{-n}\absval{(T_\om^n)'(x)}^{-t}.
	\end{align}
	Moreover,
	\begin{align*}
		\ol{B}\left(x,K_\eta^{-1}\absval{(T_\om^n)'(x)}^{-1}\eta\right)\sub H_n(\om,x)\sub\ol{B}\left(x,K_\eta\absval{(T_\om^n)'(x)}^{-1}\eta\right).
	\end{align*}
	The first two give 
	\begin{align}\label{2mis137}
		\nu_\om''(H_n(\om,x))\leq K_\eta^{2t}Q_\cG^{-1}(t)\frac{(\lm_\om')^n}{(\lm_\om'')^n}\nu_\om'(H_n(\om,x)),
	\end{align}
	while the third one can be rewritten in the form
	\begin{align}\label{1mis137}
		\ol{B}\left(x,r_n(\om,x)\right)\sub H_n(\om,x)\sub\ol{B}\left(x,Mr_n(\om,x)\right)
	\end{align}
	where $r_n(\om,x):=	K_\eta^{-1}\absval{(T_\om^n)'(x)}^{-1}(1-2\eta)$ and $M:=K_\eta^2$.
	
	Keep $\om\in\Om_0$ and fix $E\sub\cJ^*_\om(T)$, an arbitrary compact set, as well as $\dl>0$. In view of \eqref{2mis135} and Corollary \ref{c1mis81} there exists $n_\om(\delta)\in N_\om(x)$ such that for every $x\in E$ we have that 
\begin{align}\label{3mis137}
\diam(H_{n_\om(\dl)}(\om,x))<\dl 
\end{align}
and 
\begin{align}\label{3mis137B}
\nu_\om'\left(\union_{x\in E}\ov B(x,Mr_{n_\om}(\om,x))\right)<\dl+\nu_\om'(E) \
{\rm and } \ \nu_\om''\left(\union_{x\in E}\ov B(x,Mr_{n_\om}(\om,x))\right)<\dl+\nu_\om''(E).
\end{align}
	Since all the sets $H_{n_\om(\delta)}(\om,x)$ are closed and convex and since \eqref{1mis137} holds, all the hypotheses of Morse's Theorem from \cite{Morse} (see also p. 6 of \cite{Guzman}) are satisfied, and it gives us a natural number $\ell\geq 1$ depending only on $M$, and a countable set $D\sub E$ such that the family $\set{H_{n_\om(\delta)}(\om,x)}_{x\in D}$ is a cover of $E$ (with diameters $\leq \dl$) and $D$ can be split into $\ell$ disjoint subsets $D_1,D_2,\dots,D_\ell$ such that for each $1\leq k\leq \ell$ the family $\set{H_{n_\om(\delta)}(\om,x)}_{x\in D_k}$ consists of mutually disjoint sets. Using \eqref{1mis137}, \eqref{2mis137}, and \eqref{3mis137B} we then get 
	\begin{align*}
		\nu_\om'\left(\union_{x\in D}H_{n_\om(\delta)}(\om,x)\right)\leq\nu_\om'\left(\union_{x\in D}\ov B(x,Mr_{n_\om(\delta)}(\om,x))\right)<\dl+\nu_\om'(E)
	\end{align*}
	and 
	\begin{align}
		\nu_\om''(E)
		&\leq \nu_\om''\left(\union_{x\in D}H_{n_\om(\delta)}(\om,x)\right)
		\leq\sum_{k=1}^\ell\sum_{x\in D_k}\nu_\om''(H_{n_\om(\delta)}(\om,x))\nonumber\\
		&\leq K_\eta^{2t}Q_\cG^{-1}(t)\frac{(\lm_\om')^{n_\om(\delta)}}{(\lm_\om'')^{n_\om(\delta)}}\sum_{k=1}^\ell\sum_{x\in D_k}\nu_\om'(H_{n_\om(\delta)}(\om,x))\nonumber\\
		&=K_\eta^{2t}Q_\cG^{-1}(t)\frac{(\lm_\om')^{n_\om(\delta)}}{(\lm_\om'')^{n_\om(\delta)}}\sum_{k=1}^\ell \nu_\om'\left(\union_{x\in D_k}H_{n_\om(\delta)}(\om,x)\right)\nonumber\\
		&\leq \ell K_\eta^{2t}Q_\cG^{-1}(t)\frac{(\lm_\om')^{n_\om(\delta)}}{(\lm_\om'')^{n_\om(\delta)}}\nu_\om'\left(\union_{x\in D}H_{n_\om(\delta)}(\om,x)\right)\nonumber\\
		&\leq \ell K_\eta^{2t}Q_\cG^{-1}(t)\frac{(\lm_\om')^{n_\om(\delta)}}{(\lm_\om'')^{n_\om(\delta)}}(\nu_\om'(E)+\dl).\label{3mis139}
	\end{align}	
If 
$$
\liminf_{k\to\infty}\frac{(\lm_{t,\om}')^{n_\om(1/k)}}{(\lm_{t,\om}'')^{n_\om(1/k)}}=0
$$
then considering a subsequence of the sequence $\(n_\om(1/k)\)_{k=1}^\infty$ witnessing the above lower limit to be $0$, we would get that $\nu_\om''(E)=0$. Since $\nu_\om''$ is a regular measure, this would yield $\nu_\om(\cJ_\om^*(T))=1$, contrary to $(a)$ and \eqref{1mis.cmii3}. Thus 
	\begin{align*}
\liminf_{k\to\infty}\frac{(\lm_{t,\om}')^{n_\om(1/k)}}{(\lm_{t,\om}'')^{n_\om(1/k)}}>0
	\end{align*}
	for all $\om\in\Omega_0$. Exchanging the roles of $\nu'$ and $\nu''$, we thus also have that 
\begin{align*}
\liminf_{k\to\infty}\frac{(\lm_{t,\om}'')^{n_\om(1/k)}}{(\lm_{t,\om}')^{n_\om(1/k)}}>0.
\end{align*}	
Equivalently,
	\begin{align}\label{2mis.cmii3}
		M(\om):=\limsup_{k\to\infty}\frac{(\lm_{t,\om}')^{n_\om(1/k)}}{(\lm_{t,\om}'')^{n_\om(1/k)}}<\infty.
	\end{align}
So, taking a subsequence of the sequence $\(n_\om(1/k)\)_{k=1}^\infty$ witnessing \eqref{2mis.cmii3}, we conclude both that item (c) holds, and, also with the help of \eqref{3mis139}, that 
	$$
	\nu_\om''(E)\leq K_\eta^{2t} Q_\cG^{-1}(t)M(\om)\ell\nu_\om'(E)
	$$
for every $\om\in\Omega_0$ and every compact set $E\sub \cJ_\om^*(T)$. 
Since $\nu_\om''$ is a regular measure this inequality extends to all Borel sets yielding $\nu_\om''<<\nu_\om'$. Thus $\nu_\om'\comp\nu_\om''$ for all $\om\in\Omega_0$, and $(b)$ is proved. The proof of Lemma \ref{l1mis.cmii2} is therefore complete. 
\end{proof}
The case of $t=0$ is even clearer and the picture is complete. 
\begin{proposition}\label{p1cmii4}
	If $t=0$, then there exists a unique 0--conformal measure for the random map $T:\cJ(T)\to\cJ(T)$. Denote it by $\nu^0$. We then have that 
	\begin{align*}
		\nu_\om^0(T_{\om,\Gm}^{-n}(\cJ_{\ta^n(\om)}(T)))=\lm_\om^{-n}=(\#\cG)^{-n}
	\end{align*} 
	for every $\om\in\Omega$, every $n\geq 0$, and every $\Gm\in \cG^n$. 
	In particular, $\lm_\om=\#\cG$ and, the measure $\nu^0$ on $[0,1]$ is atomless. Moreover, $\nu^0$ is $0$--conformal.
\end{proposition}
\begin{proof}
	Let $\nu$ be an arbitrary $0$--conformal measure. The formula $\nu_\om(T_{\om,\Gm}^{-n}(\cJ_{\ta^n(\om)}(T)))=\lm_\om^{-n}$ follows immediately from $0$--conformality. And further, from this:
	\begin{align*}
		1=\nu(\cJ_\om(T))=\sum_{\Gm\in\cG^n}\nu_\om(T_{\om,\Gm}^{-n}(\cJ_{\ta^n(\om)}(T)))=\sum_{\Gm\in\cG^n}\lm_\om^{-n}=(\#\cG)^n\lm_\om^{-n}.
	\end{align*}
	Whence $\lm_\om^n=(\#\cG)^n$. Uniqueness now follows since the finite collections of sets 
	\begin{align*}
		\left(\set{T_{\om,\Gm}^{-n}(\cJ_{\ta^n(\om)}(T)):\Gm\in\cG^n} \right)_{n=0}^\infty
	\end{align*} 
	form a sequence of finer and finer partitions of $\cJ_\om(T)$, whose diameters, due to Corollary \ref{c1mis81} converge to zero. The existence part of the proof is also straightforward; just declaring 
	\begin{align*}
		\nu_\om^0(T_{\om,\Gm}^{-n}(\cJ(T))):=(\#\cG)^{-n}
	\end{align*}
defines a Borel probability measure on $\cJ(T)$ due to Kolmogorov's Extension Theorem. As $0$--conformality of $\nu^0$ is obvious, the proof is now complete. 
\end{proof}
We denote by $CM(T)$ the set of those reals $t\geq 0$ for which a $t$--conformal measure exists.

\section{Expected Pressure 1}\label{sec: Topological Pressure}
Let $\NN_\infty=\NN\cup\set{\infty}=\set{1,2,\dots,\infty}$. For every $t\geq 0$ and every $k\in\NN$, let $\nu_k$ and $\lm_k$ be the objects resulting from Lemma \ref{l1mis.cm5}; see also formula \eqref{2mis.cm5} and Lemma~\ref{l1_2018_03_23}. Define 
\begin{align*}
	\cEP_k(t):=\int_{\Om}\log\lm_k\, dm,
\end{align*}
which is well defined because of Lemma~\ref{l1mis.cm7.2}. Define further 
\begin{align*}
	\ul{\cE}\P_\infty(t):=\liminf_{k\to\infty}\cEP_k(t)\spand \ol{\cE}\P_\infty(t):=\limsup_{k\to\infty}\cEP_k(t).
\end{align*}
If $\ul{\cE}\P_\infty(t)=\ol{\cE}\P_\infty(t)$, denote their common value by 
\begin{align*}
	\cEP_\infty(t).
\end{align*}
We also define 
\begin{align*}
	\chi_0:=\int_{\Om}\log\absval{T_\om'(0)}dm(\om)\geq\log\kp>0
\end{align*}
and call this quantity the \textit{Lyapunov exponent} of $T$ at 0.
\begin{definition}\label{d1tp3}
	We say that the number $t\geq 0$ is $A$(asymptotically)--admissible if 
	\begin{align*}
		\ol{\cE}\P_\infty(t)>-\frac{1}{1+\gm_+}\chi_0t.
	\end{align*}
	We denote the set of all $A$--admissible parameters $t$ by $AA(T)$.
\end{definition}

Frequently, for the sake of uniform exposition, we will denote the objects $\nu$ and $\lm$ (both depending on $t$) produced in Theorem~\ref{t1mis.cm7} respectively by $\nu_\infty$ and $\lm_\infty$. Exactly the same proof as that of Lemma~\ref{l1cmii3.2} gives the following.

\begin{lemma}\label{l1cmii3.2B}
There exists a constant $\eta_*\in (0,\min\set{1/2,1-(|\Dl_1|/2)})$ such that
	\begin{align*}
\nu_{k,\om}([\eta_*,1-\eta_*])\geq Q_\cG(t)=\frac12(\kappa A^{-1})^{2t}
	\end{align*}
for all $t\ge 0$, all $\om\in\Om$ and all $k\in\NN_\infty$.
\end{lemma}

\noindent As an immediate consequence of this lemma, we get the following.

\begin{lemma}\label{l1cmii3.2BC}
For every $\eta\in(0,\eta_*]$ we have that
	\begin{align*}
		\nu_{k,\om}(I_\om(1)\cap[0,1-\eta])\geq Q_\cG(t)=\frac12(\kappa A^{-1})^{2t}
	\end{align*}
for all $\om\in\Om$ and all $k\geq 1$.
\end{lemma}

Having $\eta\in(0,\eta_*]$, put 	
\begin{equation}\label{220180529}
	I_{\om}(j,\eta):=
	\begin{cases}
	I_\om(j) &\text{ if } j\geq 2,\\
	I_\om(1)\cap[0,1-\eta] &\text{ if } j=1
	\end{cases}
\end{equation}
and recall that 
$$
I_\om(1)\cap[0,1-\eta]=I_\om(1)\cap[\eta,1-\eta].
$$
We shall prove the following. 

\begin{lemma}\label{l1cmii3.2BD}
For all $\eta\in(0,\eta_*]$ and all $j\in\NN$ there exists $Q_j\geq 1$ such that for all $k\in\set{j,j+1,\dots,\infty}$, all $1\leq i\leq j$, and all $\om\in\Om$
	\begin{align*}
		\nu_{k,\om}(I_{\om}(i,\eta))\geq Q_j.
	\end{align*} 
\end{lemma}
\begin{proof}
It follows from Lemma~\ref{l1cmii3.2BC} that
\begin{equation}\label{120180428}
\nu_{k,\om}(I_\om(1))
\ge \nu_{k,\om}(I_{\om}(1,\eta))
\geq Q_\cG(t).
\end{equation}
Having this and applying \ref{(M5b)} along with Lemma~\ref{l1mis.cm7.2} and the last assertion of Theorem~\ref{t1mis.cm7}, we get for all $2\le i\le j$ that
$$
\begin{aligned}
\nu_{k,\om}(I_{\om}(i,\eta))
&=\nu_{k,\om}\(T_{\om,0}^{-(i-1)}\(I_{\theta^{i-1}(\om)}(1)\)\)
\ge A^{-t(i-1)}\lm_{k,\om}^{-(i-1)}\nu_{k,\theta^{i-1}(\om)}\(I_{\theta^{i-1}(\om)}(1))\) \\
&\ge Q_\cG(t) A^{-t(i-1)}\beta_1^{-(i-1)} \\
&\ge Q_\cG(t) A^{-t(j-1)}\min\{1,\beta_1^{-(j-1)}\}.
\end{aligned}
$$
Along with \eqref{120180428} this completes the proof.
\end{proof}

Now we can prove the following. 

\begin{lemma}\label{l2tp3}
Let $t\ge 0$. Then for every $q\in\NN$ there exists a constant $\Ga_t(q)\ge 1$ such that for all $k\ge q$, all $n\in\NN$, and all $\om\in\Om$, we have
	\begin{align*}
		\lm_{k,\om}^n\geq \Ga_t^{-1}(q)\lm_{q,\om}^n.
	\end{align*}
\end{lemma}
\begin{proof}
For every $k\in\NN$ and $n\in\NN$ extend the function $\tr_{k,\om}^n\ind:\cJ_{\ta^n(\om)}^{(k)}(T)\lra (0,+\infty)$ to the whole interval $[0,1]$ by setting
\begin{equation}\label{120180529}
\tr_{k,\om}^n\ind(x):=\tr_{\om}^n\ind_{\Ga(k,n)}(x)
\end{equation}
for all $x\in[0,1]$, where
$$
\Ga(k,n):=\bigcap_{\ell=0}^nT_\om^{-\ell}\(I_*\cap U_{\theta^\ell(\om)}(k)\).
$$
Of course
$$
\tr_{k,\om}^n\ind(x)\le \tr_{\ell,\om}^n\ind(x)
$$
for all $\om\in\Om$, all $1\le k\le \ell\le\infty$, all $n\in\NN$ and all $x\in[0,1]$. 
In addition, by applying Proposition~\ref{p120180630} and observing that
$$
1-T_{\om,0}^{-u}\gek |I_\om(u)|
$$
for all $u\ge 1$ and all $\om\in\Om$, we see that for all $\eta\in(0,\eta_*]$ there exists $\hat K_\eta\geq 1$ such that for all $j\ge 1$, all $1\leq i\leq j$, all $k\in\NN_\infty$, all $n\in\NN$, all $\om\in\Om$, and all $x,y\in I_{\ta^n(\om)}(i,\eta)$
\begin{align}\label{1tp4}
\hat K_\eta^{-t}\leq\frac{\tr_{k,\om}^n\ind(y)}{\tr_{k,\om}^n\ind(x)}\leq \hat K_\eta^t.
\end{align}
Also, for every $k\ge 1$ there exists $\eta_k\in(0,\eta_*]$ such that
$$
\cJ_\om^{(k)}(T)\subset [\eta_k,1-\eta_k]
$$
for all $\om\in\Om$. By formula \eqref{2mis.cm5} and Lemma~\ref{l1_2018_03_23}, for all $k,n\in\NN$ and all $\om\in\Om$, we have that
	\begin{align*}
		\lm_{k,\om}^n
=\int_{\cJ_{\ta^n(\om)}^{(k)}(T)}\tr_{k,\om}^n\ind d\nu_{k,\ta^n(\om)}
=\int_0^1\tr_{k,\om}^n\ind(x) d\nu_{k,\ta^n(\om)}(x).
	\end{align*}
By Theorem~\ref{t1mis.cm7} and Proposition~\ref{p1mis.cm7.1.2}, we have
$$
\begin{aligned}
\int_0^1\tr_{\infty,\om}^n\ind(x) d\nu_{\ta^n(\om)}(x)
&=\int_{\cJ_{\ta^n(\om)}(T)}\tr_{\infty,\om}^n\ind d\nu_{\ta^n(\om)} =\int_{\cJ_{\ta^n(\om)}(T)\bs\set{0,1}}\tr_{\infty,\om}^n\ind d\nu_{\ta^n(\om)}\\
&=\lm_{\infty,\om}^n\int_{T_\om^{-1}(\cJ_{\ta^n(\om)}(T)\bs\set{0,1})}\ind d\nu_{\om} 
=\lm_{\infty,\om}^n\nu_\om(T_\om^{-1}(\cJ_{\ta^n(\om)}(T)\bs\set{0,1}))\\
&\leq \lm_{\infty,\om}^n.
\end{aligned}
$$
	
Fix $q\in\NN$. Fix also arbitrary $n\in\NN$ and $\om\in\Om$. Let $x_{\om,q,n}\in\cJ_{\ta^n(\om)}^{(q)}(T)$ be any point maximizing the function 
$$
\tr_{q,\om}^n\ind:\cJ_{\ta^n(\om)}^{(q)}(T)\lra(0,\infty).
$$
Then $x_{\om,q,n}\in I_{\ta^n(\om)}(j,\eta_q)$ for some $1\leq j\leq q$. Hence, using Lemma~\ref{l1cmii3.2BD}, we get for every $k\geq q$ ($k=\infty$ is allowed) that 
\begin{align*}
\lm_{k,\om}^n
&\geq\int_0^1\tr_{k,\om}^n\ind(x) d\nu_{k,\ta^n(\om)}(x)
 \geq \int_0^1\tr_{q,\om}\ind(x) d\nu_{k,
   \ta^n(\om)}(x)
\geq \int_{I_{\ta^n(\om)}(j,\eta_q)}\tr_{q,\om}\ind d\nu_{k,\ta^n(\om)} \\
&\geq  \hat K_{\eta_q}^{-t}\tr_{q,\om}^n\ind(x_{\om,q,n})\nu_{k,\ta^n(\om)}\(I_{\ta^n(\om)}(j,\eta_q)\)
\geq \hat K_{\eta_q}^{-t}Q_q\tr_{q,\om}^n\ind(x_{\om,q,n}) \\
&\geq \hat K_{\eta_q}^{-t}Q_q\int_{\cJ_{\ta^n(\om)}^{(q)}(T)}\tr_{q,\om}^n\ind d\nu_{q,\ta^n(\om)}\\
		&=\hat K_{\eta_q}^{-t}Q_q\lm_{q,\om}^n.
\end{align*}
The proof is complete.
\end{proof}
	
As an immediate consequence of this lemma and Birkhoff's Ergodic Theorem we get the following. 
\begin{corollary}\label{c1tp5}
	For every $t\geq 0$ the sequence $(\cEP_k(t))_{k=1}^\infty$ is weakly increasing. In consequence $\ul{\cE}\P_\infty(t)=\ol{\cE}\P_\infty(t)$, and we recall that this common value is denoted by $\cEP_\infty(t)$.	
\end{corollary}
\begin{remark}\label{r2tp5}
Together with Birkhoff's Ergodic Theorem, the proof of Lemma \ref{l2tp3}, more precisely its last displayed formula taken with $k=q$, also shows that the value of $\cEP_k(t)$ is independent of the conformal measure $\nu_k(t)$; it depends only on $t$ and $k$.
\end{remark}
Having Lemma \ref{l2tp3} at our disposal, we can prove the following. 

\begin{theorem}\label{t1tp5}
	$AA(T)\sub CM(T)$.
\end{theorem}
\begin{proof}
The case of $t=0$ is entirely taken care of by Proposition \ref{p1cmii4}. So, assume from now on that $t>0$. According to Theorem \ref{t1mis.cm7} in order to complete the proof, we are to show that 
$$
\nu\(\Om\times\Crit(\cG)\)=0
$$
with $\nu=\nu_k$ being a weak limit measure resulting from Lemma \ref{l1mis.cm5}; see also formula \eqref{2mis.cm5} and Lemma~\ref{l1_2018_03_23}. We will do it now. Denote 

\begin{equation}\label{120180525}
		\dl:=\frac{1}3\left(\cEP_\infty(t)+\frac{1}{1+\gm_+}\chi_0t\right)>0,
	\end{equation}
	due to $A$--admissibility of $t$. Take then $q^*\in\NN$ so large (see Corollary \ref{c1tp5}) that
	\begin{equation}\label{220180525}
		\cEP_q(t)+\frac{1}{1+\gm_+}\chi_0t>2\dl
	\end{equation}
for all $q\ge q^*$. Fix such a $q$.	Since, by Birkhoff's Ergodic Theorem and the last assertion of Theorem \ref{t1mis.cm7}, 
	\begin{align*}
		\cEP_q(t)=\limty{n}\frac{1}{n}\log\lm_{q,\om}^n
	\end{align*}
	for $m$--a.e. $\om\in\Om$, say $\om\in\Om^*$ with $m(\Om^*)=1$, we have for every $\om\in\Om^*$ that
	\begin{align*}
		\frac{1}{n}\log\lm_{q,\om}^n+\frac{1}{n}\frac{t}{1+\gm_+}\log\absval{(T_\om^n)'(0)}\geq 2\dl
	\end{align*}
	for all $n\geq 1$ large enough, say $n\geq N_1(\om)\geq 1$. By virtue of Lemma \ref{l2tp3} this gives 
	\begin{align*}
		\frac{1}{n}\log\lm_{k,\om}^n+\frac{1}{n}\frac{t}{1+\gm_+}\log\absval{(T_\om^n)'(0)}\geq \dl
	\end{align*}
	for all $\om\in\Om^*$, all $k\geq q$, and all $n\geq N_1(\om)$ large enough, say $n\geq N_2(\om)\geq N_1(\om)$. Equivalently, 
\begin{equation}\label{220180526}
\lm_{k,\om}^{-n}\absval{(T_\om^n)'(0)}^{-\frac{t}{1+\gm_+}}\leq e^{-\dl n}.
\end{equation}
It therefore follows from formula \eqref{3mis118} and \eqref{1mis118} respectively that for all $\Dl\in\cG$ we have 
	\begin{equation}\label{120180523}
		\nu_{k,\om}(T_{\om,\Dl}^{-1}(V_{\ta(\om)}(n)))\leq\sum_{l=n}^\infty e^{-\dl l}=(1-e^{-\dl})^{-1}e^{-\dl n}
	\end{equation}
	and for all $\Dl\in\cG$ we have 
	\begin{equation}\label{220180523}
		\nu_{k,\om}(T_{\om,\Dl}^{-1}\circ T_{\ta(\om),1}^{-1}(V_{\ta^2(\om)}(n)))\leq\sum_{l=n}^\infty e^{-\dl l}=(1-e^{-\dl})^{-1}e^{-\dl n}
	\end{equation}
	for all $\om\in\Om^*$, all $k\geq q$, and all $n\geq N_2(\om)$. Since we could have obviously chosen the function $\Om^*\ni\om\mapsto N_2(\om)$ in a measurable way, for every $\ep>0$ there exists $N_\ep\geq 1$ and a measurable set $\Om_\ep^*\sub\Om$ such that $m(\Om_\ep^*)\geq 1-\ep$,
	\begin{equation}\label{320180523}
		\nu_{k,\om}(T_{\om,\Dl}^{-1}(V_{\ta(\om)}(n)))
		\leq (1-e^{-\dl})^{-1}e^{-\dl n}
	\end{equation}		
and	
\begin{equation}\label{420180523}
\nu_{k,\om}(T_{\om,\Dl}^{-1}\circ T_{\ta(\om),1}^{-1}(V_{\ta^2(\om)}(n)))
		\leq (1-e^{-\dl})^{-1}e^{-\dl n}
	\end{equation}
for all $\Dl\in\cG$, all $n\geq n_\ep\geq 1$ and all $\om\in\Om_\ep^*$. Letting $k\to\infty$ we thus conclude that 
$$
\begin{aligned}
\nu\(\Om_\ep^*\times\Crit(\cG))
&\le \nu\left(\bigcup_{\om\in\Om_\ep^*}\{\om\}\times \left(\bigcup_{\Dl\in\cG_C^{(0)}}\!\!\!T_{\om,\Dl}^{-1}\(V_{\ta(\om)}(n)\)
\cup \!\!\!\bigcup_{\Dl\in\cG_C^{(1)}}\!\!\!T_{\om,\Dl}^{-1}\circ T_{\ta(\om),1}^{-1}\(V_{\ta^2(\om)}(n)\)\right)\right) \\
&\le \#\cG_C (1-e^{-\dl})^{-1}e^{-\dl n}.
\end{aligned}
$$
Letting in turn $n\to\infty$, we thus conclude that 
$$
\nu\(\Om_\ep^*\times\Crit(\cG))=0.
$$
Finally letting $\ep\to 0$, we get that 
$$
\nu\(\Om^*\times\Crit(\cG))=0.
$$
So,
$$
\nu\(\Om\times\Crit(\cG))=0.
$$
The proof is complete. 
\end{proof}

\begin{remark}\label{r120180526}
By a straightforward enhancement of the proof of Theorem~\ref{t1tp5}, we get that for every $\ell\ge 0$,

$$
\lim_{\eta\to 0}\int_\Om\sup\big\{\nu_{k,\om}\(B\(T_\om^{-\ell}(\{0,1\}),\eta\)\):q^*\le k\le+\infty\big\}\,dm(\om)=0.
$$
Thus
$$
\lim_{\eta\to 0}\sup_{q^*\le k\le+\infty}\lt\{\nu_{k}\lt(\bu_{\om\in\Om}\{\om\}\times B\(T_\om^{-\ell}(\{0,1\}),\eta\)\rt)\rt\}=0.
$$
\end{remark}

\begin{remark}\label{r1tp6.2}
We have in fact shown in the proof of Theorem \ref{t1tp5} that if $t\in AA(T)$, then any weak limit of measures $(\nu_k)_{k=1}^\infty$ is $t$--conformal. 
\end{remark}
\begin{definition}\label{d1tp6}
Because of Theorem \ref{t1tp5} for each $t\in AA(T)$ there exists a $t$--conformal measure $\nu$ for $T$. Furthermore, because of Birkhoff's Ergodic Theorem and Lemma \ref{l1mis.cmii2} (c), the number 
	\begin{align*}
		\cEP(t):=\int_{\Om}\log\lm_{t,\nu,\om}dm(\om)
	\end{align*} 
is independent of the choice of $t$--conformal measure $\nu$. We call it the \textit{expected topological pressure} of the parameter $t$.
\end{definition}
As an immediate consequence of Lemma \ref{l2tp3} and Birkhoff's Ergodic Theorem, we get the following. 
\begin{lemma}\label{l1tp6}
	If $t\in AA(T)$, then $\cEP(t)\geq \cEP_\infty(t)$. 
\end{lemma}
Let 
$$
p_2:\Om\times[0,1]\longmapsto [0,1]
$$
be the natural projection on the second coordinate, i.e.
$$
p_2(\om,x)=x. 
$$ 
Let $\eta_*\in (0,\min\{1/2,1-(|\Dl_1|/2))$ be the number produced in Lemma~\ref{l1cmii3.2B}. Fix an arbitrary $\eta\in(0,\eta_*]$. Let $A\sub \cJ(T)$ be an arbitrary measurable set such that 
	\begin{align*}
		p_2(A)\sub[\eta,1-\eta].
	\end{align*}
It then follows from Proposition~\ref{p120180630} and conformality of the random measures $\nu_k$, $1\le k\le +\infty$, that
$$
\begin{aligned}
K_{\eta}^{-t}\nu_{k,\om}\(T^{-j}_{\om}([\eta,1-\eta])\)&\frac{\nu_{k,\ta^j(\om)}(A_{\ta^j(\om)})}{\nu_{k,\ta^j(\om)}([\eta_*,1-\eta_*])}
\le  \\
&\le \nu_{k,\om}\(T^{-j}_{\om}(A_{\ta^j(\om)})\) \le \\
&\le K_{\eta}^t\nu_{k,\om}\(T^{-j}_{\om}([\eta,1-\eta])\)\frac{\nu_{k,\ta^j(\om)}(A_{\ta^j(\om)})}{\nu_{k,\ta^j(\om)}([\eta,1-\eta])}\\
&\le K_{\eta}^t\nu_{k,\om}\(T^{-j}_{\om}([\eta,1-\eta])\)\frac{\nu_{k,\ta^j(\om)}(A_{\ta^j(\om)})}{\nu_{k,\ta^j(\om)}([\eta_*,1-\eta_*])}
\end{aligned}
$$
for all $j\geq 0$ and all $\om\in\Om$. So, applying Lemma~\ref{l1cmii3.2B}, we get 
\begin{equation}\label{1mis121}
\begin{aligned}
K_{\eta}^{-t}\nu_{k,\om}\(T^{-j}_{\om}([\eta,1-\eta_*])\)&\nu_{k,\ta^j(\om)}(A_{\ta^j(\om)})\le  \\
&\le \nu_{k,\om}\(T^{-j}_{\om}(A_{\ta^j(\om)})\) \le \\
&\le K_{\eta}^tQ_\cG^{-1}(t)\nu_{k,\om}\(T^{-j}_{\om}([\eta_*,1-\eta_*])\)\nu_{k,\ta^j(\om)}(A_{\ta^j(\om)})\\
&\le K_{\eta}^tQ_\cG^{-1}(t)\nu_{k,\ta^j(\om)}(A_{\ta^j(\om)})
\end{aligned}
\end{equation}
for all $1\le k\le +\infty$, all $\om\in\Om$ and all $j\geq 0$.

Now we shall prove the following technical, but very useful fact. 

\begin{lemma}\label{l2tp6}
If $t\in AA(T)$, then 
\begin{align*}
\lim_{\eta\to 0}\sup_{n\ge 0}
\int_{\Om}\sup\left\{\nu_{k,\om}\(T_\om^{-n}\([\eta,1-\eta]^c\):q^*\leq k\le\infty\right\} dm(\om)=0.
\end{align*}
\end{lemma}
\begin{proof}
First note that there exists $\eta_1\in(0,\eta_*)$ such that
\begin{align}\label{1mis123}
\union_{\Dl\in \cG\sms\{\Dl_0,\Dl_1\}}T_{\om,\Dl}^{-1}([0,1])
\cup\union_{\Dl\in \cG}T_{\om,\Dl}^{-1}(\Dl_1)
\sub[\eta_1,1-\eta_1]
\end{align}
for all $\om\in\Om$. For every $n\geq 0$, all $0\leq\ell\leq n$, and every set $A\sub [0,1]$ we can write 
\begin{align}
T_\om^{-n}(A)
=T_{\om,0}^{-n}(A)&\cup
\union_{j=0}^{n-1}T_\om^{-(n-j-1)}\left(\union_{\Dl\in \cG\sms\{\Dl_0,\Dl_1\}} T_{\ta^{n-(j+1)}(\om),\Dl}^{-1}\left(T_{\ta^{n-j}(\om),0}^{-j}(A)\right)\right)\cup \nonumber\\
&
\cup T_{\om,1}^{-1}\left(T_{\ta(\om),0}^{-(n-1)}(A)\right)\cup\nonumber\\
&
\cup\union_{j=0}^{n-2} T_\om^{-(n-j-2)}\left(\union_{\Dl\in \cG}T_{\ta^{n-(j+2)}(\om),\Dl}^{-1}\left(T_{\ta^{n-(j+1)}(\om),1}^{-1}\left(T_{\ta^{n-j}(\om),0}^{-j}(A)\right)\right)\right).\label{2mis123}
\end{align}

Fix $\ep>0$. Keep $\delta>0$ defined by \eqref{120180525} and and recall that $q^*\ge 1$ is determined by \eqref{220180525}, i.e. both are the same as in the proof of Theorem~\ref{t1tp5}. Fix an integer $s_1\ge 0$ so large that for every $s\ge s_1$ we have that
\begin{equation}\label{120180528}
m\(\{\omega\in\Omega:N_2(\om)> s\}\)<\ep/4,
\end{equation}
where $N_2(\om)$ also comes from the proof of Theorem~\ref{t1tp5}. Denote this set by $\Omega_s(\ep)$.

Therefore, applying \eqref{1mis121}, \eqref{1mis123} and \eqref{120180523}, we get for every $k\geq q^*$, every $n\ge s+1$, and every $\omega\in\Omega_s^c(\ep)$, that
\begin{align}
\Sigma_n^{(1)}(k,s;\omega)&:=
\sum_{j=s}^{n-1} \nu_{k,\om}\left(T_\om^{-(n-j-1)}\left(\union_{\Dl\in \cG\sms\{\Dl_0,\Dl_1\}}T^{-1}_{\ta^{n-(j+1)}(\om),\Dl}\left(T_{\ta^{n-j}(\om),0}^{-j}([0,1])\right)\right)\right)\nonumber\\
	&=\sum_{j=s}^{n-1}\sum_{\Dl\in \cG\sms\{\Dl_0,\Dl_1\}}\nu_{k,\om}\left(T_\om^{-(n-j-1)}\left(T_{\ta^{n-(j+1)}(\om),\Dl}^{-1}\(V_{\ta^{n-j}(\om)}(j)\)\right)\right)\nonumber\\
	&\leq K_{\eta_1}^tQ_\cG^{-1}(t)\sum_{j=s}^{n-1}\sum_{\Dl\in \cG\sms\{\Dl_0,\Dl_1\}} \nu_{k,\ta^{n-(j+1)}(\om)}\left(T_{\ta^{n-(j+1)}(\om),\Dl}^{-1}\(V_{\ta^{n-j}(\om)}(j)\)\right)          \nonumber\\
	&\leq K_{\eta_1}^tQ_\cG^{-1}(t)\sum_{j=s}^{n-1}\sum_{\Dl\in \cG\sms\{\Dl_0,\Dl_1\}} (1-e^{-\dl})^{-1}e^{-\dl j}  \nonumber\\
	&\leq \#\cG K_{\eta_1}^tQ_\cG^{-1}(t)(1-e^{-\dl})^{-1}\sum_{j=s}^{n-1}e^{-\dl j}\nonumber\\
	&\leq \#\cG K_{\eta_1}^tQ_\cG^{-1}(t)(1-e^{-\dl})^{-2}e^{-\dl s}.\label{1mis127}
\end{align}
Likewise, applying \eqref{1mis121}, \eqref{1mis123} and \eqref{220180523}, we get for every $k\geq q^*$, every $n\ge s+2$, and every $\omega\in\Omega_s^c(\ep)$, that
\begin{equation}\label{2mis127}
\begin{aligned}
\Sigma_n^{(2)}(k,s;\omega):&=
\sum_{j=s}^{n-2}\nu_{k,\om}\left(T_\om^{-(n-j-2)}\left(\union_{\Dl\in \cG}T_{\ta^{n-(j+2)}(\om),\Dl}^{-1}\left(T_{\ta^{n-(j+1)}(\om),1}^{-1}\left(T_{\ta^{n-j}(\om),0}^{-j}([0,1])\right)\right)\right)\right)\\
&=\sum_{j=s}^{n-2}\sum_{\Dl\in \cG}\nu_{k,\om}\left(T_\om^{-(n-j-2)}\left(T_{\ta^{n-(j+2)}(\om),\Dl}^{-1}\left(T_{\ta^{n-(j+1)}(\om),1}^{-1}\(V_{\ta^{n-j}(\om)}(j)\)\right)\right)\right)\\
&\leq K_{\eta_1}^tQ_\cG^{-1}(t)\sum_{j=s}^{n-2}\sum_{\Dl\in \cG}
\nu_{k,\ta^{n-(j+2)}(\om)}\left(T_{\ta^{n-(j+2)}(\om),\Dl}^{-1}\left(T_{\ta^{n-(j+1)}(\om),1}^{-1}\(V_{\ta^{n-j}(\om)}(j)\)\right)
\right)\\
&\leq K_{\eta_1}^tQ_\cG^{-1}(t)\sum_{j=s}^{n-2}\sum_{\Dl\in \cG}(1-e^{-\dl})^{-1}e^{-\dl j}\\
&\leq \#\cG K_{\eta_1}^tQ_\cG^{-1}(t)(1-e^{-\dl})^{-1}\sum_{j=s}^{n-2}e^{-\dl j}\\
&\leq \#\cG K_{\eta_1}^tQ_\cG^{-1}(t)(1-e^{-\dl})^{-2}e^{-\dl s}.
\end{aligned}
\end{equation}
We also declare 
$$
\Sigma_n^{(1)}(k,s;\omega):=0
$$
for all $0\le n\le s$ and  
$$
\Sigma_n^{(1)}(k,s;\omega)=\Sigma_n^{(2)}(k,s;\omega):=0
$$
for all $0\le n\le s+1$. With these declarations both \eqref{1mis127} and \eqref{2mis127} become true for every $k\geq q^*$, every $n\ge 0$, and every $\omega\in\Omega_s^c(\ep)$.

Now there are only left four terms in \eqref{2mis123} to take care of. Two of them are easy. Indeed, invoking \eqref{220180526}, we get for every $k\geq q^*$, every $n\ge s$, and every $\omega\in\Omega_s^c(\ep)$, that
\begin{equation}\label{120180526}
\begin{aligned}
\Sigma_n^{(3)}(k,s;\omega)
:&=\nu_{k,\om}\left(T_{\om,0}^{-n}([0,1])\right)
\le C \lm_{k,\om}^{-n}\absval{(T_\om^n)'(0)}^{-t}
\le C \lm_{k,\om}^{-n}\absval{(T_\om^n)'(0)}^{-\frac{t}{1+\gm_+}} \\
&\leq Ce^{-\dl n}
\leq Ce^{-\dl s},
\end{aligned}
\end{equation}
and
\begin{equation}\label{320180526}
\begin{aligned}
\Sigma_n^{(4)}(k,s;\omega)
:&=\nu_{k,\om}\left(T_{\om,1}^{-1}\left(T_{\ta(\om),0}^{-(n-1)}([0,1])\right)\right)
\le C \lm_{k,\om}^{-n}\absval{(T_\om^n)'(0)}^{-t} \\
&\le C \lm_{k,\om}^{-n}\absval{(T_\om^n)'(0)}^{-\frac{t}{1+\gm_+}} \\
&\leq Ce^{-\dl n}
\leq Ce^{-\dl s},
\end{aligned}
\end{equation}
with some constant $C\in [1,\infty)$. Now fix an integer $s\ge s_1$ so large
that 
\begin{equation}\label{420180526}
2Ce^{-\dl s}+2\#\cG K_{\eta_1}^tQ_\cG^{-1}(t)(1-e^{-\dl})^{-2}e^{-\dl s}
<\ep/4.
\end{equation}
Now we deal with the last two, somewhat more involved terms. Take any $\eta\in(0,\eta_1]$. Then 
\begin{equation}\label{520180526}
\begin{aligned}
\Sigma_n^{(5)}(k,s;\omega):&=
\sum_{j=0}^{s-1} \nu_{k,\om}\left(T_\om^{-(n-j-1)}\left(\union_{\Dl\in \cG\sms\{\Dl_0,\Dl_1\}}T^{-1}_{\ta^{n-(j+1)}(\om),\Dl}\left(T_{\ta^{n-j}(\om),0}^{-j}([\eta,1-\eta]^c)\right)\right)\right)\\
	&=\sum_{j=0}^{s-1}\sum_{\Dl\in \cG\sms\{\Dl_0,\Dl_1\}}\nu_{k,\om}\left(T_\om^{-(n-j-1)}\left(T_{\ta^{n-(j+1)}(\om),\Dl}^{-1}\left(T_{\ta^{n-j}(\om),0}^{-j}([\eta,1-\eta]^c)\right)\right)\right)\\
	&\leq K_{\eta_1}^tQ_\cG^{-1}(t)\sum_{j=0}^{s-1}\sum_{\Dl\in \cG\sms\{\Dl_0,\Dl_1\}} \nu_{k,\ta^{n-(j+1)}(\om)}\left(T_{\ta^{n-(j+1)}(\om),\Dl}^{-1}\left(T_{\ta^{n-j}(\om),0}^{-j}([\eta,1-\eta]^c)\right)\right)       
\end{aligned}
\end{equation}
and
\begin{equation}\label{620180526}
\begin{aligned}
\Sigma_n^{(6)}(k,&s;\omega)
:= \\
&=\sum_{j=0}^{s-1}\nu_{k,\om}\left(T_\om^{-(n-j-2)}\left(\union_{\Dl\in \cG}T_{\ta^{n-(j+2)}(\om),\Dl}^{-1}\left(T_{\ta^{n-(j+1)}(\om),1}^{-1}\left(T_{\ta^{n-j}(\om),0}^{-j}([\eta,1-\eta]^c)\right)\right)\right)\right)\\
&=\sum_{j=0}^{s-1}\sum_{\Dl\in \cG}\nu_{k,\om}\left(T_\om^{-(n-j-2)}\left(T_{\ta^{n-(j+2)}(\om),\Dl}^{-1}\left(T_{\ta^{n-(j+1)}(\om),1}^{-1}\left(T_{\ta^{n-j}(\om),0}^{-j}([\eta,1-\eta]^c)\right)\right)\right)\right)\\
&\leq K_{\eta_1}^tQ_\cG^{-1}(t)\sum_{j=0}^{s-1}\sum_{\Dl\in \cG}
\nu_{k,\ta^{n-(j+2)}(\om)}\left(T_{\ta^{n-(j+2)}(\om),\Dl}^{-1}\left(T_{\ta^{n-(j+1)}(\om),1}^{-1}\left(T_{\ta^{n-j}(\om),0}^{-j}([\eta,1-\eta]^c)\right)\right)\right)\\
\end{aligned}
\end{equation}
for every $k\geq q^*$, every $n\ge s$, and every $\omega\in\Omega$. Now, in view of Remark~\ref{r120180526} there exists $\eta_2\in(0,\eta_1]$ such that for every $k\geq q^*$ and every $\eta\in(0,\eta_2]$, we have that 
\begin{equation}\label{720180526}
K_{\eta_1}^tQ_\cG^{-1}(t)\sum_{j=0}^{s-1}\sum_{\Dl\in \cG\sms(\Dl_0\cup\Dl_1)} \nu_{k,\ta^{n-(j+1)}(\om)}\left(T_{\ta^{n-(j+1)}(\om),\Dl}^{-1}\left(T_{\ta^{n-j}(\om),0}^{-j}([\eta,1-\eta]^c)\right)\right)
<\ep/4
\end{equation}
and
\begin{equation}\label{820180526}
K_{\eta_1}^tQ_\cG^{-1}(t)\sum_{j=0}^{s-1}\sum_{\Dl\in \cG}
\nu_{k,\ta^{n-(j+2)}(\om)}\left(T_{\ta^{n-(j+2)}(\om),\Dl}^{-1}\left(T_{\ta^{n-(j+1)}(\om),1}^{-1}\left(T_{\ta^{n-j}(\om),0}^{-j}([\eta,1-\eta]^c)\right)\right)\right)<\ep/4.
\end{equation}
Put:
\begin{equation}\label{520180528}
\Xi_n(\eta,\om)
:=\sup\left\{\nu_{k,\om}\(T_\om^{-n}\([\eta,1-\eta]^c\):q\leq k\le\infty\right\}
\end{equation}
and 
$$
\Sigma_n^{(i)}(s;\omega)
:=\sup\big\{\Sigma_n^{(i)}(k,s;\omega):q\leq k\le\infty\big\}
$$
for every $i=1,2,\ldots,6$. As the last ingredient, by virtue of \eqref{120180528}, we get that
\begin{equation}\label{220180528}
\int_{\Om_s(\ep)}\Xi_n(\eta,\om)\,dm(\om)
\le m(\Om_s(\ep))<\ep/4,
\end{equation}
for all $n\ge 1$. Finally, taking the fruits of all the above estimates, i.e. using \eqref{2mis123}, \eqref{1mis127}, \eqref{2mis127}, \eqref{120180526}, \eqref{320180526}, \eqref{420180526}, \eqref{520180526}, \eqref{620180526}, \eqref{720180526}, 
\eqref{820180526}, and \eqref{220180528}, we get for all $k\ge q$, all $n\ge s$ and every $\eta\in(0,\eta_2]$, that
$$
\begin{aligned}
\int_{\Om}\Xi_n(\eta,\om)\, dm(\om)
&\le \int_{\Om_s(\ep)}\Xi_n(\eta,\om)dm(\om)+\sum_{i=1}^6\int_{\Om_s^c(\ep)}
\Sigma_n^{(i)}(s;\omega)\,dm(\om) \\
&<\frac{\ep}4+\frac{\ep}4+\frac{\ep}4+\frac{\ep}4
=\ep.
\end{aligned}
$$
The proof of Lemma \ref{l2tp6} is complete.
\end{proof}	

Let us record the following immediate corollary of Lemma~\ref{l2tp6}.

\begin{corollary}\label{c520180606}
If $t\in AA(T)$, then 
\begin{align*}
\lim_{\eta\to 0}\sup_{n\ge 0}\sup_{q^*\leq k\le\infty}
\Big\{\nu_{k}\(T^{-n}\(\Om\times [\eta,1-\eta]^c\)\)\Big\}=0.
\end{align*}
\end{corollary}

Now, because of Lemma~\ref{l2tp6}, we are in a position to prove the following technical lemma which is a significant strengthening of Lemma~\ref{l2tp3}.

\begin{lemma}\label{l1tp8}
If $t\in AA(T)$, then for all $\ep>0$ there exist $c_\ep\geq1$ and a measurable set $\Om_\ep\subset \Om$ such that $m\(\Om_\ep\)>1-\ep$ and for all integers $q\ge q^*$ ($q^*\ge 1$ being determined by \eqref{220180525}), all integers $k\geq q$, all $n\geq 0$, and all $\om\in\Om_\ep$ we have that 
	\begin{align*} 
		\lm_{k,\om}^n\geq c_\ep^{-1}\lm_{q,\om}^n.
	\end{align*}	  
\end{lemma}	
\begin{proof}
	For every $\eta\in (0,1)$ and $\om\in\Om$ let
	\begin{align*}
		\NN_\eta(\om):=\set{\ell\in\NN: I_\ell(\om)\cap [\eta,1-\eta]\nonempty}
	\end{align*}
	and let 
	\begin{align*}
		\NN_\eta^*(\om):=N_\eta(\om)\cap\set{2,3,\dots}=\NN_\eta(\om)\bs\set{1}.
	\end{align*}
	Let $\ell_\eta(\om)$ be the largest element in $\NN_\eta(\om)$. Of course 
	\begin{align*}
		\NN_\eta(\om)=\set{1,2,\dots,\ell_\eta(\om)}\spand\NN_\eta^*(\om)=\set{2,3,\dots,\ell_\eta(\om)}.
\end{align*}
By conditions \ref{(M5c)} of the definition of $T$, 
\begin{equation}\label{520180529}
T_{\om,0}^{-j}(1)\leq C\kp^{-j}
\end{equation}
for all $j\geq 0$ with some constant $C\ge 1$. In particular, $T_{\om,0}^{-(\ell_\eta(\om)-1)}(1)\leq C\kp^{-(\ell_\eta(\om)-1)}$. But, on the other hand, $T_{\om,0}^{-(\ell_\eta(\om)-1)}(1)\geq\eta$. Hence, $\kp^{-(\ell_\eta(\om)-1)}\geq\eta/C$. This gives
	\begin{align}\label{2tp10}
		\ell_\eta(\om)\leq 1+\frac{\log(C/\eta)}{\log\kp}:=\ell_\eta
	\end{align}
	for all $\om\in\Om$.
Now, fix $\ep\in(0,1)$. Then, by virtue of Lemma \ref{l2tp6}, with $\Xi_n(\eta,\om)$ defined by \eqref{520180528} there exists $\eta_\ep\in(0,1/2)$ so small that 
$$
\int_{\Om}\Xi_n(\eta_\ep,\om)\,dm(\om)<\frac{\ep}{2}
$$	
for all integers $n\ge 0$. Hence, by Tchebyshev's Inequality
$$
m\(\big\{\om\in\Om:\Xi_n(\eta_\ep,\om)>1/2\big\}\)<\ep
$$
for all integers $n\ge 0$. So, denoting 
$$
\Om_\ep:=\Om\sms \big\{\om\in\Om:\Xi_n(\eta_\ep,\om)>1/2\big\},
$$
we get that
$$
m\(\Om_\ep\)>1-\ep
$$
and
\begin{align}\label{1tp11}
\inf\left\{\nu_{k,\om}\(T_\om^{-n}\([\eta_\ep,1-\eta_\ep]\)\):q\leq k\le\infty,\,  \om\in \Om_\ep\right\}\geq\frac{1}{2}
\end{align} 
for all integers $n\ge 0$. For the rest of the proof keep $\om\in \Om_\ep$. By \eqref{1tp11} and \eqref{2tp10} for every integer $n\ge 0$ there exists $1\leq j_n\leq \ell_{\eta_\ep}$ such that 
	\begin{align}\label{2tp11}
		\nu_{k,\om}\left(T_\om^{-n}\left(I_{\ta^n(\om),\eta_\ep}(j_n)\right)\right)
		\geq (2\ell_{\eta_\ep})^{-1}
	\end{align} 
for all $k\ge q^*$. As in the proof of Lemma~\ref{l2tp3}, we treat all the functions $\cL_{k,\om}^n\ind$, $n\ge 0$, as defined on the whole interval $[0,1]$ according to the definition \eqref{120180529}. Having any $q\ge q^*$, let $x(q,\om,j_n)\in I_{\ta^n(\om),\eta_\ep}(j_n)$ be a point maximizing the value of $\cL^n_{q,\om}\ind$ restricted to the set $I_{\ta^n(\om),\eta_\ep}(j_n)$, which, we recall, is defined by the formula \eqref{220180529}. Using now \eqref{2tp11} (with $k$ being $q$) along with Lemma~\ref{l1cmii3.2BD} and \eqref{1tp4}, we obtain
	\begin{align*}
		\lm_{k,\om}^n
		&\geq \int_{\cJ_{\ta^n(\om)}(T)}\cL_{k,\om}^n\ind d\nu_{k,\ta^n(\om)}
		\geq\int_{I_{\ta^n(\om),\eta_\ep}(j_n)}\cL_{k,\om}^n\ind d\nu_{k,\ta^n(\om)}\geq \int_{I_{\ta^n(\om),\eta_\ep}(j_n)}\cL_{q,\om}\ind d\nu_{k,\ta^n(\om)}\\
		&\geq K^{-t}_{\eta_\ep}\cL_{q,\om}^n\ind\(x(q,\om,j_n)\)\nu_{k,\ta^n(\om)}(I_{\ta^n(\om),\eta_\ep}(j_n))
		\geq K_{\eta_\ep}^{-t}Q_{\ell_{\eta_\ep}}\cL_{q,\om}^n\ind\(x(q,\om,j)\)\\
		&\geq K_{\eta_\ep}^{-t}Q_{\ell_{\eta_\ep}}\frac{1}{\nu_{q,\ta^n(\om)}(I_{\ta^n(\om),\eta_\ep}(j_n))}\int_{I_{\ta^n(\om),\eta_\ep}(j)}\cL_{q,\om}^n\ind d\nu_{q,\ta^n(\om)}\\
		&\geq K_{\eta_\ep}^{-t}Q_{\ell_{\eta_\ep}}\int_{I_{\ta^n(\om),\eta_\ep}(j_n)}\cL_{q,\om}^n\ind d\nu_{q,\ta^n(\om)}
		=K_{\eta_\ep}^{-t}Q_{\ell_{\eta_\ep}}\lm_{q,\om}^n\int_{I_{\ta^n(\om),\eta_\ep}(j_n)}\left(\lm_{q,\om}^{-n}\cL_{q,\om}^n\right)\ind d\nu_{q,\ta^n(\om)}\\
		&=K_{\eta_\ep}^{-t}Q_{\ell_{\eta_\ep}}\lm_{q,\om}^n\nu_{q,\om}\left(T_\om^{-n}\left(I_{\ta^n(\om),\eta_\ep}(j_n)\right)\right) \\
		&\geq K_{\eta_\ep}^{-t}Q_{\ell_{\eta_\ep}}(2\ell_{\eta_\ep})^{-1}\lm_{q,\om}^n.
	\end{align*}
	The proof of Lemma \ref{l1tp8} is complete.   
\end{proof}

We are now in a position to prove the following. 
\begin{proposition}\label{p1tp12}
	If $t\in AA(T)$, then $\cEP(t)=\cEP_\infty(t)$. 
\end{proposition}
\begin{proof}
	In view of Lemma \ref{l1tp6}, it suffices to prove that 
	\begin{align}
		\cEP(t)\leq \cEP_\infty(t).\label{1tp12}
	\end{align}
	For every $\eta>0$ and all $k,n\in\NN$, put
	\begin{align*}
		\Dl_{k,n}(\eta):=\set{\om\in\Om: \log\lm_{k,\om}^n\leq(\cEP_\infty(t)+\eta)n}.
	\end{align*}
By virtue of Theorem~\ref{t1mis.cm7}, we have $\log(\lm_{k,\om}^n/\bt_1^n)\geq 0$ for all $\om\in\Om$. Hence, applying Tchebyshev's Inequality along with Corollary \ref{c1tp5} and Theorem~\ref{t1mis.cm7}, we get
	\begin{align*}
		m(\Dl_{k,n}^c(\eta))&=m\left(\set{\om\in\Om: \log(\lm_{k,\om}^n/\bt^n_1)\geq (\cEP_\infty(t)-\log\bt_1+\eta)n}\right)\\
		&\leq \frac{\int_{\Om}\log(\lm_{k,\om}^n/\bt_1^n)dm(\om)}{(\cE P_\infty(t)-\log\bt_1+\eta)n}
		=\frac{(\cE P_k(t)-\log\bt_1)n}{(\cEP_\infty(t)-\log\bt_1+\eta)n}
		=\frac{\cE P_k(t)-\log\bt_1}{\cEP_\infty(t)-\log\bt_1+\eta}\\
		&\leq \frac{\cEP_\infty(t)-\log\bt_1}{\cEP_\infty(t)-\log\bt_1+\eta}
		 =1-\frac{\eta}{\cEP_\infty(t)-\log\bt_1+\eta}\\
		&\leq 1-\frac{\eta}{\log\bt_2-\log\bt_1+\eta}.
	\end{align*}
	Therefore, 
	\begin{align*}
		m(\Dl_{k,n}(\eta))\geq \frac{\eta}{\log\bt_2-\log\bt_1+\eta}.
	\end{align*}
	Keeping $\eta>0$ and $n\in\NN$ fixed, set
	\begin{align*}
		\Dl_n^\infty:=\intersect_{j=1}^\infty\union_{k=j}^\infty\Dl_{k,n}(\eta).
	\end{align*}
	Of course, 
	\begin{align}
		m(\Dl_n^\infty(\eta))\geq \frac{\eta}{\log\bt_2-\log\bt_1+\eta}>0\label{1tp13}
	\end{align}
	and 
	\begin{align}
		\Dl_n^\infty(\eta)\sub \set{\om\in\Om:\varliminf_{k\to\infty} \lm_{k,\om}^n\leq \exp((\cEP_\infty(t)+\eta)n)}.\label{2tp13}
	\end{align}

Now fix an arbitrary $\ep\in(0,\frac{1}{3}\eta(\log\bt_2-\log\bt_1+\eta)^{-1})$. Let the set $\Om_\ep\sub\Om$ and the constant $c_\ep\ge 1$ be those produced by Lemma \ref{l1tp8}. It then follows from this lemma along with \eqref{1tp13} and \eqref{2tp13} that
	\begin{align}
		\lm_{q,\om}^n\leq c_\ep\exp((\cEP_\infty(t)+\eta )n)\label{3tp13}	
	\end{align}
	for all $q\geq q^*$, $n\geq 0$, and $\om\in\Gm_n^*(\eta):=\Om_\ep\cap\Dl_n^\infty(\eta)$, and also that
	\begin{align*}
		m(\Gm_n^*(\eta))\geq\frac{2}{3}\eta(\log\bt_2-\log\bt_1+\eta)^{-1}.
	\end{align*} 
	Since $\cEP(t)=\int_{\Om}\log\lm_\om dm(\om)$, it follows from Birkhoff's Ergodic Theorem, that there exist an integer $n_\ep\geq 1$ and a measurable set $\hat\Om_\eta\sub\Om$ such that 
	\begin{align*}
		m(\hat\Om_\eta)\geq 1-\frac{1}{3}\eta(\log\bt_2-\log\bt_1+\eta)^{-1}
	\end{align*} 
	and 
	\begin{align*}
		\frac{1}{n}\log\lm_\om^n\geq \cEP(t)-\eta
	\end{align*}
	for all $\om\in \hat\Om_\eta$ and all $n\geq n_\ep$. Letting 
	\begin{align*}
		\Gm_n(\eta):=\hat\Om_\eta\cap\Gm_n^*(\eta),
	\end{align*}
	we then have that 
	\begin{align}
		m(\Gm_n(\eta))\geq \frac{1}{3}\eta (\log\bt_2-\log\bt_1+\eta)^{-1}>0\label{2tp14}
	\end{align}
	and 
	\begin{align}
		\lm_\om^n\geq \exp((\cEP(t)-\eta)n)\label{1tp14}
	\end{align}
for all $\om\in\Gm_n(\eta)$ and all $n\geq n_\ep$. Now for every $j\geq 3$ put
	\begin{align*}
		\cJ_\om(T,j):=\cJ_\om(T)\cap (1/j,1-1/j).
	\end{align*}
	
	Then for all $q\geq 1$ large enough, say $q\geq q_{j,n}\ge q^*$,
	\begin{align*}
		\cL_{q,\om}^n\ind\rvert_{\cJ_\om(T,j)}=\cL_\om^n\ind\rvert_{\cJ_\om(T,j)} 
	\end{align*}
	for all $\om\in\Om$. Therefore, if $\Gm\sub\Om$ is an arbitrary measurable set, we get for every $q\geq q_{j,n}$ that
	\begin{align*}
		\int_{\Gm}\lm_{q,\om}^n dm(\om)
		&=\int_{\Gm}\int_{\cJ_{\ta^n(\om)}(T)}\cL_{q,\om}^n\ind d\nu_{q,\ta^n(\om)}dm(\om)
		\geq \int_{\Gm}\int_{\cJ_{\ta^n(\om)}(T,j)}\cL_{q,\om}^n\ind d\nu_{q,\ta^n(\om)}dm(\om)\\
		&=\int_{\Gm}\int_{\cJ_{\ta^n(\om)}(T,j)}\cL_{\om}^n\ind d\nu_{q,\ta^n(\om)}dm(\om) \\
		&=\int_{\Om}\int_{\cJ_{\om}(T,j)}\ind_{\ta^{n}(\Gm)}\cL_{\ta^{n}(\om)}^n\ind d\nu_{q,\om}dm(\om).
	\end{align*}
	Now, since 
$\left(\cJ_\om(T,j)\right)_{\om\in\Om}$ is an open random set, since 
$$
\bu_{\tau\in\Om}\{\tau\}\times\cJ_\tau(T,j)\ni (\om,x)\longmapsto\ind_{\ta^{n}(\Gm)}\cL_{\ta^{-n}(\om)}^n\ind(x)\in[0,+\infty)
$$
is a random continuous function, and since the sequence $(\nu_q)_{q=1}^\infty$ converges weakly to $\nu$, applying Portmanteau's Theorem (see \cite{Crauel}), we get
	\begin{align*}
\liminfty{q}\int_{\Gm} \lm_{q,\om}^n dm(\om)
		&\geq\liminfty{q} \int_{\Om}\int_{\cJ_\om(T,j)} \ind_{\ta^{n}(\Gm)}\cL_{\ta^{-n}(\om)}^n\ind d\nu_{q,\om}dm(\om)\\
		&\geq \int_{\Om}\int_{\cJ_\om(T,j)} \ind_{\ta^{n}(\Gm)}\cL_{\ta^{-n}(\om)}^n\ind d\nu_{\om}dm(\om)\\
		&=\int_{\ta^{n}(\Gm)}\int_{\cJ_\om(T)}\ind_{\cJ_\om(T,j)}\cL_{\ta^{-n}(\om)}^n\ind d\nu_{\om}dm(\om).
	\end{align*} 
But since the sequence of non--negative measurable functions 
	\begin{align*}
		\left(\ind_{\cJ_\om(T,j)}\cL_{\ta^{-n}(\om)}^n\ind\right)_{j\geq 3}
	\end{align*}
converges weakly increasingly $m$--a.e. (remember that $\nu(\Om\times\set{0,1})=0$) to the function $(\om,x)\mapsto\cL_{\ta^{-n}(\om)}^n\ind(x)$ as $j\to\infty$, by virtue of the Lebesgue Monotone Convergence Theorem, we get that 
	\begin{align*}
		\liminfty{q}\int_{\Gm}\lm_{q,\om}^ndm(\om)
		&\geq \int_{\ta^{n}(\Gm)}\int_{\cJ_\om(T)}\cL^n_{\ta^{-n}(\om)}\ind d\nu_\om dm(\om)
		=\int_\Gm\int_{\cJ_{\ta^n(\om)}(T)}\cL_\om^n\ind d\nu_{\ta^n(\om)}dm(\om)\\
		&=\int_\Gm\lm_\om^n dm(\om).
	\end{align*}
	Substituting into this formula $\Gm_n(\eta)$ for $\Gm$ and using \eqref{3tp13} along with \eqref{1tp14}, we obtain
	\begin{align*}
		c_\ep\exp((\cEP_\infty(t)+\eta)n)m(\Gm_n(\eta))
		&\geq \liminfty{q}\int_{\Gm_n(\eta)}\lm_{q,\om}^ndm(\om)
		\geq \int_{\Gm_n(\eta)}\lm_\om^n dm(\om)\\
		&\geq \exp((\cEP(t)-\eta)n)m(\Gm_n(\eta)). 
	\end{align*}
	Since by \eqref{2tp14}, $m(\Gm_n(\eta))>0$, we thus get
	\begin{align*}
		c_\ep\exp((\cEP_\infty(t)+\eta)n)\geq \exp((\cEP(t)-\eta)n).
	\end{align*} 
	Equivalently, we have 
	\begin{align*}
		\cE _\infty(t)+\eta+\frac{1}{n}\log c_\ep\geq \cEP(t)-\eta.
	\end{align*}
	Letting $n\to\infty$, this gives $\cEP(t)\leq \cEP_\infty(t)+2\eta$. Finally, letting $\eta\to0$, formula \eqref{1tp12} follows and the proof of Proposition~\ref{p1tp12} is complete.  
\end{proof}

\section{Expected Pressure 2}
For every $(\om,x)\in \cJ(T)$ let 
$$
\underline\chi_\om(x)
:=\varliminf_{n\to\infty}\frac1n\log\big|\(T_\om^n\)'(x)\big|
\  \  \  {\rm and} \  \   \
\overline\chi_\om(x)
:=\varlimsup_{n\to\infty}\frac1n\log\big|\(T_\om^n\)'(x)\big|.
$$
The numbers $\underline\chi_\om(x)$ and $\overline\chi_\om(x)$ are respectively called the lower and the upper Lyapunov exponent of $T$ at the point $(\om,x)\in\cJ(T)$. If 
$$
\underline\chi_\om(x)=\overline\chi_\om(x),
$$
then this common value is simply called the Lyapunov exponent of $T$ at the point $(\om,x)\in\cJ(T)$. For all integers $1\le k<+\infty$, we set
$$
\ul{\chi}(k):=\inf\big\{\underline\chi_\om(x):(\om,x)\in \cJ^{(k)}(T)\big\}>0,
$$
where positivity follows immediately from Corollary~\ref{c1mis81} and Proposition~\ref{p120180630}. We recall that given $t\ge 0$ and an integer $1\le k\le\infty$, 
$$
\nu_{t,k}=\big\{\nu_{t,k,\om}\big\}_{\om\in\Om}
$$
is the $t$--conformal measure for the map $T|_{\cJ^{(k)}(T)}:\cJ^{(k)}(T)\to\cJ^{(k)}(T)$ defined in Lemmas~\ref{l1mis.cm5}, \ref{l1_2018_03_23}, and  Theorem~\ref{t1mis.cm7}.

All of our investigations of the expected pressure function are based on the following technical result.
\begin{lemma}\label{l1ep1}
	If $0\leq s\leq t$ and $k\in\NN$, then 
	\begin{align*}
		(t-s)\ul{\chi}(k)\leq\cEP_k(s)-\cEP_k(t)\leq(t-s)\log\norm{T'}_\infty.
	\end{align*}
\end{lemma}
\begin{proof}
	Let $\Dl_k\sub[0,1]$ be the convex hull of $\cJ^{(k)}(T)$. For every $\om\in\Om$, every $x\in\cJ^{(k)}(T)$, and every $n\geq 0$ let 
	\begin{align*}
		I_\om(x,n):=T_{\om,x}^{-n}(\Dl_k).
	\end{align*}
	By conformality of both measures $\nu_{t,k}$ and $\nu_{s,k}$ and Proposition~\ref{p120180630}, we have that 
	\begin{align*}
	\nu_{t,k,\om}(I_\om(x,n))\comp\lm_{t,k,\om}^{-n}\absval{(T_\om^n)'(x)}^{-t}
	\end{align*}
	and
	\begin{align*}
	\nu_{s,k,\om}(I_\om(x,n))\comp\lm_{s,k,\om}^{-n}\absval{(T_\om^n)'(x)}^{-s}
	\end{align*}
	for all $\om\in\Om$, all $n\in\NN$, and all $x\in\cJ_\om^{(k)}(T)$, where the comparability constants depend in general on $s$, $t$, and $k$ but are independent of $\om$, $n$ and $x$. Therefore, 
	\begin{align*}
		\frac{\nu_{t,k,\om}(I_\om(x,n))}{\nu_{s,k,\om}(I_\om(x,n))}\comp\absval{(T_\om^n)'(x)}^{s-t}\frac{\lm_{s,k,\om}^n}{\lm_{t,k,\om}^n}.	
	\end{align*}
	Now, for every $\om\in\Om$ there exists a measurable set $\cJ_\om^{(k)}(T,s)\sub\cJ_\om^{(k)}(T)$ such that $\nu_{s,k,\om}(\cJ_\om^{(k)}(T,s))=1$ and 
	\begin{align*}
		\liminf_{n\to\infty}\frac{\nu_{t,k,\om}(I_\om(x,n))}{\nu_{s,k,\om}(I_\om(x,n))}<+\infty
	\end{align*}
	for all $x\in\cJ_\om^{(k)}(T,s)$. Likewise, for every $\om\in\Om$ there exists a measurable set $\cJ_\om^{(k)}(T,t)\sub\cJ_\om^{(k)}(T)$ such that $\nu_{t,k,\om}(\cJ_\om^{(k)}(T,t))=1$ and 
	\begin{align*}
	\liminf_{n\to\infty}\frac{\nu_{s,k,\om}(I_\om(x,n))}{\nu_{t,k,\om}(I_\om(x,n))}<+\infty	
	\end{align*}
	for all $x\in\cJ_\om^{(k)}(T,t)$. By Birkhoff's Ergodic Theorem
	\begin{align*}
		\cEP_k(t)=\lim_{n\to\infty}\frac{1}{n}\log\lm_{t,k,\om}^n
	\end{align*}
	for $m$--a.e. $\om\in\Om$, say $\om\in\Om_t$, and 
	\begin{align*}
	\cEP_k(s)=\lim_{n\to\infty}\frac{1}{n}\log\lm_{s,k,\om}^n
	\end{align*}
	for $m$--a.e. $\om\in\Om$, say $\om\in\Om_s$. Therefore, for all $\om\in\Om_s\cap\Om_t$, we have 
	\begin{align*}
		\cEP_k(s)-\cEP_k(t)\leq (t-s)\ol{\chi}(x)\leq\log\norm{T'}_\infty(t-s)
	\end{align*}
	for all $x\in\cJ_\om^{(k)}(T,s)$, and 
	\begin{align*}
		\cEP_k(s)-\cEP_k(t)\geq (t-s)\ul{\chi}(x)\geq \ul{\chi}(k)(t-s)
	\end{align*}
	for all $x\in\cJ_\om^{(k)}(T,s)$. The proof is now complete. 
\end{proof}
As an immediate consequence of this lemma we get the following. 
\begin{proposition}\label{p1ep2}
	If $0\leq s\leq t$, then
	\begin{align*}
		\cEP_\infty(s)-\cEP_\infty(t)\leq \log\norm{T'}_\infty(t-s)
	\end{align*}
	and
	\begin{align*}
		\cEP_\infty(s)-\cEP_k(t)\geq\ul{\chi}(k)(t-s)
	\end{align*}
	for every $k\in\NN$. 
\end{proposition}
By letting $k\to\infty$ in the second formula of this proposition and copying the first one, we get the following.

\begin{corollary}\label{c120180531}
If $0\leq s\leq t$, then
	\begin{align*}
		0\le \cEP_\infty(s)-\cEP_\infty(t)\leq \log\norm{T'}_\infty(t-s).
	\end{align*}
In particular, the function
	$$
	[0,+\infty)\ni u\longmapsto \cEP_\infty(u)\in\RR
	$$
is monotone decreasing and Lipschitz continuous with a Lipschitz constant equal to $\log\norm{T'}_\infty$.
\end{corollary}

Define
\begin{align*}
	D_T:=\set{t\geq 0:\cEP_\infty(t)\geq 0}.
\end{align*}
Of course 
\begin{align*}
	D_T\sub AA(T).
\end{align*}
Consequently, $\cE P(t)$ is well--defined for each $t\in D_T$ and 
\begin{align*}
	\cEP(t)=\cEP_\infty(t)
\end{align*} 
for all $t\in D_T$. Since $\cEP_\infty(0)=\cEP(0)>0$, it immediately follows from Corollary~\ref{c120180531} that there exists a unique number $b_T\in[0,+\infty]$ such that either 
\begin{align*}
	D_T=[0,b_T) \quad \text{ or }\quad D_T=[0,b_T].
\end{align*}

As an immediate consequence of this corollary, we get the following
\begin{corollary}\label{t1ep3}
If $b_T<+\infty$, then 
$$
	\cEP_\infty(b_T)=0.
$$
and (consequently)
	\begin{align*}
		D_T=[0,b_T].
	\end{align*}
\end{corollary}
Now, we shall prove more.
\begin{theorem}\label{t2ep4}
If $b_T<+\infty$, then there exists a number $b_T^+>b_T$ such that 
	\begin{align*}
		[0,b_T^+)\sub AA(T).
	\end{align*}
\end{theorem}
\begin{proof}
	The second inequality of Proposition \ref{p1ep2} tells us that 
	\begin{align*}
		\cEP_\infty(t)\geq\cEP_k(b_T)-\ul{\chi}(k)(t-b_T)
	\end{align*}
	for all $t\geq b_T$ and all $k\geq 1$. So, if $t\geq b_T$ and
	\begin{align*}
		\cEP_k(b_T)-\ul{\chi}(k)(t-b_T)>-\frac{1}{1+\gm_+}\chi_0t
	\end{align*}
for some $k\geq 1$, then $t\in AA(T)$. But this inequality is equivalent to 
	\begin{align*}
(1+\gm_+)(t-b_T)\ul{\chi}(k)<\chi_0t+(1+\gm_+)\cEP_k(b_T)
	\end{align*}
	and further, after dividing both sides by $t$, equivalent to
	\begin{align*}
(1+\gm_+)\lt(1-\frac{b_T}{t}\rt)\ul{\chi}(k)<\chi_0+(1+\gm_+)\frac{\cEP_k(b_T)}{t}.
	\end{align*}
Since $\ul{\chi}(k)\leq\log\norm{T'}_\infty$, it therefore follows that in order to know that $t\in AA(t)$, it suffices to have
	\begin{align*}
(1+\gm_+)\lt(1-\frac{b_T}{t}\rt)\log\norm{T'}_\infty
<\chi_0+(1+\gm_+)\frac{\cEP_k(b_T)}{t}   
	\end{align*}
for some $k\geq 1$. But by Corollary~\ref{t1ep3} and Corollary~\ref{c1tp5}, $\cEP_k(t)\le 0$ for every $k\geq 1$. Hence $\cEP_k(b_T)/t\ge 
\cEP_k(b_T)/b_T$. Therefore, it suffices to have
	\begin{align*}
(1+\gm_+)\lt(1-\frac{b_T}{t}\rt)\log\norm{T'}_\infty
<\chi_0+(1+\gm_+)\frac{\cEP_k(b_T)}{b_T}.   
	\end{align*}
for some $k\geq 1$. Since $\lim_{k\to\infty}\cE _k(b_T)=\cEP_\infty(b_T)\geq 0$ and since $\chi_0>0$, there exists $k\geq 1$ such that $-(1+\gm_+)\cE P_k(b_T)/b_T<\chi_0/2$. Thus, it suffices to have
	\begin{align}
(1+\gm_+)\lt(1-\frac{b_T}{t}\rt)\log\norm{T'}_\infty
<\chi_0/2.    \label{1ep4}
	\end{align}
But this is obviously true if $t>b_T$ is sufficiently close to $b_T$, say $t<b_T^+$. The proof of Theorem \ref{t2ep4} is complete.
\end{proof}

\section{Invariant Versions of Conformal Measures}\label{sec: invariant versions of conformal measures}
As the main result of this section, we shall prove the following.

\begin{theorem}\label{t2mis119}
	For every admissible parameter $t\geq 0$, i.e. belonging to AA(T), there exists a unique measure $\mu_t\in M^1_m(T)$ absolutely continuous with respect to $\nu_t$. In addition $\mu_t$ is equivalent to $\nu_t$ and ergodic.
\end{theorem}
\begin{proof}
Recall that
$$
p_2:\Om\times[0,1]\longmapsto [0,1]
$$
is the natural projection on the second coordinate, i.e.
$$
p_2(\om,x)=x. 
$$ 
Recall also that $\eta_*\in (0,\min\{1/2,1-(|\Dl_1|/2))$ is the number produced in Lemma~\ref{l1cmii3.2B}. Fix $t\geq 0$ and set $\nu:=\nu_t$. For every $n\geq 0$ set 
	\begin{align*}
		\nu^{(n)}:=\nu\circ T^{-n} \spand \mu^{(n)}:=\frac{1}{n}\sum_{j=0}^{n-1}\nu^{(j)}. 
	\end{align*}
Pointwise this means that for all $\om\in\Om$ and every Borel set $F_\om\sbt [0,1]$:
\begin{equation}\label{1020180606}
\nu_\om^{(n)}(F_\om)
=\nu_{\ta^{-n}(\om)}\(T_{\ta^{-n}(\om)}^{-n}(F_\om)\) 
\  \  \  {\rm and} \  \  \ 
\mu_\om^{(n)}(F_\om)
=\frac{1}{n}\sum_{j=0}^{n-1}\nu_\om^{(j)}(F_\om).
\end{equation}
Fix $\eta\in(0,\eta_*]$ arbitrary. Let $A\sub \cJ(T)$ be an arbitrary measurable set such that 
	\begin{align*}
		p_2(A)\sub[\eta,1-\eta].
	\end{align*}
It then follows from \eqref{1mis121} that 
$$
K_{\eta}^{-t}\nu_\om^{(j)}([\eta,1-\eta])\)\nu_{\om}(A_{\om})
\le \nu_\om^{(j)}(A_\om)
\le K_{\eta}^tQ_\cG^{-1}(t)\nu_{\om}(A_{\om})
$$
for all $j\geq 0$ and all $\om\in\Om$. Hence,
\begin{align}\label{2mis121}
K_{\eta}^{-t}\mu_\om^{(n)}([\eta,1-\eta])\nu_\om(A_\om)
\leq \mu_\om^{(n)}(A_\om)
\leq K_{\eta}^tQ_\cG^{-1}(t)\nu_{\om}(A_{\om})
\end{align}
for all $\om\in\Om$ and all integers $n\ge 1$.

	Let now $\mu=\mu_t$ be a weak limit of the sequence $(\mu^{(n)})_{n=1}^\infty$ in the narrow topology on $\Om\times[0,1]$, say
	\begin{align*}
		\mu=\lim_{k\to\infty}\mu^{(n_k)}
	\end{align*}
for some increasing sequence $(n_k)_{k=1}^\infty$. Obviously, $\mu\in M_m^1(T)$. By virtue of Portmanteau's Theorem for $G\sub p_2^{-1}([\eta,1-\eta])$, an arbitrary random open set in $\Om\times[0,1]$, we thus get from \eqref{2mis121} that 
	\begin{align*}
\mu(G)&\leq \liminf_{k\to\infty}\mu^{(n_k)}(G)
=\liminf_{k\to\infty}\int_{\Om}\mu_\om^{(n_k)}(G_\om)dm(\om) \\
&\leq K_{\eta}^tQ_\cG^{-1}(t)\int_{\Om}\nu_\om(G_\om)dm(\om)\\
&=K_{\eta}^tQ_\cG^{-1}(t)\nu(G).
	\end{align*}
It therefore immediately follows from Theorem \ref{t1mis119} (outer regularity) that
	\begin{align}\label{3mis123}
		\mu(A)\leq K_{\eta}^tQ_\cG^{-1}(t)\nu(A)
	\end{align}
for every measurable set $A\sub p_2^{-1}([\eta,1-\eta])$.

Now, fix $\ep>0$ arbitrary. It follows from Lemma~\ref{c520180606} and Portmanteau's Theorem that with $\eta_1\in (0,\eta]$ sufficiently small,
	\begin{align}\label{1mis131}
		\mu(\Om\times([\eta_1,1-\eta_1]^c))<\ep/2.
	\end{align} 
Take now an arbitrary measurable set $A\sub \cJ(T)$ such that 
$$
\nu(A)<\ep (2K_{\eta}^t)^{-1}Q_\cG.
$$
It then follows \eqref{3mis123} and \eqref{1mis131} that
$$
\begin{aligned}
\mu(A)
&=\mu(A\cap(\Om\times[\eta_1,1-\eta_1]^c))+\mu(A\cap(\Om\times[\eta_1,1-\eta_1])) \\
&<\ep/2+K_{\eta}^tQ_\cG^{-1}(t)\nu(A\cap(\Om\times[\eta_1,1-\eta_1]))\\
&\leq\ep/2+K_{\eta}^tQ_\cG^{-1}(t)\nu(A)<\ep/2+\ep/2\\
&=\ep.
\end{aligned}
$$
The proof of the absolute continuity of $\mu$ with respect to $\nu$ is thus complete. 

	Now we shall prove that the measures $\mu$ and $\nu$ are equivalent. So, let $A$ be an arbitrary measurable subset of $\cJ(T)$ with $\nu(A)>0$. Because of Theorem~\ref{t1mis119} we can assume without loss of generality that $A$ is a closed random set. By virtue of Remark~\ref{r120180526} there then exists $\eta_2\in(0,\eta_1]$ so small that 
	$$
	\nu(A\cap(\Om\times [\eta_2,1-\eta_2]))>0.
	$$
	 Then there exists $\al>0$ such that 
	\begin{align*}
		m(A_\Om(\al))>0 \  \text{ and } \  \nu(A(\al))>0,
	\end{align*} 
	where 
	\begin{align*}
		A_\Om(\al):=\set{\om\in\Om:\nu_\om(A_\om\cap[\eta_2,1-\eta_2])\geq \al} \text{ and } A(\al):=\union_{\om\in A_\Om(\al)}\set{\om}\times (A_\om\cap[\eta_2,1-\eta_2]).
	\end{align*}
	Decreasing $\eta_2\in(0,\eta_1]$ as needed, we will get in the same way as \eqref{1mis131}, that 
	\begin{align*}
		\mu\(\Om\times(\eta_2,1-\eta_2)^c\)<\frac{1}{4}m(A_\Om(\al)).
	\end{align*}
	Then, by applying Tchebyschev's Inequality, we get
	\begin{align*}
m\left(\set{\om\in\Om:\mu_\om((\eta_2,1-\eta_2)^c)
\geq \frac{1}{2}}\right)&\leq \frac{\int_{\Om}\mu_\om\left([\eta_2,1-\eta_2]^c\right)\, dm(\om)}{1/2}\\
&=2\mu\(\Om\times((\eta_2,1-\eta_2)^c)\)\\
&<\frac{1}{2}m(A_\Om(\al)).
	\end{align*}
	Therefore, 
	\begin{align*}
		m(A_\Om^*(\al))&\geq m(A_\Om(\al))-m\left(\set{\om\in\Om:\mu_\om((\eta_2,1-\eta_2))<\frac{1}{2}}\right)\\
		&\geq m(A_\Om(\al))-\frac{1}{2}m(A_\Om(\al))\\
		&=\frac{1}{2}m(A_\Om(\al)),
	\end{align*}
	where 
	\begin{align*}
		A_\Om^*(\al):=\set{\om\in A_\Om(\al): \mu_\om((\eta_2,1-\eta_2))\geq\frac{1}{2} }.
	\end{align*}
	Integrating now the left--hand side of \eqref{2mis121} over the set $A_\Om^*(\al)$ with $A$ replaced by $A(\al)$, we thus get 
	\begin{align}
		\int_{A_\Om^*(\al)}\mu_\om^{(n)}(A_\om(\al))\, dm(\om)
&\geq K_{\eta_2}^{-t}\int_{A_\Om^*(\al)}\mu_\om^{(n)}([\eta_2,1-\eta_2])\nu_\om(A_\om(\al))\, dm(\om) \\
&\geq K_{\eta_2}^{-t}\int_{A_\Om^*(\al)}\mu_\om^{(n)}((\eta_2,1-\eta_2))\nu_\om(A_\om(\al))\, dm(\om). 
	\end{align}
Taking the limit with $n_k\to\infty$, and observing that $A(\al)$, like $A$, is a closed random set while $\bigcup_{\om\in\Om}\{\om\}\times(\eta_2,1-\eta_2)$ is an open random set, we thus get that
	\begin{align*}
\mu(A)&\geq \mu(A(\al))
\geq \limsup_{k\to\infty}\int_{A_\Om^*(\al)}\mu_\om^{(n_k)}(A_\om(\al))\, dm(\om)\\
&\geq \limsup_{k\to\infty}K_{\eta_2}^{-t}\int_{A_\Om^*(\al)}\mu_\om^{(n_k)}((\eta_2,1-\eta_2))\nu_\om(A_\om(\al))\, dm(\om)\\
&\ge K_{\eta_2}^{-t}\int_{A_\Om^*(\al)}\mu_\om((\eta_2,1-\eta_2))\nu_\om(A_\om(\al))\, dm(\om)\\
&\geq (2K_{\eta_2}^t)^{-1}\int_{A_\Om^*(\al)}\nu_\om(A_\om(\al))\, dm(\om)
\geq (2K_{\eta_2}^t)^{-1}\al m(A_\Om^*(\al))\\
		&\geq (4K_{\eta_2}^t)^{-1}\al m(A_\Om(\al))>0.
	\end{align*}
Hence $\nu$ is absolutely continuous with respect to $\mu$, and in consequence, $\mu$ and $\nu$ are equivalent.

	Now we can easily finish the proof of Theorem \ref{t2mis119}. We show the ergodicity of $\mu_t$ first. Seeking contradiction suppose that $\mu_t$ is not ergodic. This would mean that there exists a measurable set $E\sub \cJ(T)$ such that 
	\begin{align}\label{1mis139}
		0<\mu(E)<1
	\end{align}
	and 
	\begin{align}\label{2mis139}
		T^{-1}(E)=E.
	\end{align}	
	Since $\ta:\Om\to\Om$ is ergodic, this implies that $\nu_\om(E_\om)>0$ for $m$--a.e. $\om\in\Om$. Since we already know that $\mu$ and $\nu$ are equivalent, \eqref{1mis139} yields 
	\begin{align*}
		0<\nu(E)<1.
	\end{align*}
	We can then consider the following two measures $\nu'$ and $\nu''$ on $\cJ(T)$ given by their fiberwise disintegrations as follows  
	\begin{align*}
		\nu_\om'(B):=\frac{\nu_\om(B\cap E_\om)}{\nu_\om(E_\om)} 
		\  \ \text{ and } \  \  
\nu_\om''(B):=\frac{\nu_\om(B\bs E_\om)}{\nu_\om(E\bs E_\om)}, \quad \om\in\Om.
	\end{align*}
Let us first verify that the measures/collections $\nu'=\{\nu'_\om\}_{\om\in\Om}$ and $\nu''=\{\nu'_\om\}_{\om\in\Om}$ are random measures supported on $\cJ(T)$ with respect to the base measure $m$. First, 
	\begin{align*}
		\nu_\om'([0,1])=\frac{\nu_\om([0,1]\cap E_\om)}{\nu_\om(E_\om)}=\frac{\nu_\om(E_\om)}{\nu_\om(E_\om)}=1
	\end{align*}
for $m$--a.e. $\om\in\Om$ and
	\begin{align*}
		\nu_\om'(\cJ_\om(T))=\frac{\nu_\om(\cJ_\om(T)\cap E_\om)}{\nu_\om(E_\om)}=\frac{\nu_\om(E_\om)}{\nu_\om(E_\om)}=1
	\end{align*}
	for $m$--a.e. $\om\in\Om$.
Therefore, for every measurable set $A\sub\Om$, we have 
	\begin{align*}
		\nu'\circ\pi_\Om^{-1}(A)=\nu'(A\times[0,1])=\int_A\nu'_\om([0,1])\, dm(\om)=\int_A\ind \, dm(\om)=m(A).
	\end{align*}		
So, $\mu'$ is a random measure with respect to the base measure $m$ supported on $\cJ(T)$. The same calculation shows it for $\mu''$. An equivalent form of \eqref{2mis139} is that 
	\begin{align*}
		T_\om^{-1}(E_{\ta(\om)})=E_\om
	\end{align*}
	for all $\om\in\Om$. Using this, we get for every $g\in L^+_\infty([0,1])$ that
	\begin{equation}\label{120180619}
	\begin{aligned}
		\tr_\om^*(\nu_{\ta(\om)}')g
		&=\nu_{\ta(\om)}'(\tr_\om g)
		=\frac{\int_{E_{\ta(\om)}}\tr_\om g \,d\nu_{\ta(\om)}}{\nu_{\ta(\om)}(E_{\ta(\om)})}
		=\frac{\int_{\cJ_{\ta(\om)}(T)}\ind_{E_{\ta(\om)}}\tr_\om g\, d\nu_{\ta(\om)}}{\nu_{\ta(\om)}(E_{\ta(\om)})}\\
		&=\frac{\int_{\cJ_{\ta(\om)}(T)}\tr_\om\(g\cdot( \ind_{E_{\ta(\om)}}\circ T_\om)\)\, d\nu_{\ta(\om)}}{\nu_{\ta(\om)}(E_{\ta(\om)})}
		=\frac{\tr_\om^*\nu_{\ta(\om)}\(g\cdot( \ind_{E_{\ta(\om)}}\circ T_\om)\)}{\nu_{\ta(\om)}(E_{\ta(\om)})}\\
		&=\frac{\tr_\om^*\nu_{\ta(\om)}(g\cdot\ind_{E_{\om}})}{\nu_{\ta(\om)}(E_{\ta(\om)})}
		=\lm_\om\frac{\nu_{\om}(g\cdot\ind_{E_{\om}})}{\nu_{\ta(\om)}(E_{\ta(\om)})}
		=\lm_\om\frac{\nu_{\om}({E_{\om}})}{\nu_{\ta(\om)}(E_{\ta(\om)})} \cdot\frac{\nu_{\om}(g\cdot\ind_{E_{\om}})}{\nu_{\om}(E_{\om})}\\
		&=\lm_\om'\nu_\om'(g),
	\end{aligned}
	\end{equation}
	where 
	\begin{align*}
		\lm_\om'=\lm_\om\frac{\nu_\om(E_\om)}{\nu_{\ta(\om)}(E_{\ta(\om)})}.
	\end{align*}
	So, see Remark~\ref{c1mis.cmi6.1} and the proof of Proposition~\ref{p1mis.cmi5}, 
	\begin{align*}
		\tr_\om^*\nu_{\ta(\om)}'=\lm_\om'\nu_\om'.
	\end{align*}
	Likewise,
	\begin{align*}
		\tr_\om^*\nu_{\ta(\om)}''=\lm_\om''\nu_\om'',
	\end{align*}	
	where 
	\begin{align*}
		\lm_\om''=\lm_\om\frac{\nu_\om(E_\om^c)}{\nu_{\ta(\om)}(E_{\ta(\om)}^c)}.
	\end{align*} 
It follows from the above formulas, in fact from the formula \eqref{120180619} alone, and from Proposition~\ref{p1mis.cmi5}, that the random measure $\nu'$ is $t$--Bconformal. Likewise $\nu''$. Hence, by virtue of Corollary~\ref{c520180405}, both measures $\nu'$ and $\nu''$ are $t$--conformal. Therefore, by Lemma~\ref{l1mis.cmii2} (b), these two measures, $\nu'$ and $\nu''$, are equivalent. However, on the other hand, by their very definition they are mutually singular (in fact for each $\om\in\Om$ the fiber measures $\nu_\om'$ and $\nu_\om''$ are mutually singular.) This contradiction finishes the proof of the ergodicity of the measure $\mu$. The uniqueness of $\mu$ is now immediate. Since $\mu$ and $\nu$ are equivalent, any measure $\eta\in M_m^1(T)$ absolutely continuous with respect to $\nu$ is also absolutely continuous with respect to $\mu$. Since $\mu$ is ergodic, $\eta=\mu$. The proof of Theorem~\ref{t2mis119} is complete. 
\end{proof}

	\section{Bowen's Formula}\label{sec: Bowen's Formula}
	This section is entirely devoted to proving the following theorem, which is a version of Bowen's Formula proven first by Rufus Bowen in \cite{bowen} in the context of quasi--Fuchsian groups.
	\begin{theorem}\label{t1mis.bf1}
		If $T:\cJ(T)\to\cJ(T)$ is a random critically finite map, then $b_T<+\infty$ (in fact $b_T\leq 1$) and 
		\begin{align*}
		\HD(\cJ_\om(T))=b_T
		\end{align*}
		for $m$--a.e. $\om\in\Om$.
	\end{theorem}
	\begin{proof}
We abbreviate $b_T$ by $b$. We will show first that 
		\begin{align*}
		\HD(\cJ_\om(T))\geq b
		\end{align*}
for $m$--a.e. $\om\in\Om$. This will also gives us that things: firstly that 
		\begin{align*}
		 b\leq 1
		\end{align*}
as $\HD(\cJ_\om(T))\leq 1$ for every $\om\in\Om$. 
		
Let $\nu$ and $\mu$ be the respective conformal and $T$--invariant measures corresponding to the parameter $b$ (see Theorem~\ref{t2mis119}). 
Fix $\om\in\Om$ and 
$$
z\in\cJ_\om(T)\bs\union_{n=0}^\infty T_\om^{-n}(\set{0,1}).
$$
Set 
$$
y:=(\om,z).
$$
Now, fix $\eta_*\in (0,1/2)$ coming from Lemma~\ref{l1cmii3.2}. Since $T_\rho(1)=T_\rho(0)=0$ for all $\rho\in\Om$ and since $0$ is a uniformly expanding fixed point for every $\rho\in\Om$, we conclude that there exists $\eta\in(0,\eta_*/2]$ such that the set 
		\begin{align*}
		N(y):=\set{j\geq 0: T_\om^j(z)\in(2\eta,1-2\eta)}
		\end{align*}
		is infinite. 
For every $r\in(0,\eta)$, let $k=k(y,r)$ be the largest integer greater than or equal to zero such that 
		\begin{align*}
		T_\om^k(z)\in(2\eta,1-2\eta)
		\end{align*}
		and 
		\begin{align}\label{msu5.2}
		B(z,r)\sub T_y^{-k}(B(T_\om^k(z),\eta)).
		\end{align}
Then, using Proposition~\ref{p120180630}, we get that
\begin{equation}\label{1120180823}
B(z,r)\sub B\lt(z,K_\eta\absval{(T_\om^k)'(z)}^{-1}\eta\rt).
\end{equation}
Hence 
\begin{equation}\label{5020180823}
r\le K_\eta\absval{(T_\om^k)'(z)}^{-1}\eta.
\end{equation}
By $b$--conformality of the random measure $\nu$ and by Proposition~\ref{p120180630} again, we also get that
\begin{align}
\nu_\om(B(z,r))
&\leq\nu_\om(T_y^{-k}(B(T_\om^k(z),\eta)))\lesssim\lm_\om^{-k}\absval{(T_ \om^k)'(z)}^{-b}\nu_{\ta^k(\om)}(B(T_\om^k(z),\eta))\nonumber\\
&\leq\lm_\om^{-k}\absval{(T_\om^k)'(z)}^{-b}.\label{msu5.3}
		\end{align}
Let $k^+=k^+(y,r)$ be the least integer greater than $k$ such that $T_\om^{k^+}(z)\in(2\eta,1-2\eta)$. Then 
		\begin{align*}
		B(z,r)\not\sub T_y^{-k^+}\(B(T_\om^{k^+}(z),\eta)\),
		\end{align*}
		Since, by Proposition~\ref{p120180630},
		\begin{align*}
		T_y^{-k^+}\(B(T_\om^{k^+}(z),\eta)
		\)\bus B\lt(z,K_\eta^{-1}\absval{(T_\om^{k^+})'(z)}^{-1}\eta\rt),
		\end{align*}
		we get that
		\begin{align}\label{2bf2}
		r\geq K_\eta^{-1}\absval{(T_\om^{k^+})'(z)}^{-1}\eta.
		\end{align}
Along with \eqref{5020180823}, this gives
\begin{equation}\label{1220180823}
\begin{aligned}
-\frac{\log(K_\eta\eta)}{k(y,r)}+\frac1{k(y,r)}\log\absval{(T_\om^{k(y,r)})'(z)}
&\le \frac{-\log r}{k(y,r)} \le  \\ 
&\le \frac{\log(K_\eta/\eta)}{k(y,r)}+\frac{k^+(y,r)}{k(y,r)}\cdot \frac1{k^+(y,r)}\log\absval{(T_\om^{k^+(y,r)})'(z)}.
\end{aligned}
\end{equation}
Inserting \eqref{2bf2} into \eqref{msu5.3}, we obtain
		\begin{align*}
		\nu_\om(B(z,r))\lesssim\lm_\om^{-k}r^b\left(\frac{\absval{(T_\om^{k^+})'(z)}}{\absval{(T_\om^{k})'(z)}}\right)^b,
		\end{align*}
		and further,
\begin{align}
\frac{\log(\nu_\om(B(z,r)))}{\log r}
\geq b-\frac{\log\lm_\om^k}{\log r}+\frac{1}{\log r}\left(\log\absval{(T_\om^{k^+})'(z)}-\log\absval{(T_\om^{k})'(z)}\right)-\frac{C}{\log r}\label{msu5.6}
\end{align}
with some constant $C>0$. Equivalently, 

\begin{align}\label{120180828}
\frac{\log(\nu_\om(B(z,r)))}{\log r}
\geq b-\frac{\frac{1}{k}\log\lm_\om^k}{\frac{1}{k}\log r}+\frac{k}{\log r}\left(\frac{k^+}{k}\frac{1}{k^+}\log\absval{(T_\om^{k^+})'(z)}-\frac{1}{k}\log\absval{(T_\om^{k})'(z)}\right)-\frac{C}{\log r},
\end{align}
where, we recall $k=k(y,r)$ (so, it does depend on $r$) and $k^+=k^+(y,r)$.
Thus
\begin{align}\label{1mis.bf2}
\liminf_{r\to 0}\frac{\log(\nu_\om(B(z,r)))}{\log r}\geq b+\frac{\lim_{k\to\infty}\frac{1}{k}\log\lm_\om^k}{\lim_{k\to\infty}\frac{1}{k}\log\absval{(T_\om^{k^+})'(z)}}+\frac{\lim_{k\to\infty}\frac{1}{k}\log\absval{(T_\om^k)'(z)}}{\lim_{k\to\infty}\frac{1}{k}\log\absval{(T_\om^{k^+})'(z)}}-1
\end{align}
provided all the above limits exist.
Now, the map $\ta:\Om\to\Om$ is ergodic with respect to the measure $m$ on $\Om$ by our hypotheses and the map $T:\cJ(T)\to\cJ(T)$ is ergodic with respect to the random measure $\mu$ by Theorem~\ref{t2mis119}. It therefore follows from Birkhoff's Ergodic Theorem and from Proposition~\ref{p1mis83}, that there exists a measurable set $\Om_*\sub\Om$, with $m(\Om_*)=1$, such that for every $\om\in\Om_*$ there exists a Borel set $\cJ_\om^*(T)\sub\cJ_\om(T)$, with $\nu_\om(\cJ_\om^*(T))=1$, such that for every $z\in\cJ_\om^*(T)$, we have 
\begin{align}\label{6020180828}
\lim_{\ell\to\infty}\frac{1}{\ell}\log\lm_\om^\ell
=\cEP(b_T)=0,
\end{align}
		\begin{align}\label{120180531}
		\lim_{\ell\to\infty}\frac{1}{\ell}\log\absval{(T_\om^\ell)'(z)}=\chi_\mu>0,
		\end{align}
and 
		\begin{align*}
		\lim_{\ell\to\infty}\frac{\ell^+(\om,z)}{\ell}=1,
\end{align*}
where $\ell^+(\om,z)$ is the least integer greater than $\ell$ such that $T_\om^{\ell^+(\om,z)}(z)\in(2\eta,1-2\eta)$. 
Inserting these three properties into \eqref{120180828}, letting $r>0$ in this formula go to $0$, and making use of \eqref{1220180823}, we get for all $\om\in\Om_*$ and $z\in\cJ_\om^*(T)$, that 
\begin{align*}
\liminf_{r\to 0}\frac{\log(\nu_\om(B(z,r)))}{\log r}\geq b.
\end{align*} 

So, in order to complete the entire proof we are left to show that 
		\begin{align*}
		\HD(\cJ_\om(T))\leq b
		\end{align*}
for $m$--a.e. $\om\in\Om$. We keep $\nu$, $\mu$, and  $y=(\om,z)$ the same as in the former part of the proof.
		
		For every $j\in N(y)$, let $r_j:=r_j(y)>0$ be the least radius such that 
		\begin{align*}
		T_y^{-j}\((\eta,1-\eta)\)\sub B(z,r_j).
		\end{align*}
		It follows from Corollary \ref{c1mis81} that $r_j<\eta$ for all $j\geq 0$ large enough, say $j\geq j_0$. We fix $j\geq j_0$. By applying Proposition~ \ref{p120180630} and invoking the very deffintion of $r_j$, we get 
		\begin{align}
		r_j
		\leq\diam\(T_y^{-j}\((\eta,1-\eta)\)\)
		\leq K_\eta\absval{(T_\om^j)'(z)}^{-1}.\label{1mis.bf3}
		\end{align}
Applying this proposition again along with Lemma~\ref{l1cmii3.2} and the definition of $\eta$, the conformality of the measure $\nu$ yields 
		\begin{align*}
\nu_\om(B(z,r_j))
&\geq \nu_\om(T_y^{-j}(\eta,1-\eta))
\geq K_\eta^{-b}\lm_\om^{-j}\absval{(T_\om^j)'(z)}^{-b}\nu_{\ta^j(\om)}((\eta,1-\eta))\\
&\geq Q_{\eta_*} K_\eta^{-b}\lm_\om^{-j}\absval{(T_\om^j)'(z)}^{-b}.
		\end{align*}
By inserting \eqref{1mis.bf3} into this inequality, we get 
		\begin{align*}	
		\nu_\om(B(z,r_j))\geq K_\eta^{-2b}Q_{\eta_*}\lm_\om^{-j}r_j^b.	
		\end{align*}
		By Corollary \ref{c1mis81}, $\lim_{N(y)\ni j\to\infty}r_j=0$.
Therefore, for such $\om$, using also  \eqref{1mis.bf3}, \eqref{6020180828}, and \eqref{120180531}, we get
		\begin{align*}
		\liminf_{r\to 0}\frac{\log\nu_\om(B(z,r))}{\log r}
		&\leq \liminf_{N(y)\ni j\to\infty}\frac{\log\nu_\om(B(z,r_j))}{\log r_j}
		\leq b+\limsup_{N(y)\ni j\to\infty}\frac{\log\lm_x^{-j}}{\log r_j}\\
		&=b+\limsup_{N(y)\ni j\to\infty}\frac{\frac{1}{j}\log\lm_x^{-j}}{\frac{1}{j}\log r_j}
=b.
		\end{align*}
		Therefore, 
		\begin{align*}
		\HD\(\cJ_\om(T)\bs\cup_{n\geq 0}T_\om^{-n}(\set{0,1})\)\leq b.
		\end{align*}
		So, since $\HD(\cup_{n\geq 0}T_\om^{-n}(\set{0,1}))=0$ (as this set is countable), we finally get that 
		\begin{align}\label{1mis.bf4}
		\HD(\cJ_\om(T))\leq b
		\end{align}
		for $m$--a.e. $\om\in\Om$. The proof is complete.
	\end{proof}
\begin{remark}\label{r120180828}
We have in fact proved in the first part of the above proof that for any measurable set 
$$
\bu_{\om\in\Om_0}\{\om\}\times E_\om\sbt J(T)
$$
there exists a measurable set $\Om_1\sbt\Om_0$ such that $m(\Om_0\sms\Om_1)=0$ and 
$$
\HD(E_\om)\ge b
$$
for all $\om\in\Om_1$.
\end{remark}
	
As the last result of this paper we shall prove the following theorem which shows that the sets $\cJ_\om(T)$, $\om\in\Om$, are all, up to a set of $m$--measure zero, true fractals unless
\begin{align*}
	I_*=\union_{\Dl\in\cG}\Dl=[0,1],
\end{align*}
in which case each set $\cJ_\om(T)$ is equal to $[0,1]$. 
\begin{theorem}\label{t1v4}
	If $T:\cJ(T)\to\cJ(T)$ is a random critically finite map, then 
\begin{align*}
		b_T=1 \  \text{ if an only if }  \  \union_{\Dl\in \cG}\Dl=[0,1],
	\end{align*}
and then $\cJ_\om(T)=[0,1]$ for all $\om\in\Om$. 
\end{theorem}
\begin{proof}
	Obviously, if $\union_{\Dl\in \cG}\Dl=[0,1]$ then $\cJ_\om(T)=[0,1]$ for all $\om\in\Om$, and hence $b_T=1$. The easier part of our theorem is thus established, and so we now suppose that 
	\begin{align*}
		\union_{\Dl\in \cG}\Dl\neq [0,1].
	\end{align*}
	Hence, there exists $H$, a connected component of $[0,1]\bs \union_{\Dl\in \cG}\Dl$, which is a non--degenerate open interval. Note that both endpoints of $H$ belong to $\cJ_\om(T)$ for all $\om\in\Om$ and that at least one of them, denoted by $\xi$, belongs to $(0,1)$. Let $D\sub [0,1]$ be a closed interval which is 
	\begin{enumerate}
	\,	
	
	\item disjoint from $H$,
	
	\,	\item $\xi\in D$,
	
	\,	\item $\diam(D)\leq\frac{1}{2}\min\set{\diam(H),\xi,1-\xi}$.
	\end{enumerate}
	As in the proof of the previous theorem (Bowen's Formula) denote the Bowen's parameter $b_T$ by $b$. Also, as in the previous proof denote by $\nu$ and $\mu$ the respective $b$--conformal measure for $T:\cJ(T)\to\cJ(T)$ and the equivalent $T$--invariant measure. Since $T_\om(\xi)=0$ or $T_\om^2(\xi)=0$ for all $\om\in\Om$, since $0\in\supp(\mu_\om)=\supp(\nu_\om)$ for all $\om\in\Om$, and since 
	\begin{align*}
		(\Om\times H)\cap\cJ(T)=\emptyset,
	\end{align*}
	it follows that 
	\begin{align*}
		\mu_\om(D),\, \nu_\om(D)>0
	\end{align*}
	for all $\om\in\Om$. In particular, 
	\begin{align*}
		\mu(\Om\times D),\, \nu(\Om\times D)>0.
	\end{align*}
	It therefore follows from Birkhoff's Ergodic Theorem and ergodicity of the measure $\mu$, that there exists a closed random measurable set $(E_\om)_{\om\in\Om_1}$, and an integer $N\geq 1$ such that 
	\begin{enumerate}[(a)]
	
	\,	\item \label{v(a)}$\Om_1\sub\Om$ is measurable and $m(\Om_1)>0$,

\,		\item \label{v(b)}$\mu_\om(E_\om)>0$ for all $\om\in\Om_1$, 
	
	\,	\item \label{v(c)}
	
	$$\#\set{1\leq j\leq n:T_\om^j(z)\in D}\geq \frac{\mu(E)}{2}
	$$
	for all $\om\in\Om_1$, all $z\in E_\om$, and all $n\geq N$ where 
		\begin{align*}
			E:=\union_{\om\in\Om_1}\set{\om}\times E_\om.
		\end{align*}
	\end{enumerate}
	By \ref{(M5b)}, we have that 

\, \begin{enumerate}[(d)]
		\item \label{v(d)} $\frac{1}{k}\log \absval{(T_\om^k)'(z)}\leq A$
	\end{enumerate}
\, \noindent for all integers $k\geq 1$, all $\om\in\Om$, and all $z\in\cJ_\om(T)$. 
	
	We now recall (see \cite{KR}, compare \cite{PR}) that a set $X$ contained in a metric space $(Y,\varrho)$ is called mean porous if there exists a constant $C\geq 1$ such that for every $x\in X$ there exist a sequence $(n_j)$ of positive integers and a sequence $(x_j)$ of points in $Y$ such that 
	\begin{align*}
		n_j\leq Cj
	\end{align*}
	for all integers $j$ large enough. The most significant property of mean porous sets is this. 
	\begin{theorem}[\cite{KR}, comp. \cite{PR}]\label{t1v6}
		If $X$ is a porous subset of a metric Euclidean space $\RR^d$, $d\geq 1$, then the upper box--counting dimension of $X$ is smaller than $d$. In particular, 
		\begin{align*}
			\HD(X)<d.
		\end{align*}	
	\end{theorem}
	We shall prove the following. 
	\begin{claim}\label{claim1v6}
		For every $\om\in\Om_1$, the set $E_\om$ is porous in $\RR$.
	\end{claim}
	\begin{proof}
		Fix $y$, the middle point of $H$, $\om\in\Om_1$, and $x\in E_\om$. Let $(q_j)_{j\geq 1}$ be the infinite increasing sequence of all integers $\ell\geq 1$ such that 
		\begin{align*}
			T_\om^\ell(x)\in D.
		\end{align*}
		For every $\geq 1$ set 
		\begin{align*}
			x_j:=T_{\om,x}^{-q_j}(y).
		\end{align*}
Then, by the Special Bounded Distortion Property, Proposition~\ref{p120180630},
\begin{align*}
		\absval{x-x_j}\leq K\absval{(T_\om^{q_j})'(x)}^{-1}\cdot\absval{T_\om^{q_j}(x)-y}\leq K_\eta\diam(H)\absval{(T_\om^{q_j})'(x)}^{-1}
	\end{align*}
	and 
	\begin{align*}
		\dist{x_j}{\cJ_\om(T)}\geq K_\eta^{-1}\frac{1}{2}\diam(H)\absval{(T_\om^{q_j})'(x)}^{-1},
	\end{align*}
where $\eta:=\dist{D}{\{0,1\}}$.
	Now, there exists a unique integer $n_j\geq 1$ such that for every $j\geq N$ large enough 
	\begin{align}\label{1v7}
		e^{-(n_j+1)}<\absval{(T_\om^{q_j})'(x)}^{-1}\leq e^{-n_j}.
	\end{align}
	We thus have 
	\begin{align*}
		\absval{x-x_j}\leq K_\eta\diam(H)e^{-n_j}
	\end{align*}
	and 
	\begin{align*}
		\dist{x_j}{\cJ_\om(T)}\geq (2K_\eta e)^{-1}\diam(H)e{-n_j}.
	\end{align*}
	Now, if $j\geq N$, then $q_j\geq N$, whence it follows from \ref{v(c)} that 
	\begin{align*}
		j\geq \frac{\mu(E)}{2}q_j
	\end{align*}
	and from \ref{v(d)} that 
	\begin{align*}
		q_j\geq \frac{1}{A}\log\absval{(T_\om^{q_j})'(x)}.
	\end{align*} 
	Combining these two inequalities along with \eqref{1v7}, we get 
	\begin{align*}
		n_j\leq\frac{2A}{\mu(E)}j.
	\end{align*}
	The proof of Claim \ref{claim1v6} is thus complete. 
	\end{proof}
	The conjunction of Claim \ref{claim1v6} and Theorem \ref{t1v6} tells us that 
	\begin{align*}
		\HD(E_\om)<1
	\end{align*}
	for all $\om\in\Om_1$. But, by Remark~\ref{r120180828} and item \ref{v(b)}, $\HD(E_\om)=b_T$ for $m$--a.e. $\om\in\Om_1$. Invoking also \ref{v(a)}, we therefore conclude that $b_T<1$. The proof of Theorem \ref{t1v4} is complete. 
\end{proof}

\end{document}